\documentclass[11pt]{article}

\usepackage{amsfonts,cite}
\usepackage{amsmath,amsthm,amscd,amssymb,mathrsfs,setspace,bm}
\usepackage{latexsym,epsf,epsfig}
\usepackage{color}
\usepackage[hmargin=3cm,vmargin=3.5cm]{geometry}
\usepackage{graphicx}
\usepackage{hyperref}
\usepackage{graphicx}
\usepackage{tikz}

\hypersetup{backref=true}

\newcommand{\R}{\mathbb{R}}

\newcommand{\bn}{\mathbf n}

\newcommand{\grad}{\nabla}

\newcommand{\bxi}{\boldsymbol{\xi}}

\theoremstyle{plain}
\newtheorem{theorem}{Theorem}[section]
\newtheorem{lemma}[theorem]{Lemma}
\newtheorem{proposition}[theorem]{Proposition}
\newtheorem{condition}[theorem]{Condition}
\newtheorem{corollary}[theorem]{Corollary}

\newtheorem{definition}{Definition}
\theoremstyle{remark}
\newtheorem{remark}{Remark}[section]
\numberwithin{equation}{section} \numberwithin{theorem}{section}
\numberwithin{remark}{section} 

\DeclareMathOperator{\diverg}{div}

\DeclareMathOperator{\N}{\mathbb{N}}

\newcommand{\abs}[1]{|{#1}|}

\newcommand{\norm}[1]{\|#1\|}

\newcommand{\eps}{\epsilon}

\newcommand{\meantmp}[2]{#1\langle{#2}#1\rangle}
\newcommand{\mean}[1]{\meantmp{}{#1}}

\providecommand{\skptmp}[3]{{\ensuremath{#1\langle {#2}, {#3} #1\rangle}}}
\providecommand{\skp}[2]{\skptmp{}{#1}{#2}}

\usepackage{xcolor}

\usepackage[normalem]{ulem}
\usepackage{verbatim}

\newcommand{\flow}{\mathbf{b}}
\newcommand{\nrml}{\mathbf{n}}
\newcommand{\pdb}{\partial_{\flow}}
\newcommand{\pdn}{\partial_{\nrml}}

\newcommand{\inttime}{\int_0^T}
\newcommand{\intwave}{\int_{\Omega_W}}
\newcommand{\intbndwave}{\int_{\partial{\Omega_W}}}
\newcommand{\intgamm}{\int_{\Gamma}}

\newcommand{\exx}{\mathbf{e}}

\newcommand{\ArcCos}{\mathrm{arc\hspace{0.3mm}cos}}
\newcommand{\typo}[2]{#2}

\title{Time-Periodic Solutions for Hyperbolic-Parabolic Systems}
\author{\small  Stanislav Mosn\'y$^\dag$ \hskip 1cm Boris Muha\footnote{University of Zagreb, Faculty of Science, Department of Mathematics, Croatia;~ {\em borism@math.hr}}  \hskip 1cm Sebastian Schwarzacher \footnote{University of Uppsala, Sweden; {\em stanislav.mosny@math.uu.se} and {\em sebastian.schwarzacher@math.uu.se}}  \footnote{Charles University, Prague, Czechia; {\em schwarz@karlin.mff.cuni.cz}} \hskip 1cm Justin T. Webster\footnote{University of Maryland, Baltimore County, 1000 Hilltop Circle, Baltimore, MD, 21250;~ {\em websterj@umbc.edu}} }

\begin{document}

\maketitle
\begin{abstract}
\noindent Time-periodic weak solutions for a coupled hyperbolic-parabolic system are obtained. A linear heat and wave equation are considered on two respective $d$-dimensional spatial domains that share a common $(d-1)$-dimensional interface, $\Gamma$. The system is only partially damped, leading to an indeterminate case for existing theory. We construct periodic solutions by obtaining novel a priori estimates for the coupled system, reconstructing the total energy via the interface $\Gamma$. As a byproduct, geometric constraints manifest on the wave domain which are reminiscent of classical boundary control conditions for wave stabilizability. We note a  ``loss" of regularity between the forcing and  solution which is  greater than that associated with the heat-wave Cauchy problem. However, we consider a broader class of spatial domains and mitigate this regularity loss by trading time and space differentiations, a feature unique to the periodic setting. This seems to be the first constructive result addressing existence and uniqueness of periodic solutions in the heat-wave context, where no dissipation is present in the wave interior. Our results speak to the open problem of the (non-)emergence of resonance in complex systems, and are readily generalizable to related systems and certain nonlinear cases. 

\noindent 
\noindent 
\vskip.2cm

\vskip.2cm
\noindent
{\em 2010 AMS}: 35B10, 35Q93, 74F10, 35M10
\vskip.2cm
\noindent {\em Acknowledged Support}: Authors 1 and 3 were supported by the VR-Grant 2022-03862 by the Swedish Research Council. Author 3 was supported by the ERC-CZ grant LL2105 of the Czech ministry of education, youth and sports and by the Charles University Research Centre program No. UNCE/24/SCI/005. Author 3 is a member of the Ne\v{c}as Centre for Mathematical Modeling. Author 2 was supported by the Croatian Science Foundation project IP-2022-10-2962. Author 4 was partially supported by NSF-DMS 2307538 and UMBC's Strategic Award for Research Transitions (START).
\end{abstract}

\maketitle

\section{Introduction}
In the context of evolutionary partial differential equations (PDEs), the study of time-periodic solutions has a long and robust history \cite{Vejvoda81}. We can mention some modern treatments (and references therein) for hyperbolic \cite{celik,coron,gaz1} and parabolic \cite{lunardi,casanova,fw2} systems in both linear and nonlinear cases. Motivated by many applications in physics and mechanics, the general aim for this class of problems is to determine under what circumstances a time-periodic input (body force, boundary condition, coefficient, etc.) will yield a time-periodic response (PDE solution). Moreover, one wishes to determine under which conditions the possibility of {\em resonance}---an unbounded solution corresponding to a bounded input---can be excluded \cite{gaz1,gaz2,g2}. These problems are {\em highly challenging}, even in the linear case, and {\em presently open} \cite{galdi}. 
In this treatment, our goal is to resolve the issue of existence and uniqueness of periodic solutions for a certain coupled system which is {\em partially damped} \cite{ZZ1,galdi}. To the knowledge of the authors, this is a heretofore open problem, and not covered by extant literature. Moreover, we resolve this difficult open problem in a constructive fashion, explicitly providing PDE estimates that lead to periodic well-posedness in the finite-energy setting. 

We consider a bounded  domain $\Omega$ with $\overline{\Omega} = \overline{\Omega_W} \cup \overline{\Omega_H} \subset \mathbb R^d$, such that the two subdomains {$\Omega_H$ and $\Omega_W$ are each connected, and share a  lower-dimensional interface $\Gamma = \partial \Omega_W \cap \partial \Omega_H$}.  
The central variables are $u$, which satisfies a parabolic (heat-type) partial differential equation (PDE), and $w$, satisfying an undamped hyperbolic (wave-type) PDE. The body forces, $f, g$ are assumed to have point-wise almost-everywhere values in time on $[0,T]$, with $T>0$ arbitrary but fixed; if $f$ or $g$ ({and/or their derivatives}) are continuous in time on $\mathbb R$, then they are assumed to be $T$-periodic. As we are interested in periodic solutions for the system, we will require that both $u$ and $w$---and their derivatives, when defined---are time-periodic over the interval $[0,T]$.
\begin{align}\label{sys1}
	\partial_t^2{w} + L w &= g\text{ on }[0,T]\times \Omega_W,\quad w(0)=w(T), ~w_t(0)=w_t(T);
	\\
	\partial_t{u} + A u &= f\text{ on }[0,T]\times \Omega_H,\quad u(0)=u(T).
\end{align}

The operators $L$ and $A$ above, representing the hyperbolic and parabolic components, respectively, of the dynamics, can be taken to be quite general. However, we focus on the simplest case for the sake of exposition, noting that the problem is open even at this simplified, linear level. Thus we take $L=A=-\Delta$ for the remainder of the analysis here. We remark below in Section \ref{generalize} how  $L$ and $A$ can be generalized. For coupling across the interface $\Gamma$, we choose conditions motivated by fluid-structure interactions, as in \cite{ZZ1,ZZ2,trig1}. With our choices up to this point, we obtain certain Dirichlet and Neumann interface conditions:
\begin{align}\label{coupling}
	\partial_t{w}=u ~\text{ and } ~ \partial_{\bn} w=\partial_{\bn} u ~\text{ on }~[0,T]\times \Gamma,
\end{align}
where $\bn$ is chosen to be the outer normal of the domain $\Omega_W$. Away from $\Gamma$, we will supplement the system with homogeneous Dirichlet-type boundary conditions. 

\subsection{Periodic Problems and Associated Challenges}
Periodic problems have a different structure than traditional Cauchy problems for evolutionary systems. Indeed, Cauchy problems typically seek Hadamard well-posedness through the propagation of initial  (and perhaps other) data. In contrast, {\em the initial state is an unknown in periodic problems}. One is given only a forcing  $F_T$ associated to a period $T>0$, and the goal is to construct a solution $y(\cdot)$ which is somehow continuous in time and satisfies $y(0)=y(T)$ (in some functional sense). Several interesting outcomes are a priori possible:

\begin{itemize}
\setlength\itemsep{.03cm}
\item For each $F_T$ in a given functional class, there exists a unique solution $y$ of period $T$ (and hence the initial condition $y(0)$ is uniquely determined). 
\item For some $F_T$ in a given functional class, there exists multiple solutions of period $T$ (and hence the initial conditions associated to periodic solutions are not uniquely determined).
\item For some $F_T$ in a given functional class, there does not exist a periodic solution for any initial condition $y(0)$ coming from the designated class.
\end{itemize}
\begin{remark}We acknowledge but do not explicitly address the interesting question of when a forcing $F_T$ could generate a periodic solution of (minimal) period $ T^* \neq T$. \end{remark}

The last bullet point above is particularly important, since it encompasses the circumstance of {\em resonance}. For a resonant system, there exists particular frequencies and associated {\em bounded} forcing functions for which the response (solution) grows unboundedly\footnote{the words ``bounded" and ``unbounded" here can refer to pointwise information, or some prescribed norm}, and hence is not periodic. A simple example can be seen via an undamped harmonic oscillator with a periodic forcing taken at the natural frequency (or other integer multiples \cite{isaac})---see the nice modern treatment in \cite{playing}. As the seminal reference, \cite{galdi}, points out, determining conditions for which resonance occurs in complex systems, or, relatedly, when periodic solutions exist and are unique, is interesting across a broad range of applications and also an {\em  exceptionally challenging} mathematical problem. 

Another interesting aspect contrasting Cauchy and periodic problems is that of regularity. In periodic problems, without prescribed initial conditions from which to propagate, regularity of the solution is at issue. In particular, the regularity gap between data and the solution may be greater than what is expected from typical hyperbolic or parabolic Cauchy problems. As we shall see, to obtain finite energy periodic solutions, we will need smoother-in-time forcing data. We  provide examples in Section \ref{exams} which explicitly illustrate this gap.

Two cases for which an abundant periodic literature exists are  hyperbolic evolutions (no dissipation present) and  parabolic evolutions (fundamentally dissipative). For parabolic systems, the periodic theory is well-established, including for nonlinear dynamics \cite{casanova,bostan,lunardi}. The essential idea is to utilize dissipation in the system to produce a priori estimates from which fixed point methods can be utilized---see the discussion in the next section.  This includes some dynamics for which nonlinear fluids and fluid-structure interactions are involved, but we are quick to note that such works utilize ad hoc methods or have strong notions of dissipation present in the dynamics \cite{srd,giusy,sebastian,fw2,fn1,fn2,g1,gs}. For hyperbolic systems \cite{brezis,breziscoron,cesari,coron,rab,fw1,fw3,gaz1}, existence and uniqueness of periodic solutions is a demanding endeavor for at least two reasons: first, kernels associated to wave operators are, in some sense, large, and can give rise to non-uniqueness or non-existence (resonance in the simple harmonic oscillator); secondly, in a conservative hyperbolic system, the lack of dissipation provides no useful a priori estimates. 

\subsection{Energy Inequality and Approach to Construction}
To illustrate the previous discussion, suppose now that $E(t)$ is a positive energy functional on a space $X$ corresponding to an evolutionary system on $t\in [0,T]$, with forcing provided by $F \in L^2(0,T; X)$. Suppose $D(t)$ represents a dissipation functional, which tracks a positive, dissipated quantity as time advances, with the system obeying the  energy inequality
$$E(t) +\int_0^t D(\tau)d\tau  \lesssim E(0)+||F||_{L^2(0,T;X)}^2.$$ 
We can make two observations, assuming that the solution is $T$-periodic. 
First, since the initial state $y(0)$ is an unknown in this context, the associated value $E(0)$ is also unknown; hence, the inequality carries no useable information---for a construction---at time $t$. Secondly, if we let $t=T$ in order to invoke the periodicity of the solution, we observe that $E(T)=E(0)$ (since $y(T)=y(0)$, and associated derivatives agree). Then the energy inequality reduces to
$$\int_0^T D (\tau)d\tau \lesssim ||F||_{L^2(0,T;X)}^2.$$ So long as the dissipation in the problem provides enough information to ``reconstruct" appropriate norms of the solution (perhaps after repeated differentiations of the system), then a priori estimates can be obtained that permit a solution construction. This is akin to the notion of {\em observability} in classical control theory for PDE systems \cite{ZZ1,redbook}.

The  previous two situations suggest that the role of dissipation is central to  existence and uniqueness of periodic solutions. Indeed, this is a key point in  the key reference \cite{galdi}. There,  connections between asymptotic stability and periodicity are expounded. Informally:  a linear system, forced by an $L^2$ space-time function, demonstrates existence of a unique periodic solution whenever the underlying solution semigroup exhibits uniform stability (associated to the unforced dynamics). This says that the dissipation in the problem---which provides the uniform stability---is ``strong enough" to reconstruct the state space energy. If no stability of the semigroup is observed (or some weaker notion of stability) then a unique periodic solution may not necessarily follow. We elaborate upon other results from \cite{galdi} in Section \ref{results}.

This brings us to periodic solutions for the coupled hyperbolic-parabolic system presented above. In this partially-damped framework, it is not immediately clear which aspects of the dynamics are most important from the point of view of periodicity. Two natural questions emerge: {\em Is the dissipation in the parabolic component sufficient to ``control" the hyperbolic component?} And, secondly,   as resonance (non-existence) is possible for the hyperbolic component, {\em can resonant solutions persist even when the hyperbolic dynamics is coupled with a parabolic dynamics?} This paper answers those questions directly, in a manner which has not previously appeared in the literature: we construct a periodic solution explicitly from novel PDE estimates that demonstrate the ``propagation" of dissipation from the parabolic component to the hyperbolic component via the interface. This demonstrates a remarkable connection between periodic problems and the boundary control of hyperbolic dynamics.

	\subsection{Precise Mathematical Question} \label{athand}
We consider the bounded, { connected}, Lipschitz domains $\Omega_W$, $\Omega_H$, $\Omega\subset \R^d$ such that $\Omega_W\cap\Omega_H=\emptyset$ and  $\overline{\Omega} =\overline{\Omega_W \cup \Omega_H}$. The domains admit the interface $\Gamma = \partial \Omega_W \cap \partial \Omega_H$,  with $\Gamma$ assumed to be of class $\mathcal C^1$.  We will denote the outer (inactive) boundary of $\Omega =\Omega_H\cup \Omega_W$ by $\Gamma_0=\partial (\Omega_H \cup \Omega_W) \setminus \Gamma \equiv \Gamma_{H}\cup \Gamma_{W}$, where the heat and wave variables satisfy homogeneous Dirichlet boundary conditions.

\begin{figure}[htbp]
	\centering
	\begin{minipage}[b]{0.4\textwidth}
		\centering
		\resizebox{!}{4cm}{%
			\tikzset{every picture/.style={line width=0.75pt}} 

\begin{tikzpicture}[x=0.75pt,y=0.75pt,yscale=-1,xscale=1]

\draw  [fill={rgb, 255:red, 0; green, 0; blue, 0 }  ,fill opacity=0.1 ] (391.2,187.3) .. controls (388.2,198.3) and (419.2,199.3) .. (463.2,185.3) .. controls (507.2,171.3) and (539.2,196.3) .. (561.2,197.3) .. controls (583.2,198.3) and (592.2,198.3) .. (596.2,187.3) .. controls (600.2,176.3) and (539.2,159.3) .. (480.2,161.3) .. controls (421.2,163.3) and (394.2,176.3) .. (391.2,187.3) -- cycle ;
\draw    (391.2,187.3) .. controls (385.69,207.23) and (403.69,212.43) .. (398.89,229.63) .. controls (394.09,246.83) and (386.35,266.04) .. (427.69,272.43) .. controls (469.04,278.82) and (495.54,246.82) .. (515.79,245.57) .. controls (536.04,244.32) and (549.54,250.57) .. (566.04,256.82) .. controls (582.54,263.07) and (598.54,246.82) .. (592.49,231.63) .. controls (586.45,216.44) and (605.33,194.55) .. (595.08,183.3) ;
\draw    (393.13,192.5) .. controls (375.88,179.25) and (369.63,139.1) .. (399.88,112.6) .. controls (430.13,86.1) and (453.62,99.86) .. (479.22,113.46) .. controls (504.82,127.06) and (523.37,102.86) .. (552.02,102.96) .. controls (580.66,103.06) and (618.33,132.55) .. (595.08,190.05) ;
\draw  [fill={rgb, 255:red, 255; green, 255; blue, 255 }  ,fill opacity=0 ][dash pattern={on 4.5pt off 4.5pt}] (379.6,148.5) .. controls (377.01,159.15) and (413.26,155.9) .. (430.51,152.65) .. controls (447.76,149.4) and (472.51,145.4) .. (490.76,145.15) .. controls (509.01,144.9) and (543.26,149.9) .. (555.51,153.15) .. controls (567.76,156.4) and (601.8,158.91) .. (601.41,147.41) .. controls (601.01,135.9) and (574.26,132.9) .. (558.51,132.65) .. controls (542.76,132.4) and (508.76,136.4) .. (495.01,136.15) .. controls (481.26,135.9) and (436.76,132.15) .. (421.51,133.15) .. controls (406.26,134.15) and (382.19,137.85) .. (379.6,148.5) -- cycle ;
\draw  [color={rgb, 255:red, 0; green, 0; blue, 0 }  ,draw opacity=1 ][fill={rgb, 255:red, 255; green, 255; blue, 255 }  ,fill opacity=0 ][dash pattern={on 4.5pt off 4.5pt}] (396.29,239.32) .. controls (393.29,249.82) and (424.29,264.82) .. (469.29,244.07) .. controls (514.29,223.32) and (535.43,229.53) .. (551.18,232.03) .. controls (566.93,234.53) and (594.95,247.19) .. (592.49,231.63) .. controls (590.04,216.07) and (510.54,208.32) .. (494.54,209.07) .. controls (478.54,209.82) and (448.04,212.47) .. (432.04,217.57) .. controls (416.03,222.68) and (399.29,228.82) .. (396.29,239.32) -- cycle ;

\draw (544.23,179.2) node    {$\Gamma $};
\draw (451.11,229.89) node    {$\Omega _{2}$};
\draw (434.16,120.89) node    {$\Omega _{1}$};

\end{tikzpicture}
		}
	\end{minipage}
	\hspace{1cm}
	\begin{minipage}[b]{0.4\textwidth}
		\centering
		\resizebox{!}{4cm}{%
			\tikzset{every picture/.style={line width=0.75pt}} 

\begin{tikzpicture}[x=0.75pt,y=0.75pt,yscale=-1,xscale=1]

\draw   (217.62,198.38) .. controls (169.96,198.38) and (131.31,192.42) .. (131.24,185.05) -- (304,185.05) .. controls (303.93,192.42) and (265.28,198.38) .. (217.62,198.38) -- cycle ;
\draw  [dash pattern={on 4.5pt off 4.5pt}] (217.62,171.68) .. controls (265.33,171.68) and (304,177.66) .. (304,185.03) .. controls (304,185.04) and (304,185.05) .. (304,185.05) -- (131.24,185.05) .. controls (131.24,185.05) and (131.24,185.04) .. (131.24,185.03) .. controls (131.24,177.66) and (169.91,171.68) .. (217.62,171.68) -- cycle ;
\draw [color={rgb, 255:red, 255; green, 255; blue, 255 }  ,draw opacity=1 ][line width=1.5]    (131.24,185.03) -- (304,185.05) ;
\draw   (131.24,185.05) .. controls (131.24,137.35) and (169.91,98.67) .. (217.62,98.67) .. controls (265.33,98.67) and (304,137.35) .. (304,185.05) .. controls (304,232.76) and (265.33,271.43) .. (217.62,271.43) .. controls (169.91,271.43) and (131.24,232.76) .. (131.24,185.05) -- cycle ;
\draw  [color={rgb, 255:red, 0; green, 0; blue, 0 }  ,draw opacity=0 ][fill={rgb, 255:red, 0; green, 0; blue, 0 }  ,fill opacity=0.1 ] (131.24,185.03) .. controls (131.24,177.66) and (169.91,171.68) .. (217.62,171.68) .. controls (265.33,171.68) and (304,177.66) .. (304,185.03) .. controls (304,192.41) and (265.33,198.38) .. (217.62,198.38) .. controls (169.91,198.38) and (131.24,192.41) .. (131.24,185.03) -- cycle ;

\draw (217.62,185.05) node    {$\Gamma $};
\draw (220.16,135.29) node    {$\Omega _{1}$};
\draw (218.31,231.89) node    {$\Omega _{2}$};

\end{tikzpicture}
		}
	\end{minipage}
	\caption{Example configurations.}
\end{figure}
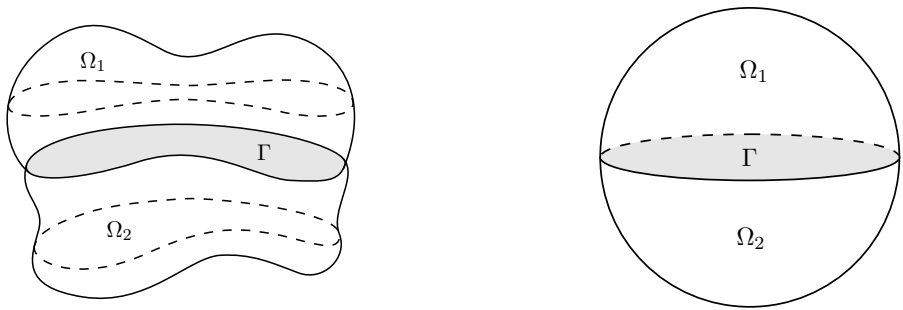

	The precise PDE system we are interested in is: 
	\begin{align}
\label{sys2}
	 w_{tt}-\Delta w&=g\;{\rm in}\; [0,T]\times\Omega_W,
	\\
	u_t-\Delta u&=f \;{\rm in}\; [0,T]\times\Omega_H,
	\\
	w_t&=u\;{\rm on}\; [0,T]\times \Gamma,
	\label{bcd}
	\\
	\partial_{\bf n} u&= \partial_{\bn}w\;{\rm on}\; [0,T]\times \Gamma,\\
	w & \equiv 0\;{\rm on}\; [0,T]\times\Gamma_{W},\\
		u & \equiv 0\;{\rm on}\; [0,T]\times\Gamma_{H} \label{sys2ef}\\
		u(0)=&~u(T);~~w(0)=w(T),~~w_t(0)=w_t(T). \label{sys2e}
\end{align}
We seek conditions on $\Omega_W$, $\Omega_H$, $\Gamma$, and $f,g$ to produce a unique  solution $(u,w)$ to \eqref{sys2}--\eqref{sys2e}.

\subsection{Function Spaces and Weak Solutions}
We will utilize standard Sobolev Hilbert spaces on the constitutive domains, $\Omega_W$ and $\Omega_H$, of the form $H^s(\Omega_{\square})$ with $\square = W,H$, and $H^s_0(\Omega_{\square})$ denoting the closure of the test functions $\mathscr D(\Omega_{\square}) \equiv C_0^{\infty}(\Omega_{\square})$ in the standard $H^s$ norm \cite{grisvard}. Utilizing the standard trace operator (or its extension, when appropriately defined), we can define subspaces of $H^s(\Omega_{\square})$ where boundary values vanish on subsets of $\partial \Omega_{\square}$. We will then denote by $H^k_{\Gamma_{\square}}(\Omega_{\square})$ for $k \in \mathbb N$ the subspace of $H^k(\Omega_{\square})$ which is the collection of $H^k$ functions such that all traces vanish on $\Gamma_{\square}$ up to order $k-1$. Denoting the standard Dirichlet trace operator $\gamma_0: H^1(\Omega_{W}) \to H^{1/2}(\partial \Omega_W)$, we can then write
\begin{equation}
H^1_{\Gamma_W}(\Omega_W) = \{ u \in H^1(\Omega_W)~:~r_{\Gamma_W}\gamma_0[u] = 0\},
\end{equation}
where $r_{\Gamma_W}$ is the restriction in the sense of distributions to $\Gamma_W \subset \partial \Omega_W$. The same definition holds, mutatis mutandis for $H^1_{\Gamma_H}(\Omega_H)$. We can define the associated dual spaces, which will be denoted as $H^1_{\Gamma_W}(\Omega_W)'$, for instance. 
Finally, for the statement of our main results, we will be working with forcing functions which are smoother (in time) than $L^2(0,T;L^2(\Omega_{\square}))$. { Consider $X$ to be a Banach space and $f\in H^{k}(0,T;X)$. For that we define $\tilde{f}:\mathbb{R}\to X$ by $\tilde{f}(jT+t)=f(t)$, for $j\in \mathbb{Z}$ and $t\in [0,T]$. We then say that $f\in H^{k}(0,T;X)$ is time periodic, in notation
$
f\in H^k_\sharp(0,T;X)$, if $\tilde{f}\in H^{k}_{\textrm{loc}}(\mathbb{R};X).
$ In this case, $\tilde f$ has viable time traces (in particular at $t=0$ and $t=T$) and is periodic up to derivatives of order $k-1$. Further we write that $f\in C^k_{\sharp} (0,T;X)$, if $\tilde{f}\in C^{k}(\mathbb{R};X)$. We take the norm in $H^k_{\sharp}(0,T;X)$ to coincide with that of $H^k(0,T;X)$, understanding that the integrability properties are thus local on $\mathbb R$.}

{ We address the coupling condition \eqref{bcd} through the trace operator, in a weak sense. Namely, we can interpret \eqref{bcd}, written strongly as $w_t=u$ on $\Gamma$, through the identity\footnote{The weak boundary condition \eqref{bcdw} implies  that $w_t$ will possess the same boundary regularity as $u$ via duality.}
\begin{align}
\label{bcdw}
\int_0^T\int_{\Gamma} \gamma_0(u)\xi\, dx=-\int_0^T\int_{\Gamma} \gamma_0(w)\partial_t\xi\, dx\text{ for all }\xi\in C^1_\sharp(0,T;L^2(\Gamma)).
\end{align}}

We can now define a {\em finite energy weak solution} to the periodic problem in \eqref{sys2}--\eqref{sys2e}.
\begin{definition} \label{weaksol} 
Assume that the forcing functions satisfy $f \in L^2(0,T;H^1_{\Gamma_H}(\Omega_H)')$ and $g \in L^2(0,T;L^2(\Omega_W))$. We say that $$(w,u) \in  \left [L^2(0,T;H^1_{\Gamma_W}(\Omega_W))\cap H^1_{\sharp}(0,T;L^2(\Omega_W))\right] \times L^2(0,T;H^1_{\Gamma_H}(\Omega_H))$$ is a finite energy weak solution to \eqref{sys2}--\eqref{sys2e} { if: 
\begin{itemize} 
\item for every pair of smooth, $T$-periodic test functions $(\psi,\phi)$ with $\phi|_{\Gamma}=\psi|_{\Gamma}$,~ $\phi|_{\Gamma_H}=\psi|_{\Gamma_W}=0$, the following identity holds:
\begin{align}\label{weakdef}
\int_0^T\int_{\Omega_H}[-u \phi_t +\nabla u \cdot \nabla \phi] + \int_0^T\int_{\Omega_W}[-w_t\psi_t + \nabla w \cdot \nabla \psi] = \int_0^T\langle f,\phi \rangle +\int_0^T\int_{\Omega_W} g \psi,
\end{align} where  $\langle \cdot, \cdot \rangle$ denotes the  pairing between $H^1_{\Gamma_H}(\Omega_H)$ and its dual;
\item $w(0)=w(T)$;
\item  and \eqref{bcdw} is satisfied. 
\end{itemize}}

\end{definition}
In the above definition, it is sufficient to take test functions 
\[(\psi,\phi) \in H^1_{\sharp}\left(0,T;   H^1_{\Gamma_W}(\Omega_W) \times H^1_{\Gamma_H}(\Omega_H)\right)\]
 which obey the added constraint that $r_{\Gamma} \gamma_0[\phi]=r_{\Gamma} \gamma_0[\psi]$, where, as above, $r_{\Gamma}$ is the restriction (in distribution) to $\Gamma$.  
\begin{remark} Typically \cite{temam,pata}, weak solutions to the unforced wave equation are considered in the functional class $W^{1,\infty}(0,T;L^2(\Omega_W))\cap L^{\infty}(0,T;H^1_{\Gamma_W}(\Omega_W))$ and obey the energy identity
\begin{equation}
||\nabla w(t)||^2_{\mathbf L^2(\Omega_W)}+||w_t(t)||^2_{L^2(\Omega)} = ||\nabla w(0)||^2_{\mathbf L^2(\Omega_W)}+||w_t(0)||^2_{L^2(\Omega)},\text{ for all }t \in [0,T].
\end{equation} Here, owing to the loss of the baseline energy identity for the wave equation, we initially work in $L^2$-based function spaces that carry less information. A posteriori, we will reconstruct the above $L^{\infty}$ estimate with more temporal regularity assumed on the forcing. 
\end{remark}

\subsection{Uniqueness of Solutions}\label{sec:unique}
	By linearity,  uniqueness of weak solutions follows from showing that the zero solution is unique when the forcing data $f \equiv 0$ and $g\equiv 0$ $a.e.$ space-time. Our uniqueness theorem is inspired by the approach in the control literature \cite{ZZ1,ZZ2,trig1} and appeals to the celebrated Holmgren Theorem \cite{holmgren} in the context of the wave equation---see  \cite{eller,eller2,littman}. We note that uniqueness is non-trivial here, since we are not working with smooth or semigroup solutions, as in \cite{ZZ1,galdi,dutch,trig1}; moreover, as we do not prescribe initial data for the periodic problem, uniqueness consists of reconstructing the initial (Cauchy) data for the problem as part of the  solution. Further comments are provided in Remark \ref{uniqueness}. 
	
	\begin{theorem}\label{th:uniqueness}  Suppose that domains $\Omega_W$, $\Omega_H$ are  connected, Lipschitz domains in $\mathbb R^d$ such that $\Omega_W\cap\Omega_H=\emptyset$ and  $\overline{\Omega} =\overline{\Omega_W \cup \Omega_H}$. { Moreover}, assume that the interface $\Gamma=\overline{\Omega_W}\cap\overline{\Omega_H}$ is of class $\mathcal C^1$. If  $f \equiv 0$ and $g \equiv 0$ in \eqref{weakdef}, then the  solution, as in Definition \ref{weaksol}, is the trivial solution. Thence, by superposition, all weak solutions in the sense of Definition \ref{weaksol} are unique.
	\end{theorem}

	\begin{proof} We begin with a weak solution in hand, $$(w,u) \in  \left [L^2(0,T;H^1_{\Gamma_W}(\Omega_W))\cap H^1_{\sharp}(0,T;L^2(\Omega_W))\right] \times L^2(0,T;H^1_{\Gamma_H}(\Omega_H)).$$ We assume that this solution satisfies \eqref{weakdef}, and hence $(u,w)$ satisfies 
	\begin{align} \label{dis1}
	 w_{tt}-\Delta w&=0\;~{\rm in}~\; \mathscr D'\big((0,T)\times\Omega_W\big),
	\\ \label{dis2}
	u_t-\Delta u&=0\;~{\rm in}~\; \mathscr D'\big((0,T)\times\Omega_H\big).
	\end{align}
 To obtain an energy estimate for weak solutions we must regularize the solution, as \typo{it not}{it is not} smooth enough to be utilized as a test function. As such, we extend $(u,w,w_t)$, in the sense of $L^2_t$, periodically (\typo{a.e.}{a.e. in} $t$) to obtain (with the same labels) $$(u,w)\in L^2_{\text{loc}}(\mathbb R;H^1_{\Gamma_H}(\Omega_H))\times \left [L^2_{\text{loc}}(\mathbb R;H^1_{\Gamma_W}(\Omega_W))\cap H^1_{\text{loc}}(\mathbb R;L^2(\Omega_W))\right].$$ We then perform the standard time-mollification on $\mathbb R$ to obtain temporally smooth functions $(u_{\rho},w_{\rho})$ satisfying \eqref{dis1}--\eqref{dis2}, as well as \eqref{weakdef} with the same spatial regularity of $(u,w)$.  Moreover, the mollified solution $(w_{\rho},u_{\rho})$ is $T$-periodic, by construction. Thus $(w_{\rho},u_{\rho})$ is a viable test function in \eqref{weakdef}:
	$$\int_0^T\big[-(w_{t},w_{\rho,tt})_{\Omega_W}+(\nabla w,\nabla w_{\rho,t})_{\Omega_W}-(u,u_{\rho,t})_{\Omega_H}+(\nabla u, \nabla u_{\rho})_{\Omega_H}\big] dt=0.$$
	Directly from the weak formulation---with $f\equiv 0$ and $g\equiv 0$---we infer that $w_{tt} \in L^2(0,T;[H^1_{\Gamma_W}(\Omega_W)]')$ and $u_{t} \in L^2(0,T;[H^1_{\Gamma_H}(\Omega_H)]')$.  Integrating by parts in time and utilizing the time-periodicity of each of the functions, we obtain 
			$$\int_0^T\big[\langle w_{tt},w_{\rho,t}\rangle_{\Omega_W}+(\nabla w,\nabla w_{\rho,t})_{\Omega_W}+\langle u_t,u_{\rho}\rangle_{\Omega_H}+(\nabla u, \nabla u_{\rho})_{\Omega_H}\big] dt=0,$$
where $\langle \cdot, \cdot \rangle$ represents the appropriate duality pairing. 
	From this,  in passing to the limit $\rho\searrow 0$ as in \cite[Lemma II.4.1]{temam}, we obtain that $u \in C([0,T];L^2(\Omega_H))$, $w \in C([0,T];H^1_{\Gamma_W}(\Omega_W))$, and $w_t \in C([0,T]; L^2(\Omega_W))$, as well as the identity
\begin{align} ||u(T) ||_{L^2(\Omega_H)}^2+& ||\nabla u||^2_{L^2(0,T;\mathbf L^2(\Omega_H))}+||\nabla w(T)||_{L^2(0,T;\mathbf L^2(\Omega_W))}^2+||w_t(T)||^2_{L^2(0,T;L^2(\Omega_W))} \nonumber  \\
 =&~  ||u(0) ||_{L^2(\Omega_H)}^2+||\nabla w(0)||_{L^2(0,T;\mathbf L^2(\Omega_W))}^2+||w_t(0)||^2_{L^2(0,T;L^2(\Omega_W))} \end{align}
Invoking time periodicity of $u,~w$ and $w_t$, and cancelling quantities at time $t=0$ and $t=T$, we obtain
	\begin{equation}
	||u||^2_{L^2(0,T;H_{\Gamma_H}^1(\Omega_H))} \lesssim 0.
	\end{equation}
	Thence we infer that the heat component $u \equiv 0 \in L^2(0,T;H^1_{\Gamma_H}(\Omega_H))$. 
	
	Now, returning to the dynamics, we  zero out the quantities in $u$. From the weak formulation, we  read off the following overdetermined boundary value problem:
	\begin{align}
	 w_{tt}-\Delta w&=0\;{\rm in}\; [0,T]\times\Omega_W,
	\\
	w_t&=0\;{\rm on}\; [0,T]\times \Gamma,
	\\
	\partial_{\mathbf n} w&= 0\;{\rm on}\; [0,T]\times \Gamma,\\
	w & \equiv 0\;{\rm on}\; [0,T]\times\Gamma_{W}.
\end{align}
Above, the interior equation in $\Omega_W$ is interpreted distributionally,  with $L^2(\Omega_W)$ data, from which the Neumann-type boundary condition may be {  considered as part of the distributional Cauchy data} \cite[{  Definition} 2.3.3.]{eller}. Moreover, by periodicity, the above equation holds on $(0,T')$ for any $T'\in\R$.
We may then differentiate the problem in time, and consider the system in the variable $\hat w =w_t$, yielding
	\begin{align}
	 \hat w_{tt}-\Delta \hat w&=0\;{\rm in}\; [0,T']\times\Omega_W,
	\end{align}
	with conditions
	\begin{align}
	\hat w=0\;{\rm on}\; [0,T']\times \Gamma,~~
	\partial_{\mathbf n} \hat w= 0\;{\rm on}\; [0,T']\times \Gamma,~~
	\hat w  \equiv 0\;{\rm on}\; [0,T']\times\Gamma_{W},
\end{align}
for any $T'\in\R$. The above system is overdetermined on $\Gamma$ in the variable $\hat w$. At this point, we invoke the classical Holmgren theorem, valid for distributional solutions to the above wave equation. {  (Consider the statement in  
\cite[Theorem 1, p.365]{littman}, requiring only that the interface is $\mathcal C^1$ and $\Omega_W$ is Lipschitz.)} Upon application, we obtain that there exists time $\tau_{\Omega_W}\in\R$ such the solution $(\hat w, \hat w_t)$ is identically zero on $(\tau_{\Omega_W},\infty)$. 

Then, per the assumed periodicity, we obtain that the solution $(\hat w,\hat w_t)\equiv 0$ on $\R$.
	
	 Finally, since $w_t \equiv 0$ in space time, $w$ is constant (a.e.) on $[0,T] \times \Omega_W$. We  then invoke the  Dirichlet condition on $\Gamma_W$ to infer that the solution $w$ to the wave equation above is trivial. This completes the proof of uniqueness of solutions in the sense of Definition \ref{weaksol}. 
	\end{proof}

\begin{remark}
	A slightly weaker uniqueness result follows indirectly from results obtained in \cite{galdi,ZZ2}. First, we note that the heat-wave Cauchy problem  is strongly (asymptotically) stable \cite[Theorem 4]{ZZ2}. According to \cite{galdi},  such systems have  unique $T$-periodic solutions for certain $T$-periodic right-hand sides \cite[Theorem 3.4.]{galdi}. However, since the system is linear, it is enough to observe uniqueness for any one right-hand side. We emphasize that \cite[Theorem 3.4.]{galdi} refers to {\em semigroup} or {\em mild solutions}, and therefore this uniqueness result is less general then Theorem \ref{th:uniqueness}. However, we note that the Holmgren's Theorem is also used in \cite{ZZ2} to obtain critical estimates needed for strong stability, and therefore both approaches ultimately rely upon the same theorem. 
\end{remark}

	\begin{remark}\label{uniqueness} As with all problems involving the wave equation, a separate argument is needed to show (a) uniqueness of weak solutions and (b) a posteriori time regularity in the sense of $C([0,T];H)$. This issue stems from the fact that the wave  velocity is not regular enough to be used as a test function in the  weak formulation. On the other hand, the constructed solution typically obeys energy inequalities and possesses more smoothness. In our results, we appeal to the uniqueness argument above, and so our constructed solution---with its regularity properties---will immediately be the unique solution. See \cite{pata} for a recent discussion of these issues and conditions under which a secondary uniqueness argument may not be needed to infer weak well-posedness with augmented temporal regularity. \end{remark}

	\section{Statement of Results, Context, and Illustrative Examples} \label{results}
	In this section we  state our main results, which place geometric conditions on the domains $\Omega_W$ and $\Gamma_W$ as well as the time regularity of the forcing functions $f$ and $g$,  to produce existence of periodic solutions, with associated well-posedness estimates. Uniqueness has already been shown under more general hypotheses in Section \ref{uniqueness}.

\subsection{Main Results}
For both of  theorems below, we recall the {standing hypotheses}: $\Omega_W,~ \Omega_H$ are connected, Lipschitz,
 and the interface $\Gamma = \partial\Omega_W\cap \partial\Omega_H$ is of class $\mathcal C^1$. Below, we focus on periodic solutions in the sense of {\em finite energy weak solution}, as in Definition \ref{weaksol}. Our theorems are stated in terms of geometric conditions, which are the main focus of the technical Section \ref{geomsec}. 
The geometric hypotheses on $(\Omega_W,\Gamma_W)$ mentioned below are: (i) a generalization of the classical Geometric Optics Condition,  shown below as Condition \ref{dotcond}, and (ii) a technical condition for regularity of wave solutions, denominated the E-property in Definition \ref{EProperty}.

Our first two main results, Theorems \ref{th:main1} and \ref{th:main2}, provide conditions for existence of {\em finite energy weak solutions}, as in Definition \ref{weaksol}. The two ancillary results, Theorems \ref{th:main1-general} and \ref{th:main2-general}, invoke the notion of {\em very weak solution} and provide analogous results
\begin{theorem}\label{th:main1}
If the tuple $(\Omega_W,\Gamma_W)$ satisfies the Generalized Optics Condition~\ref{def:GOC} and $\Omega_W$ has the E-property (Definition \ref{EProperty}), 
then for every  $f\in H_{\sharp}^3(0,T;H^1_{\Gamma_H}(\Omega_H)')$ and every  $g\in H_{\sharp}^6(0,T;L^2(\Omega_W))$
there exists a unique finite energy periodic weak solution $(w,u)$ in the sense of Definition \ref{weaksol}. 

The solution satisfies $u \in C_{\sharp}([0,T];L^2(\Omega_H)) \cap L^2(0,T;H^1_{\Gamma_H}(\Omega_H))$ and \\  $w \in C_{\sharp}([0,T];H_{\Gamma_W}^1(\Omega_W)) \cap C_{\sharp}^1([0,T];L^2(\Omega_W))$, and the following estimate
\begin{align} 
\norm{u}_{H^3(0,T;H_{\Gamma_H}^1(\Omega_H))}+&\norm{w}_{L^\infty(0,T;H_{\Gamma_W}^1(\Omega_W))}+\norm{\partial_t w}_{L^\infty(0,T;L^2(\Omega_W))}  \\ \lesssim&~ \norm{f}_{H^3(0,T;H^1_{\Gamma_H}(\Omega_H)')}+\norm{g}_{H^6(0,T;L^2(\Omega_W))}.
\end{align}
\end{theorem}

In a second case, we can consider a {\em graph} framework, which permits a broader class of domains than in Theorem \ref{th:main1}, but requires additional regularity of the forcing data. 
\begin{theorem}\label{th:main2}
If the tuple $(\Omega_W, \Gamma_W)$ satisfies the Graph Optics Condition~\ref{def:GGC}, and $\Omega_W$ has the E-property,  
then for every $f\in H_{\sharp}^4(0,T;H^1_{\Gamma_H}(\Omega_H)')$ and every  $g\in H_{\sharp}^8(0,T;L^2(\Omega_W))$, 
 there exists a unique finite energy periodic weak solution $(w,u)$ in the sense of Definition \ref{weaksol}. 
 
 The solution satisfies $u \in C_{\sharp}([0,T];L^2(\Omega_H)) \cap L^2(0,T;H^1_{\Gamma_H}(\Omega_H))$ and \\ $w \in C_{\sharp}([0,T];H_{\Gamma_W}^1(\Omega_W)) \cap C_{\sharp}^1([0,T];L^2(\Omega_W))$, and the following estimate
\begin{align}
\norm{u}_{H^4(0,T;H_{\Gamma_H}^1(\Omega_H))}+&\norm{w}_{L^\infty(0,T;H_{\Gamma_W}^1(\Omega_W))}+\norm{\partial_t w}_{L^\infty(0,T;L^2(\Omega_W))} \\ \lesssim & ~ \norm{f}_{H^4(0,T;H^1_{\Gamma_H}(\Omega_H)')}+\norm{g}_{H^8(0,T;L^2(\Omega_W))}.
\end{align}
\end{theorem}

\begin{remark}\label{regcauchy} We note that the regularity assumptions on the  forces $f$ and $g$ are   higher than that of  the corresponding  results for the standard Cauchy problem---see, for instance, \cite{ZZ2,pazy}. However, as we will show by the examples in the next sub-section, this regularity gap is  inherent to periodic problems involving the wave equation.
\end{remark}

If the above geometric assumptions are satisfied, the concept of solutions can be extended further to accommodate {essentially} any right-hand side, so long as it resides in some negative (in time) Sobolev space. 
This conclusion arises from a straightforward observation: since the system is linear and no initial conditions are imposed, the time derivative of a periodic solution is itself a periodic solution to the same problem, but with time-differentiated right-hand sides. Here, we state the main result concerning so called {\em very weak solutions}; the precise definition and detailed discussions of this notion of solution is provided in Appendix B.
\begin{theorem}\label{th:main1-general}
	If the tuple $(\Omega_W,\Gamma_W)$ satisfies the Generalized Optics Condition~\ref{def:GOC}, and $\Omega_W$ has the E-property Definition \ref{EProperty}, then
 for every  $f\in H_{\sharp}^{3-k}(0,T;H^1_{\Gamma_H}(\Omega_H)')$, every $g\in H_{\sharp}^{6-k}(0,T;L^2(\Omega_W))$, and any $k\in\N$, 
	there exists a very weak $T$-periodic solution $(w,u)$ in the sense of Definition \ref{VeryWeakDef}. The solution satisfies the following estimate
	\begin{align} 
		\norm{u}_{H^{3-k}(0,T;H_{\Gamma_H}^1(\Omega_H))}+&\norm{w}_{W^{-k,\infty}(0,T;H_{\Gamma_W}^1(\Omega_W))}+\norm{\partial_t w}_{W^{-k,\infty}(0,T;L^2(\Omega_W))}  \\ \lesssim&~ \norm{f}_{H^{3-k}(0,T;H^1_{\Gamma_H}(\Omega_H)')}+\norm{g}_{H^{6-k}(0,T;L^2(\Omega_W))}.\nonumber
	\end{align}
\end{theorem}

\begin{theorem}\label{th:main2-general}
	If the tuple $(\Omega_W, \Gamma_W)$ satisfies the Graph Optics Condition~\ref{def:GGC}, and $\Omega_W$ has E-property Definition \ref{EProperty},  
	 then for every  $f\in H_{\sharp}^{4-k}(0,T;H^1_{\Gamma_H}(\Omega_H)')$, every  $g\in H_{\sharp}^{8-k}(0,T;L^2(\Omega_W))$, and any $k\in\N$, 
	there exists a very weak solution $(w,u)$ in the sense of Definition \ref{verweak}. The solution satisfies the following estimate
	\begin{align} 
		\norm{u}_{H^{4-k}(0,T;H_{\Gamma_H}^1(\Omega_H))}+&\norm{w}_{W^{-k,\infty}(0,T;H_{\Gamma_W}^1(\Omega_W))}+\norm{\partial_t w}_{W^{-k,\infty}(0,T;L^2(\Omega_W))}  \\ \lesssim&~ \norm{f}_{H^{4-k}(0,T;H^1_{\Gamma_H}(\Omega_H)')}+\norm{g}_{H^{8-k}(0,T;L^2(\Omega_W))}.\nonumber
	\end{align}
\end{theorem}

The necessity of considering very weak solutions is demonstrated through the examples in the next section.

\subsection{Examples}\label{exams}
Here we provide two explicit solutions to problem \eqref{sys2}--\eqref{sys2e} and discuss their regularity. These examples are not finite energy weak solutions, but satisfy the periodic problem in the weaker sense of Appendix B. 
 Let us consider a simple geometric setting in $2-d$, taking $\Omega_W=(0,\pi)\times (0,1)$, $\Omega_H=(0,\pi)\times (-1,0)$ and take $T=2\pi$. Then the interface $\Gamma$ is given by $\Gamma=(0,\pi)\times\{0\}$. 

Now, let $\phi\in C^2_0(0,1)$, and define 
\begin{equation}\label{wee}
w_n(t,x,y)=\sin(nt)\sin(nx)\phi(y),\quad n\in\N.
\end{equation}
By direct calculation, we see that each $w_n$ satisfies the following periodic boundary value problem for the wave equation:

\begin{align*}
\partial_{tt}w_n-\Delta w_n=g_n\;{\rm in}\; [0,T]\times\Omega_W,
\\
w_n=0\;{\rm on}\; [0,T]\times\partial\Omega_W,
\\
\partial_{\mathbf n} w_n=-\partial_y w_n=0 \;{\rm on}\; [0,T]\times\Gamma,
\\
u_n(0)=u_n(T),\; \partial_t u_n(0)=\partial_t u_n(T) \; {\rm in}\; \Omega_W,
\end{align*}
where \begin{equation} \label{gee} g_n(\typo{x,y;t}{t, x,y})=-\sin(nt)\sin(nx)\phi''(y).\end{equation} Therefore $(0,w_n)$ is a {\em smooth solution} to problem \eqref{sys2}--\eqref{sys2e}, with data $(f,g)=(0,g_n)$.

In this framework, we will produce two instructive examples. 
\subsubsection{Example 1}
\label{ex:1}
Let $G_1=\sum_{n=1}^{\infty}\alpha_n g_n,$ where  the $\alpha_n$ are chosen so that $$\sum_{n=1}^{\infty}\alpha_n^2<\infty, ~~\text{ and }~~ \sum_{n=1}^{\infty}n^s\alpha_n^2=\infty, ~~s>0.$$ Therefore, by Plancherel's Theorem, $G_1\in L^2 ((0,T); L^2(\Omega_W))$, but $G_1\not\in H^s(0,T;L^2(\Omega_W))$ and $G_1\notin H^s(0,\pi;L^2((0,T)\times (0,1))$. We can then define an associated solution
$$
\hat{w}=\sum_{n=1}^{\infty}\alpha_n w_n.
$$
By construction, we observe that $\hat{w}\in L^2((0,T); L^2(\Omega_W))$ (and not better in time). This function, $\hat w(t,\mathbf x)$ is also a time-periodic  solution in the sense of Appendix B (i.e., weaker than Definition \ref{weaksol}).  Namely,  for every $\psi\in C^2([0,T];H^2_{\Gamma_W}(\Omega_W))$, such that $\psi(0)=\psi(T)$ and $\partial_t\psi (0)=\partial_t\psi (T)$, it holds that
\begin{align}\label{verweak}
\int_0^T\int_{\Omega_W}\hat{w}\partial_{tt}\psi-\int_0^T\int_{\Omega_W}\hat{w}\Delta\psi=\int_0^T\int_{\Omega_W}G_1\psi.
\end{align}
The above notion of solution  connects to  our discussion in Appendix B, and is delineated there a ``very weak" solution.
Notice that $\hat{w}$ {\em is not a finite energy solution}, since $\partial_x \hat{w}, ~\partial_t \hat{w}\not \in L^2((0,T); L^2(\Omega_W))$. Moreover,
the time traces $\hat{w}(0)$ and $\hat{w}(T)$ are not well-defined with point-wise values in $L^2(\Omega_W)$, but these time traces are well-defined in the sense of $[H^{2}(\Omega_W)\cap H^1_{\Gamma_W}(\Omega_W)]'$.
Thus, considering $\hat w$ as a time-periodic weak solution in the sense of \eqref{verweak}, we see that the data-to-solution mapping is of regularity
$$G_1 \in L^2(0,T;L^2(\Omega_W)) \mapsto \hat w \in L^2(0,T;L^2(\Omega_W)).$$
This constitutes an apparent loss, given that $G_1$, taken as forcing for a Cauchy problem, would produce a finite energy weak solution $w \in C([0,T]; H_{\Gamma_W}^1(\Omega_W)) \cap C^1([0,T]; L^2(\Omega_W))$ \cite{temam,pata}. 

\subsubsection{Example 2}
\label{ex:2}
Now, consider the $g_n$ as above in \eqref{gee} and let $G_2=\sum_{n=1}^{\infty}\beta_n g_n,$ where $\beta_n$ are chosen so that $$\sum_{n=1}^{\infty}n^2\typo{\alpha_n}{\beta_n}^2<\infty~ \text{ and }~ \sum_{n=1}^{\infty}n^s\typo{\alpha_n}{\beta_n}^2=\infty,$$ for all $s>2$. We then define 
$$
\bar w=\sum_{n=1}^{\infty}\typo{\alpha_n}{\beta_n} w_n.
$$
Now $(0,\bar w)$ is a $T$-periodic {\em finite energy weak solution}, in the sense of Definition \ref{weaksol}, with the given right-hand side $(0,G_2)$. This produces the data-to-solution mapping of the form
$$G_2 \in H^1((0,T) \times \Omega_W) \mapsto \hat w \in H^1((0,T) \times \Omega_W).$$

\subsubsection{Conclusion From Examples}
The aforementioned examples suggest that the minimal time regularity assumption on the right-hand side, under which we expect the existence of a finite energy  $T$-periodic weak solution,  is $(f,g)\in L^2(0,T;H^1_{\Gamma_H}(\Omega_H)')\times H^1(0,T;L^2(\Omega_W))$. In the special case of $f\equiv 0 \equiv u$, as  considered in the examples above, we find by~\eqref{est2} in Proposition~\ref{prop:WaveEst} that the apparent loss is sharp; see Remark~\ref{rem:sharp} for a precise explanation.

In Theorem \ref{th:main1} {however}, we assume several additional degrees of Sobolev regularity, as compared to the minimal regularity suggested by these examples. In Section \ref{sec:construction} we attempt to further identify the source of this additional regularity loss associated with lifting non-homogeneous ``heat" data through the boundary into interior dynamics.

\subsection{Main Contributions and Relation to Literature}\label{generalize}

We first note that there seem to be {\em very few existence and uniqueness results for periodic solutions of partially-damped systems}, such as the heat-wave system considered here. Moreover, in focusing on periodic solutions for fluid-structure interactions (linear and nonlinear), available results \cite{casanova,srd,sebastian,sebastian2} seem to require some degree of dissipation or parabolicity in the structural (wave) component. Thus  {\em the main results here are novel in providing a construction of periodic solutions without wave dissipation}. 

Returning to the reference \cite{galdi}, we can provide some comparison with the results at hand. In that work, solutions are considered corresponding to the  underlying (linear) semigroup for the partially-damped dynamics; specifically, solutions are of the form $C([0,T];H)$ (where $H$ is the energy space), with $L^1(0,T;H)$ forcing, and the solution represented through the variation of parameters formula \cite{pazy,pata,BAL}. As stated above, uniform stability of the semigroup would ensure that there exists a unique periodic solution to the system for any such force in $L^1(0,T;H)$---in particular, there would be no resonant forces in this class. However, for the partially damped heat-wave system here, it is well-known from the results of \cite{ZZ1,ZZ2,trig1} that the underlying heat-wave semigroup corresponding to \eqref{sys2}--\eqref{sys2e} is not uniformly stable. Indeed, the underlying dynamics are only {\em polynomially stable}\footnote{also known as {\em almost-uniformly stable}} (see \cite{ZZ1,ZZ2,trig1,dutch}), but only if certain geometric conditions on $\Omega_W$ and $\Gamma_W$ are satisfied---the so called {\em Geometric Optics Condition}. We will provide more details on geometric assumptions in Section \ref{geomsec}. In the general case the dynamics are only {\em strongly stable}.  In such a case, the results of \cite{galdi} provide a weaker result, ensuring only the existence of a dense subset of $L^1(0,T;H)$ for which there exists a unique $T$-periodic solution. Hence, for the system \eqref{sys2}--\eqref{sys2e} presented here, \cite{galdi} ensures that for a dense subset of $f \in L^1(0,T; L^2(\Omega_H))$ and $g \in L^1(0,T; H^1_{\Gamma_W}(\Omega_W) \times L^2(\Omega_W))$, existence and uniqueness of periodic solutions is ensured. {\em This does not preclude the possibility of some resonant forces in these functional classes.} This density result depends abstractly on the {\em Mean-Ergodic Theorem}, and does not provide a construction or elaboration of this dense set. 

Our point of view here then contrasts and supplements the results in \cite{galdi}. We explicitly construct a dense subset of $L^1(0,T; L^2(\Omega_H)\times H^1_{\Gamma_W}(\Omega_W) \times L^2(\Omega_W))$ for which a unique, finite energy (weak) periodic solution is available. Our estimates show precisely the loss of regularity in the mapping $(g,f) \mapsto (w,u)$ associated to the periodic problem. Moreover, we tie the solution construction to the geometry of $\Omega_W$ and elucidate the connection between classical boundary control estimates and a priori estimates for the heat-wave system. Since the dissipation must provide control for the interior states, we must ``propagate" the dissipation from the parabolic to the hyperbolic component  via the  interface $\Gamma$. In doing so, we inherit the need for geometric restrictions on the tuple $(\Omega_W,\Gamma_W)$. As such, we suspect an insurmountable loss will be present in any construction (the examples in Section \ref{exams} speak to this point)---see also the proof of non-compactness of the heat-wave resolvent operator, as discussed in \cite{ZZ2,trig1}.

We also point out that the results in \cite{galdi}  focus on semigroup (mild or generalized) solutions, rather than weak solutions. In this sense, our definition of weak solutions in Definition \ref{weaksol} (as well as {\em very weak solutions} in Appendix B) is also a contribution of this manuscript. We note that issues such as a priori estimates and uniqueness for weak solutions are non-trivial for systems involving the wave equation \cite{evans,BAL}, and even for the wave equation itself \cite{temam}---see the recent \cite{pata}. Hence our uniqueness result, which exploits time periodicity as well as the classical Holmgren theorem \cite{eller}, is non-trivial. 

Lastly, a feature of our analysis is that the geometric control conditions on domains are more general than those from classical control---see \cite{ZZ1,ZZ2}. Our analysis admits domains with certain types of degeneracies and singularities into the class of the domains for which one  can obtain a priori estimates---see Section \ref{geomsec}. This occurs because of the flexibility afforded to us in periodic problems, when integrating by parts in time: there is a certain ability to ``swap" time and space derivatives in the periodic setting \typo{on that}{that} is not available for Cauchy problems. The price to pay, however, is that the regularity gap between the data and the finite energy weak solution grows, as in Remark \ref{regcauchy}. 

	Further generalizations of our results are possible. First, for the heat dynamics, one can accommodate general elliptic operators and certain nonlinear operators with a  variational structure, which will admit \typo{priori}{a priori} estimates. For instance, in the case where
	\[
	Au=-\mathrm{div}(F'(\nabla u)), \text{ with }\frac{F'(\nabla u)}{\abs{F'(\nabla u)}}= \frac{\nabla u}{\abs{\nabla u}},
	\]
	with some additional homogeneous grown conditions. In this case, analogous a priori estimates will be available. Additionally, considering variable coefficients or quasilinearity in the hyperbolic component will be permissible, yet will be strongly tied to  geometric assumptions on $(\Omega_W,\Gamma_W)$.
	Indeed, assumptions on an operator of the form
	\[
	Lw= -\mathrm{div}(B^T B\nabla u)
	\]
	will yield respective assumptions on the existence of the vector field $\mathbf b$ utilized in our analysis of a priori estimates. In all cases, the coupling conditions across $\Gamma$ will be determined by the respective co-normal derivatives for $A$ and $L$.

 \section{Geometric Setting}\label{geomsec}
The geometry of wave domains is central to our analysis. While we will place no additional assumptions on the parabolic domain, $\Omega_H$, the hyperbolic domain $\Omega_W$ will be restricted. 
One of the main contributions here is to exposit the connection between: (a) the geometry of the tuple $(\Omega_W,\Gamma_W)$, (b) the existence of a finite energy weak solution, and (c) the regularity loss between forcing data and the solution.  In line with classical control theory \cite{russell}, we are concerned with whether the entire hyperbolic domain ``sees" the boundary interface, across which dissipation is propagated. This is equivalent to asking that each point in the wave domain and the non-interface boundary $\Gamma_W$ is ``visible" from $\Gamma$, in a sense we will make precise. 

Our starting point for this discussion is the so-called {\em Geometric Optics Condition} or {\em Geometric Control Condition},  prominent in the theory of boundary control \cite{lt1,lt2}---see the detailed discussion in \cite[Section 1]{ZZ2} with regard to the heat-wave dynamics.
\begin{condition}[Geometric Optics Condition] \label{dotcond}
We say that $\Omega_W$ satisfies the {\em geometric optics condition} if there exists an $\mathbf x_0 \in \mathbb R^d$ so that for $\mathbf h(\mathbf x) = \mathbf x -\mathbf x_0$ we have the property
\[ \mathbf h \cdot \mathbf n \le 0\text{ on }\Gamma_W,\] where $\mathbf n$ here is the unit outward normal to $\Omega_W$.
\end{condition}
Star-complemented domains, which appear in classical works on boundary control of the wave equation \cite{russell}, are defined below and satisfy Condition \ref{dotcond}.
\begin{definition}[Star-Complemented Domain]
We say that the pair $(\Omega_W,\Gamma_W)$ is {\em star-complemented} iff there exists a star-shaped domain $\Omega^*$ so that:
$$ \Gamma_W \subseteq \partial\Omega^* ~\text{ and }~ \Omega_W \subseteq [\overline{\Omega^*}]^c.$$
\end{definition} 
\begin{figure} \begin{center}
\includegraphics[scale=.4]{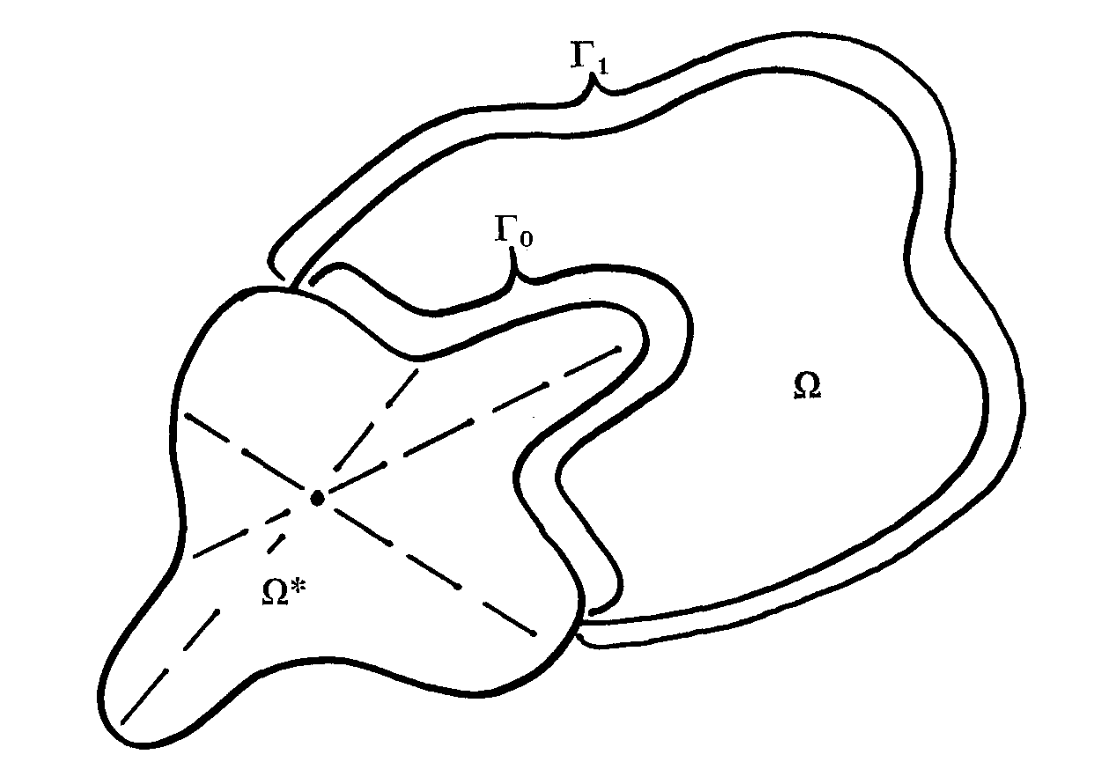} \caption{Star-complemented domain $\Omega$, as taken from \cite{russell}. Here, $\Gamma_0$ is the inactive (zero Dirichlet) portion of $\partial \Omega$ and $\Gamma_1$ is the active (or interface) portion.} \end{center} \end{figure}

In our work, we extend the above condition in two ways. First, more general vector fields will be included by replacing the vector field $\mathbf{h}$ by a general flow direction $\flow$ specifically tied to the geometry of the domain. These more general domains---described by flows---provide many interesting families of domains that may have singularities or nontrivial topologies. We provide several such examples below, and also in Appendix A. 
Secondly, using the particular features of time-periodic solutions, one can relax traditional optics conditions for a weaker condition, at the price of solution regularity. The central idea is that it is not necessary in time-periodic problems to produce traditional $H^1$  energy estimates of the wave equation as would be done for the corresponding Cauchy problem. In the time-periodic framework, a certain flexibility is introduced, permitting integration by parts in time on a full period $[0,T]$ without boundary terms in time; this permits the exchange of temporal derivatives for spatial derivatives and leads to a Poincar\'e-type condition for the wave domain. The latter includes domains such as boxes and hemispheres, which are excluded from the classical condition in Definition \ref{dotcond}.  A natural extension in this framework is first to consider a wave domain $\Omega_W$ which lies between two graphs.

\subsection{Graph and Flow Frameworks for $\Omega_W$}
Let us consider a general graph framework for domains $\Omega_W$ of interest. Take  an open set $\Omega'\subset\R^{d-1}$ and two 
 functions $\phi_1,\phi_2:\Omega'\to \R$, with $\phi_1$ continuous and $\phi_2$ upper semi-continuous given.
 \begin{condition}[Graph  Condition]\label{def:graph}
	If $\Omega_W \equiv \{(\mathbf x', y)\in \Omega'\times \R\,:\,\phi_1(\mathbf x')<  y < \phi_2(\mathbf x')\}$, and we can describe  $\Gamma\equiv \{(\mathbf x',y)\in \Omega'\times \R\,:\,\phi_1(\mathbf x')= y \}$ and $\Gamma_W\equiv\partial\Omega_W\setminus \overline{\Gamma}$, we then say that pair $(\Omega_W,\Gamma_W)$
	satisfies the {\em Graph  Condition}.
\end{condition}

\begin{figure}[htb]
	\centering
	\resizebox{!}{6cm}{%
		\tikzset{every picture/.style={line width=0.75pt}} 

\begin{tikzpicture}[x=0.75pt,y=0.75pt,yscale=-1,xscale=1]

\draw  [draw opacity=0][fill={rgb, 255:red, 220; green, 220; blue, 220 }  ,fill opacity=1 ] (259.51,157.42) -- (359.51,157.42) -- (359.51,229.42) -- (259.51,229.42) -- cycle ;
\draw [color={rgb, 255:red, 220; green, 220; blue, 220 }  ,draw opacity=1 ]   (259.36,157.12) -- (359.51,157.42) ;
\draw [draw opacity=0][fill={rgb, 255:red, 220; green, 220; blue, 220 }  ,fill opacity=1 ]   (168.95,273.42) .. controls (208.31,295.82) and (228.95,303.42) .. (268.95,273.42) .. controls (308.95,243.42) and (347.51,266.22) .. (365.91,302.22) .. controls (384.31,338.22) and (441.91,322.22) .. (459.51,258.22) .. controls (477.11,194.22) and (414.11,223.1) .. (348.49,226.45) .. controls (282.87,229.8) and (205.83,239.62) .. (168.95,233.5) ;
\draw  [color={rgb, 255:red, 220; green, 220; blue, 220 }  ,draw opacity=1 ][fill={rgb, 255:red, 220; green, 220; blue, 220 }  ,fill opacity=1 ] (448.4,245.9) .. controls (428.11,246.94) and (407.15,256.5) .. (385.82,258.22) -- (373.43,258.22) .. controls (368.8,257.81) and (364.16,256.88) .. (359.51,255.28) -- (359.51,229.42) .. controls (394.15,275.15) and (390.71,223.02) .. (459.51,229.42) -- (459.51,245.1) .. controls (455.71,245.45) and (451.94,245.72) .. (448.4,245.9) -- cycle ;
\draw  [color={rgb, 255:red, 255; green, 255; blue, 255 }  ,draw opacity=1 ][fill={rgb, 255:red, 255; green, 255; blue, 255 }  ,fill opacity=1 ] (359.51,225.66) .. controls (380.02,223.82) and (399.84,220.11) .. (416.52,218.22) -- (446.07,218.22) .. controls (452.12,219.58) and (456.74,222.25) .. (459.51,226.84) -- (459.51,229.42) .. controls (390.71,223.02) and (394.15,275.15) .. (359.51,229.42) -- (359.51,225.66) -- cycle ;
\draw [color={rgb, 255:red, 220; green, 220; blue, 220 }  ,draw opacity=1 ][fill={rgb, 255:red, 220; green, 220; blue, 220 }  ,fill opacity=1 ]   (259.51,157.42) .. controls (299.51,127.42) and (310.87,107.34) .. (359.51,157.42) ;
\draw [draw opacity=0][fill={rgb, 255:red, 220; green, 220; blue, 220 }  ,fill opacity=1 ]   (168.95,233.5) .. controls (208.95,203.5) and (219.51,227.42) .. (259.51,197.42) .. controls (299.51,167.42) and (367.51,217.74) .. (305.91,236.94) .. controls (244.31,256.14) and (169.11,291.34) .. (168.95,233.5) -- cycle ;
\draw  [draw opacity=0][fill={rgb, 255:red, 255; green, 255; blue, 255 }  ,fill opacity=1 ] (416.52,218.22) .. controls (429.62,216.65) and (440.91,216.07) .. (449.26,218.22) -- (416.52,218.22) -- cycle ;
\draw    (359.51,229.42) .. controls (393.91,275.02) and (390.71,223.02) .. (459.51,229.42) ;
\draw  [color={rgb, 255:red, 255; green, 255; blue, 255 }  ,draw opacity=1 ][fill={rgb, 255:red, 255; green, 255; blue, 255 }  ,fill opacity=1 ] (459.51,258.22) -- (459.51,226.84) .. controls (463.33,233.15) and (463.67,243.09) .. (459.51,258.22) -- cycle ;
\draw    (151.2,333.12) -- (494.96,333.12) ;
\draw [shift={(496.96,333.12)}, rotate = 180] [color={rgb, 255:red, 0; green, 0; blue, 0 }  ][line width=0.75]    (10.93,-3.29) .. controls (6.95,-1.4) and (3.31,-0.3) .. (0,0) .. controls (3.31,0.3) and (6.95,1.4) .. (10.93,3.29)   ;
\draw    (151.2,333.12) -- (151.2,104.2) ;
\draw [shift={(151.2,102.2)}, rotate = 90] [color={rgb, 255:red, 0; green, 0; blue, 0 }  ][line width=0.75]    (10.93,-3.29) .. controls (6.95,-1.4) and (3.31,-0.3) .. (0,0) .. controls (3.31,0.3) and (6.95,1.4) .. (10.93,3.29)   ;
\draw  [dash pattern={on 4.5pt off 4.5pt}]  (359.36,333.12) -- (359.36,157.12) ;
\draw  [dash pattern={on 4.5pt off 4.5pt}]  (259.36,333.12) -- (259.36,157.12) ;
\draw    (168.8,233.2) .. controls (208.8,203.2) and (219.36,227.12) .. (259.36,197.12) ;
\draw    (259.36,157.12) .. controls (299.36,127.12) and (310.72,107.04) .. (359.36,157.12) ;
\draw  [dash pattern={on 4.5pt off 4.5pt}]  (459.36,333.12) -- (459.36,229.12) ;
\draw    (168.8,273.12) .. controls (208.16,295.52) and (228.8,303.12) .. (268.8,273.12) .. controls (308.8,243.12) and (347.36,265.92) .. (365.76,301.92) .. controls (384.16,337.92) and (441.76,321.92) .. (459.36,257.92) ;
\draw  [dash pattern={on 4.5pt off 4.5pt}]  (168.8,333.2) -- (168.8,233.2) ;

\draw (486.44,353.07) node    {$x'\in \Omega '$};
\draw (131.77,103.59) node    {$x_{d}$};
\draw (307.83,278.39) node    {$\phi _{1}( x')$};
\draw (223.03,190.79) node    {$\phi _{2}( x')$};
\draw (301.58,222.26) node  [font=\large]  {$\Omega _{W}$};

\end{tikzpicture}
	}
	\caption{Graph setting.}
\end{figure}
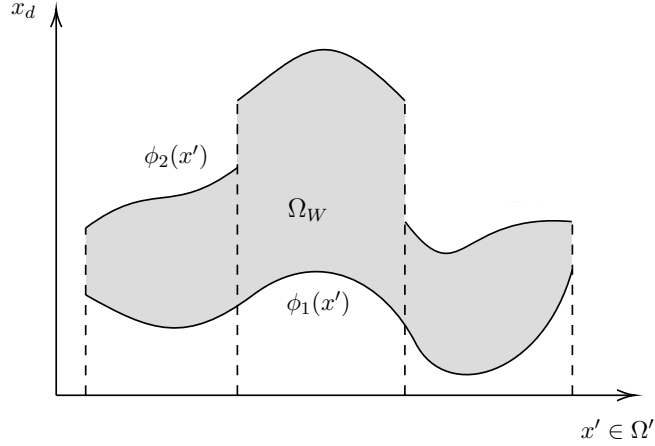
	
Another generalized setting replaces the {\em fixed} direction ${\bf h}={\bf x}- {\bf x}_0$ by a general vector field ${\bf b}:\Omega_W \to \R^d$ tailored to the domain $\Omega_W$. We start with the following definition.
\begin{definition}[Contractive Domain] \label{Contractive-Domain}
	We say that a domain $\Omega$ is {\em contractive} if there exists a Lipschitz vector field $\flow:\Omega\to \R^d$ and a  constant $C>0$ such that for almost every $\mathbf x\in\Omega$ and every $\bxi \in \R^d$ the flow $\flow$ satisfies
	\begin{align}
	\label{eq:contraction}
		\bxi^T \grad \flow(\mathbf{x}) \bxi \geq C \abs{\bxi}^2.
	\end{align} 
\end{definition}
 Observe that the above assumption provides immediately that $\diverg \flow \geq \frac{C}{2}$. In order to generalize the {\em  Geometric Optics Condition}, further assumptions on $\diverg \flow $ have to be assumed.
	\begin{condition}[Generalized Optics Condition] \label{def:GOC}
	We say that pair $(\Omega_W,\Gamma_W)$ satisfies the {\em Generalized Optics Condition} if it is a {\em contractive domain} with vector field $\flow$ and the following additional assumptions are satisfied
	\begin{align}\label{GOCSign}
	\flow\cdot \mathbf n\leq 0~\text{ on }~\Gamma_W
	\end{align}
	and further
	\begin{align}
	\label{eq:annoying}
	\Delta \diverg \flow \leq 0~\text{ in }~\Omega_W.
	\end{align}
\end{condition}
 
\begin{figure}[htbp]
	\centering
	\begin{minipage}[b]{0.4\textwidth}
		\centering
		\resizebox{!}{2cm}{%
			\tikzset{every picture/.style={line width=0.75pt}} 

\begin{tikzpicture}[x=0.75pt,y=0.75pt,yscale=-1,xscale=1]

\draw  [fill={rgb, 255:red, 0; green, 0; blue, 0 }  ,fill opacity=0.1 ] (280.8,475.65) -- (68.5,529.9) -- (68.5,421.4) -- cycle ;
\draw  [dash pattern={on 4.5pt off 4.5pt}]  (60.17,452.46) -- (280.79,475.65) ;
\draw  [dash pattern={on 4.5pt off 4.5pt}]  (59.1,467.91) -- (280.8,475.65) ;
\draw  [dash pattern={on 4.5pt off 4.5pt}]  (63,435.9) -- (280.79,475.65) ;
\draw  [dash pattern={on 4.5pt off 4.5pt}]  (59.11,483.39) -- (280.79,475.65) ;
\draw  [dash pattern={on 4.5pt off 4.5pt}]  (60.2,498.84) -- (280.79,475.65) ;
\draw  [dash pattern={on 4.5pt off 4.5pt}]  (62.36,514.17) -- (280.83,475.65) ;
\draw  [color={rgb, 255:red, 255; green, 255; blue, 255 }  ,draw opacity=1 ][fill={rgb, 255:red, 255; green, 255; blue, 255 }  ,fill opacity=1 ] (58.5,420.8) -- (67.5,420.8) -- (67.5,531.8) -- (58.5,531.8) -- cycle ;
\draw [line width=1.5]    (68.5,421.4) -- (68.5,529.9) ;
\draw  [fill={rgb, 255:red, 0; green, 0; blue, 0 }  ,fill opacity=1 ] (279.21,475.65) .. controls (279.21,474.77) and (279.92,474.06) .. (280.79,474.06) .. controls (281.67,474.06) and (282.38,474.77) .. (282.38,475.65) .. controls (282.38,476.53) and (281.67,477.24) .. (280.79,477.24) .. controls (279.92,477.24) and (279.21,476.53) .. (279.21,475.65) -- cycle ;

\draw (291.22,458.72) node    {$( 0,0)$};

\end{tikzpicture}
		}
		\caption{Triangular Domain, satisfying the Geometric Optics Condition. Described by $\flow = (x, y)$.}
	\end{minipage}
	\hspace{1cm}
	\hspace{-0.5cm}
	\begin{minipage}[b]{0.4\textwidth}
		\centering
		\resizebox{!}{2cm}{%
			\tikzset{every picture/.style={line width=0.75pt}} 

\begin{tikzpicture}[x=0.75pt,y=0.75pt,yscale=-1,xscale=1]

\draw  [color={rgb, 255:red, 255; green, 255; blue, 255 }  ,draw opacity=0 ][fill={rgb, 255:red, 0; green, 0; blue, 0 }  ,fill opacity=0.1 ] (493.2,692.8) .. controls (493.27,692.91) and (493.34,693.02) .. (493.41,693.12) -- (492.01,693.44) .. controls (447.09,658.23) and (343.44,625.76) .. (224.42,718.85) -- (206.55,639.58) .. controls (268.4,601.66) and (425.83,590) .. (493.2,692.8) -- cycle ;
\draw    (206.2,639.8) .. controls (271.7,602.3) and (425.7,589.8) .. (493.2,692.8) ;
\draw    (223.2,719.8) .. controls (347.7,623.8) and (446.7,659.8) .. (493.2,692.8) ;
\draw [line width=1.5]    (206.2,639.8) -- (223.2,719.8) ;
\draw  [dash pattern={on 4.5pt off 4.5pt}]  (220.29,705.4) .. controls (344.79,609.4) and (454.5,658.8) .. (493.2,692.8) ;
\draw    (206.78,640.59) -- (209.69,654.99) ;
\draw    (209.69,654.99) -- (212.6,669.4) ;
\draw    (212.6,669.4) -- (215.51,683.8) ;
\draw    (217.38,690.99) -- (220.29,705.4) ;
\draw    (220.29,705.4) -- (223.2,719.8) ;
\draw  [dash pattern={on 4.5pt off 4.5pt}]  (216.6,687.11) .. controls (343,599.3) and (463.5,659.3) .. (493.2,692.8) ;
\draw  [dash pattern={on 4.5pt off 4.5pt}]  (209.69,654.99) .. controls (275.19,617.49) and (422.5,602.8) .. (493.2,692.8) ;
\draw  [dash pattern={on 4.5pt off 4.5pt}]  (212.6,669.4) .. controls (278,631.3) and (412.5,606.3) .. (493.2,692.8) ;
\draw  [fill={rgb, 255:red, 0; green, 0; blue, 0 }  ,fill opacity=1 ] (491.61,692.8) .. controls (491.61,691.92) and (492.32,691.21) .. (493.2,691.21) .. controls (494.08,691.21) and (494.79,691.92) .. (494.79,692.8) .. controls (494.79,693.68) and (494.08,694.39) .. (493.2,694.39) .. controls (492.32,694.39) and (491.61,693.68) .. (491.61,692.8) -- cycle ;

\draw (466.32,703.27) node  [font=\large]  {$( 0,0)$};

\end{tikzpicture}
		}
		\caption{Horn Domain, satisfying the Generalized Optics Condition but not the Geometric Optics Condition. Described by $\flow = (\beta x, y)$, with $\beta > 0$.}
	\end{minipage}
\end{figure}
Figure 4 satisfies the Geometric Optics Condition \ref{dotcond}, while Figure 5 satisfies the Generalized Optics Condition \ref{def:GOC}. 
Finally, we can relax the above condition further by considering a generalization of the graph setting in Figure 6; the example in Figure 6 will satisfy our {\em Graph Optics Condition} \ref{def:GGC}, however relaxing the contractivity property. (In these examples, the dark boundary line represents $\Gamma$, the interface.)

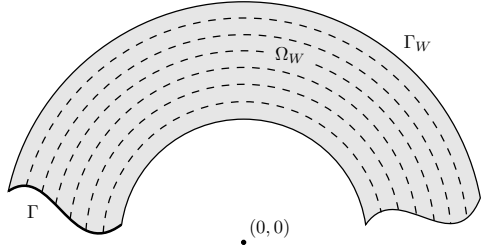
\begin{figure}[htb]
	\label{fig:arc}
	\centering
	\resizebox{!}{3.5cm}{%
		\tikzset{every picture/.style={line width=0.75pt}} 

\begin{tikzpicture}[x=0.75pt,y=0.75pt,yscale=-1,xscale=1]

\draw  [draw opacity=0][dash pattern={on 4.5pt off 4.5pt}] (191.13,209.56) .. controls (204.74,140.68) and (265.48,88.75) .. (338.35,88.75) .. controls (415.55,88.75) and (479.14,147.05) .. (487.48,222.03) -- (338.35,238.81) -- cycle ; \draw  [dash pattern={on 4.5pt off 4.5pt}] (191.13,209.56) .. controls (204.74,140.68) and (265.48,88.75) .. (338.35,88.75) .. controls (415.55,88.75) and (479.14,147.05) .. (487.48,222.03) ;  
\draw  [color={rgb, 255:red, 0; green, 0; blue, 0 }  ,draw opacity=0 ][fill={rgb, 255:red, 255; green, 255; blue, 255 }  ,fill opacity=1 ] (356.21,95.01) .. controls (356.21,88.6) and (364.27,83.4) .. (374.21,83.4) .. controls (384.15,83.4) and (392.21,88.6) .. (392.21,95.01) .. controls (392.21,101.42) and (384.15,106.62) .. (374.21,106.62) .. controls (364.27,106.62) and (356.21,101.42) .. (356.21,95.01) -- cycle ;
\draw  [fill={rgb, 255:red, 0; green, 0; blue, 0 }  ,fill opacity=0.1 ] (338.35,50.36) .. controls (430.55,50.36) and (507.28,116.57) .. (523.6,204.04) .. controls (523.54,204.19) and (523.48,204.33) .. (523.42,204.48) .. controls (502.92,254.73) and (460.43,188.48) .. (434.43,224.23) .. controls (434.24,224.48) and (434.05,224.73) .. (433.85,224.97) .. controls (427.14,178.24) and (386.94,142.32) .. (338.35,142.32) .. controls (289.82,142.32) and (249.66,178.14) .. (242.87,224.78) .. controls (242.8,224.83) and (242.74,224.88) .. (242.68,224.93) .. controls (202.77,254.86) and (188.24,175.78) .. (154.16,198.77) .. controls (172.52,113.92) and (248.01,50.36) .. (338.35,50.36) -- cycle ;
\draw  [draw opacity=0][dash pattern={on 4.5pt off 4.5pt}] (202.47,220.97) .. controls (211.22,153.73) and (268.72,101.79) .. (338.35,101.79) .. controls (406.71,101.79) and (463.37,151.85) .. (473.69,217.3) -- (338.35,238.81) -- cycle ; \draw  [dash pattern={on 4.5pt off 4.5pt}] (202.47,220.97) .. controls (211.22,153.73) and (268.72,101.79) .. (338.35,101.79) .. controls (406.71,101.79) and (463.37,151.85) .. (473.69,217.3) ;  
\draw  [draw opacity=0][dash pattern={on 4.5pt off 4.5pt}] (214.88,229.59) .. controls (219.59,165.52) and (273.07,115.01) .. (338.35,115.01) .. controls (398.15,115.01) and (448.05,157.41) .. (459.62,213.79) -- (338.35,238.81) -- cycle ; \draw  [dash pattern={on 4.5pt off 4.5pt}] (214.88,229.59) .. controls (219.59,165.52) and (273.07,115.01) .. (338.35,115.01) .. controls (398.15,115.01) and (448.05,157.41) .. (459.62,213.79) ;  
\draw  [draw opacity=0][dash pattern={on 4.5pt off 4.5pt}] (180.37,199.84) .. controls (197.83,128.82) and (261.94,76.14) .. (338.35,76.14) .. controls (423.33,76.14) and (493.09,141.3) .. (500.39,224.38) -- (338.35,238.81) -- cycle ; \draw  [dash pattern={on 4.5pt off 4.5pt}] (180.37,199.84) .. controls (197.83,128.82) and (261.94,76.14) .. (338.35,76.14) .. controls (423.33,76.14) and (493.09,141.3) .. (500.39,224.38) ;  
\draw  [draw opacity=0][dash pattern={on 4.5pt off 4.5pt}] (168.09,195.07) .. controls (187.53,119.17) and (256.39,63.07) .. (338.35,63.07) .. controls (429.11,63.07) and (503.8,131.87) .. (513.12,220.17) -- (338.35,238.81) -- cycle ; \draw  [dash pattern={on 4.5pt off 4.5pt}] (168.09,195.07) .. controls (187.53,119.17) and (256.39,63.07) .. (338.35,63.07) .. controls (429.11,63.07) and (503.8,131.87) .. (513.12,220.17) ;  
\draw  [draw opacity=0][dash pattern={on 4.5pt off 4.5pt}] (228.39,231.68) .. controls (232.07,174.16) and (279.89,128.63) .. (338.35,128.63) .. controls (391.15,128.63) and (435.27,165.77) .. (446.02,215.36) -- (338.35,238.81) -- cycle ; \draw  [dash pattern={on 4.5pt off 4.5pt}] (228.39,231.68) .. controls (232.07,174.16) and (279.89,128.63) .. (338.35,128.63) .. controls (391.15,128.63) and (435.27,165.77) .. (446.02,215.36) ;  
\draw [line width=1.5]    (154.16,198.77) .. controls (188.21,175.58) and (202.68,254.93) .. (242.68,224.93) ;
\draw  [fill={rgb, 255:red, 0; green, 0; blue, 0 }  ,fill opacity=1 ] (336.76,238.81) .. controls (336.76,237.94) and (337.47,237.23) .. (338.35,237.23) .. controls (339.22,237.23) and (339.93,237.94) .. (339.93,238.81) .. controls (339.93,239.69) and (339.22,240.4) .. (338.35,240.4) .. controls (337.47,240.4) and (336.76,239.69) .. (336.76,238.81) -- cycle ;

\draw (173.6,214.8) node    {$\Gamma $};
\draw (474.8,82.8) node    {$\Gamma _{W}$};
\draw (374.8,92.41) node    {$\Omega _{W}$};
\draw (358.82,227.67) node    {$( 0,0)$};

\end{tikzpicture}
	}
	\caption{Arc Domain. See Appendix~A for more discussion,~\ref{appendix:Arc-Construction}.}
\end{figure}

 \begin{condition}[Graph Optics Condition] \label{def:GGC}
	We say that pair $(\Omega,\Gamma_W)$ satisfies the {\em  Graph Optics Condition}, if there is a Lipschitz vector field $\flow:\Omega\to \R^d $satisfying
	\begin{itemize} 
		 \item inequalities \eqref{GOCSign} and \eqref{eq:annoying}, 
	\item 
	$ \bxi^T \grad \flow(\mathbf{x}) \bxi \geq C\abs{\bxi \cdot \flow(\mathbf x)}^2$ for $a.e.$ $\mathbf x$ and every $\bxi \in \mathbf R^d$,
	\item the following Poincar\'e type inequality:	for all $f \in H^1_{\Gamma_W}(\Omega_W)$
	\begin{align}\label{GPoincare}
	||f||_{L^2(\Omega_W)}^2 
			&\lesssim
			\int_{\Omega_W} (\grad f)^T \grad\flow \;\grad f \;
			+ 
			\norm{f}_{L^2(\Gamma)}^2- 
			\int_{\Gamma_W} (\flow\cdot \nrml) \abs{\pdn f}^2 \;.
	\end{align}

	\end{itemize}
	\end{condition}
Finally, we must introduce assumptions for some elliptic regularity, particularly relevant for the analysis of boundary triple points, where $\Gamma_W$, $\Gamma_H$, and $\Gamma$ have their closures meet. The following notation will be used in the sequel. We will say that $f\in H^{s^+}(\Omega_{\square})$ if there exists $\epsilon>0$ such that $f\in H^{s+\epsilon}(\Omega_{\square})$.
We now introduce these technical conditions on the wave domain, $\Omega_W$.
\begin{definition}\label{EllipticRegCond}
We say that a domain $\Omega_W$ has the {\em elliptic regularity property} w.r.t. the vector field $\flow(\mathbf x)$ if the following property holds:

Let $\tilde{\Gamma}_W=\{\mathbf x\in\partial \Omega_W:~\flow\cdot \mathbf n \leq 0\}$ and $\tilde{\Gamma}=\partial\Omega_W\setminus  \overline{\tilde{\Gamma}_W}$. Let {$f\in L^2(\Omega_W)$ and} $v$ be a unique solution to the following mixed boundary value problem for the Laplacian:
$$
\Delta v={f}\;{\rm in}\;\Omega_W,\quad v=0\;{\rm on}\; \tilde{\Gamma}_W,\quad \partial_{\mathbf n}v=0\;{\rm on}\; \tilde{\Gamma}.
$$
Then $v\in H^1_{\Gamma_W}(\Omega_W) \cap H^{\frac{3}{2}^+}(\Omega_W)$.
\end{definition}

\begin{definition}[E-property]\label{EProperty}
	We say that the Lipschitz domain $\Omega_W$ can be {\em approximated by the domains with the elliptic regularity property w.r.t. $\flow$}  if there exists a sequence of domains $\Omega_W^n$ satisfying the elliptic regularity property w.r.t. $\flow$ and 
	$$
	\chi_{\Omega_W^n}\to\chi_{\Omega_W}\;{\rm strongly\;in}\; L^p(\R^d),\;p<\infty,
	$$
	where $\chi_{\Omega_W^n}$ and $\chi_{\Omega_W}$ are the characteristic functions of $\Omega_W^n$ and $\Omega_W$, resp. In this case, we say that such a domain $\Omega_W$ has the {\em E-property} for short.
\end{definition}

\subsection{Remarks on Geometric Conditions}\label{geomarks} We now make several comments about the various conditions above that generalize the classical Geometric Optics Condition in Condition \ref{dotcond}.
	\begin{itemize}
		\item It is immediate that if $(\Omega_W,\Gamma_W)$ satisfies the Geometric Optic Condition in Definition \ref{dotcond}, then it also satisfies the Generalized Optics Condition in Definition \ref{def:GOC}. Namely, the linear flow $\mathbf b({\bf x})={\bf x}-{\bf x_0}$ trivially satisfies conditions \eqref{eq:contraction} and \eqref{eq:annoying}, while the sign condition \eqref{def:GOC} is the same in both definitions.
		\item Moreover, any  domain satisfying the Graph  Condition \ref{def:graph} also satisfies the Graph Optics Condition \ref{def:GGC}. To see this,  choose $c_0$ such that $\phi_2(\mathbf x')-c_0\leq 0$, $\mathbf x'\in\Omega'$, and $\flow=(0,...,0,x_d-c_0).$ Then the classical Poincar\'e inequality holds
		$$
		\|f\|_{L^2(\Omega_W)}\leq C\left (\norm{\partial_{x_d}f}_{L^2(\Omega_W)}+\norm{f}_{L^2(\Gamma))} \right ),
		$$
		which, together with the sign condition $\flow\cdot{\bf n}\leq 0$ on $\Gamma_W$, implies \eqref{GPoincare}.
		\item In our analysis, some elliptic regularity is needed for the wave component of the model. In particular, we will require interior $\Omega_W$ regularity of $H^{3/2^+}(\Omega_W)$. As such, we introduce the E-property, Definition \ref{EProperty}, above. We note that this property is immediately satisfied \cite{dauge} if the maximal wave dihedral angle of $\Omega_W$ is strictly less than $\pi$ at boundary triple points. This is to say that the E-property generally holds for domains heretofore considered. 
		\item Furthermore, any domains which satisfy the graph condition, Condition \ref{def:graph},  satisfy the E-property.  
	\end{itemize}
	
	\subsection{Instructive Geometric Example}
	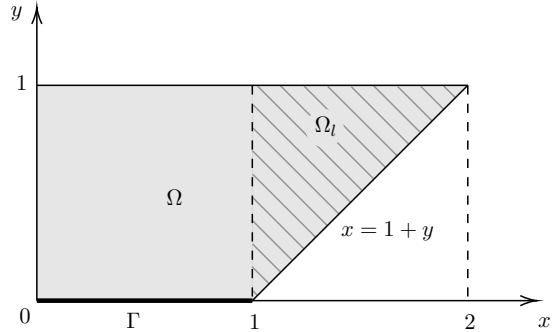
\begin{figure}[htb]
	\label{fig:trap}
	\centering
	\resizebox{!}{4.5cm}{%
		\tikzset{every picture/.style={line width=0.75pt}} 

\begin{tikzpicture}[x=0.75pt,y=0.75pt,yscale=-1,xscale=1]
	
	\draw  [draw opacity=0][fill={rgb, 255:red, 0; green, 0; blue, 0 }  ,fill opacity=0.1 ][line width=6]  (165.33,145.73) -- (306,145.73) -- (306,286.4) -- (165.33,286.4) -- cycle ;
	\draw [color={rgb, 255:red, 0; green, 0; blue, 0 }  ,draw opacity=0.35 ]   (339.82,146.22) -- (392.98,199.39) ;
	\draw [color={rgb, 255:red, 0; green, 0; blue, 0 }  ,draw opacity=0.35 ]   (309.49,146.3) -- (377.51,214.31) ;
	\draw [color={rgb, 255:red, 0; green, 0; blue, 0 }  ,draw opacity=0.35 ]   (323.97,145.57) -- (385.38,206.99) ;
	\draw  [color={rgb, 255:red, 255; green, 255; blue, 255 }  ,draw opacity=1 ][fill={rgb, 255:red, 255; green, 255; blue, 255 }  ,fill opacity=1 ] (341.23,173.8) .. controls (341.23,167.17) and (346.61,161.8) .. (353.23,161.8) .. controls (359.86,161.8) and (365.23,167.17) .. (365.23,173.8) .. controls (365.23,180.43) and (359.86,185.8) .. (353.23,185.8) .. controls (346.61,185.8) and (341.23,180.43) .. (341.23,173.8) -- cycle ;
	\draw  [color={rgb, 255:red, 0; green, 0; blue, 0 }  ,draw opacity=0 ][fill={rgb, 255:red, 0; green, 0; blue, 0 }  ,fill opacity=0.1 ][line width=6]  (446.67,145.73) -- (306,286.4) -- (306,145.73) -- cycle ;
	\draw    (165.33,286.4) -- (165.33,96.67) ;
	\draw [shift={(165.33,94.67)}, rotate = 90] [color={rgb, 255:red, 0; green, 0; blue, 0 }  ][line width=0.75]    (10.93,-3.29) .. controls (6.95,-1.4) and (3.31,-0.3) .. (0,0) .. controls (3.31,0.3) and (6.95,1.4) .. (10.93,3.29)   ;
	\draw    (165.33,286.4) -- (489.87,286.4) ;
	\draw [shift={(491.87,286.4)}, rotate = 180] [color={rgb, 255:red, 0; green, 0; blue, 0 }  ][line width=0.75]    (10.93,-3.29) .. controls (6.95,-1.4) and (3.31,-0.3) .. (0,0) .. controls (3.31,0.3) and (6.95,1.4) .. (10.93,3.29)   ;
	\draw    (306,286.4) -- (446.67,145.73) ;
	\draw    (165.33,145.73) -- (446.67,145.73) ;
	\draw  [dash pattern={on 4.5pt off 4.5pt}]  (446.67,145.73) -- (446.67,286.4) ;
	\draw  [dash pattern={on 4.5pt off 4.5pt}]  (306,145.73) -- (306,286.4) ;
	\draw [color={rgb, 255:red, 0; green, 0; blue, 0 }  ,draw opacity=0.35 ]   (306.09,188.5) -- (354.98,237.39) ;
	\draw [color={rgb, 255:red, 0; green, 0; blue, 0 }  ,draw opacity=0.35 ]   (306.24,203.85) -- (347.38,244.99) ;
	\draw [color={rgb, 255:red, 0; green, 0; blue, 0 }  ,draw opacity=0.35 ]   (305.77,218.57) -- (339.78,252.59) ;
	\draw [color={rgb, 255:red, 0; green, 0; blue, 0 }  ,draw opacity=0.35 ]   (306.04,234.05) -- (332.18,260.19) ;
	\draw [color={rgb, 255:red, 0; green, 0; blue, 0 }  ,draw opacity=0.35 ]   (306.22,249.42) -- (324.58,267.79) ;
	\draw [color={rgb, 255:red, 0; green, 0; blue, 0 }  ,draw opacity=0.35 ]   (305.87,264.27) -- (316.98,275.39) ;
	\draw [color={rgb, 255:red, 0; green, 0; blue, 0 }  ,draw opacity=0.35 ]   (306.57,173.77) -- (362.58,229.79) ;
	\draw [color={rgb, 255:red, 0; green, 0; blue, 0 }  ,draw opacity=0.35 ]   (306.57,158.57) -- (370.18,222.19) ;
	\draw [color={rgb, 255:red, 0; green, 0; blue, 0 }  ,draw opacity=0.35 ]   (355.17,146.37) -- (400.58,191.79) ;
	\draw [color={rgb, 255:red, 0; green, 0; blue, 0 }  ,draw opacity=0.35 ]   (370.64,146.65) -- (408.18,184.19) ;
	\draw [color={rgb, 255:red, 0; green, 0; blue, 0 }  ,draw opacity=0.35 ]   (385.22,146.02) -- (415.78,176.59) ;
	\draw [color={rgb, 255:red, 0; green, 0; blue, 0 }  ,draw opacity=0.35 ]   (400.94,146.55) -- (423.38,168.99) ;
	\draw [color={rgb, 255:red, 0; green, 0; blue, 0 }  ,draw opacity=0.35 ]   (416.04,146.45) -- (430.98,161.39) ;
	\draw [color={rgb, 255:red, 0; green, 0; blue, 0 }  ,draw opacity=0.35 ]   (430.98,146.19) -- (438.58,153.79) ;
	\draw [color={rgb, 255:red, 0; green, 0; blue, 0 }  ,draw opacity=0.35 ]   (306.37,279.97) -- (309.38,282.99) ;
	\draw [line width=2.25]    (165.33,286.4) -- (306,286.4) ;
	
	\draw (157.56,296.53) node    {$0$};
	\draw (307.56,300.53) node    {$1$};
	\draw (448.23,300.53) node    {$2$};
	\draw (496.89,300.53) node    {$x$};
	\draw (152.23,98.53) node    {$y$};
	\draw (155.56,143.87) node    {$1$};
	\draw (394.23,241.2) node    {$x=1+y$};
	\draw (255.29,218.8) node    {$\Omega $};
	\draw (354.14,172.45) node    {$\Omega _{l}$};
	\draw (228.49,299.2) node    {$\Gamma $};

\end{tikzpicture}
	}
	\caption{Trapezoid Domain.}
\end{figure}
	We now produce an example domain which does not fall into any of our geometric frameworks. Consider the domain $\Omega=\{(x,y)\in [0,2]\times[0,1]\, :\,0\leq x\leq 1+y\}$ with $\Gamma=\{(x,0):0\leq x\leq 1\}$ as is sketched in Figure~\ref{fig:trap}). Then, no field $\mathbf b$ can exist satisfying the Generalized Optics Condition.
Define $\Omega_\ell=\{(x,y)\in [1,2]\times[0,1]\, :\,1\leq x\leq 1+y\}$.
Now observe that the sign condition on $\Gamma_W$ implies that ${b}_1(1+y,y)-{b}_2(1+y,y)\leq 0$ and $-{b}_1(0,y)\leq 0$ for $y\in [0,1]$ and ${b}_2(x,1)\leq 0$, where $b_j$ are the components of $\mathbf b$. Now we find that
\begin{align*}
	0 
	&<\int_{\Omega} \partial_x {b}_1\, + \int_{\Omega_\ell} \partial_y {b}_2\, 
	\\
	& =\int_0^1 \left[\frac{{b}_1(1+y,y)-{b}_2(1+y,y)}{\sqrt{2}}-{b}_1(0,y)\right] \, dy +\int_1^2{b}_2(x,1)\, dx,
\end{align*}
which leads to a contradiction.
\begin{remark}
	Please observe that the above example can easily be modified to produce an arbitrarily smooth domain. Actually, any domain that ``enlarges" a graph domain (expands, rather than contracts) will violate Definition \ref{def:GOC}, and, as such, will not be covered by our theorems here.
	\end{remark}

\section{Hyperbolic Estimates}\label{Sec:Hyp}
In this section, we derive estimates for an appropriately smooth solution to an overdetermined wave equation. This problem is formally obtained by assuming the heat component $u$ of the system \eqref{sys2}--\eqref{sys2e} is known, while we solve the system for $w$. We emphasize that this overdetermined problem is generally ill-posed. Thus, we do not address the existence of solutions but instead focus on deriving suitable a priori estimates for a given solution. These estimates will later be applied in constructing a solution to the full system \eqref{sys2}--\eqref{sys2e} in Section \ref{sec:construction}.

\subsection{Invariances}\label{Sec:invariances}
It will be necessary, to produce a priori estimates, to exploit certain invariances. These have also arisen in previous works \cite{bostan,casanova}. What we invoke below allows us to integrate the wave solution in time, and, as such, will later permit weaker notions of solutions. 
We first define
\begin{equation}\label{mean}
\mean{f}(\mathbf x)\equiv\frac{1}{T}\int_0^Tf(t,\mathbf x)~~\text{ and }~~\tilde{f}(t,\mathbf x)=f(t,\mathbf x)-\mean{f}(\mathbf x)
\end{equation}
Now observe that on $\Gamma$ we have, through the coupling conditions that $u\big|_{\Gamma} = w_t\big|_{\Gamma},$ when both are defined. Integrating across the interface, we have 
$$\int_0^T u = \int_0^Tw_t=w(T)-w(0)=0,$$
when the function $w$ is $T$-periodic.
 
Thus we  find, by the time-periodicity of $w$, that $\mean{u|_{\Gamma}}=\mean{\partial_t w|_{\Gamma}}=0$, for appropriately smooth solutions.  Hence, we obtain the formal, decoupled equation for the heat component, which can be uniquely solved: 
\begin{align}\label{MeanValueHeatEq}
	-\Delta \mean{u} = \mean{f},\quad \mean{u}=0\text{ on }\partial \Omega_H.
\end{align}
Indeed, we have $\langle u\rangle = A^{-1}\langle f\rangle \in H^1_0(\Omega_H)$ with $\langle f \rangle \in L^2(\Omega_H)$, where $A=-\Delta_D$ yields an isomorphism between $H^1_0(\Omega_H) \to H^{-1}(\Omega_H)$.  Thence  we obtain $\gamma_1[u] \in H^{-1/2}(\partial \Omega_H)$, which can be restricted (in the sense of distributions) to interpret $\partial_{\mathbf n}u|_{\Gamma} \in H^{-1/2}(\Gamma)$. With the normal in hand, we can consider the wave component:
\begin{align}\label{meanw}
	-\Delta \mean{w} = \mean{g},\quad \mean{w}=0\text{ on }\Gamma_W, \text{ and } \partial_{\mathbf n} \mean{w}=\partial_{\mathbf n}\mean{u} \text{ on }\Gamma,
\end{align}
which uniquely determines $\mean{w}$ in an analgous way, albeit with mixed boundary conditions. This yields that the mean value of both solution components are always determined via  decoupling.
Furthermore, the mean-value free part, satisfies the following equation

\begin{align}
\label{meanfree}
\begin{aligned}
	{\tilde{w}}_{tt} -\Delta \tilde{w} &= \tilde{g} \text{ on }[0,T]\times \Omega_W,\quad \tilde{w}(0)=\tilde{w}(T)
	\\
	{\tilde{u}}_t -\Delta \tilde{u} &= \tilde{f}\text{ on }[0,T]\times \Omega_H,\quad \tilde{u}(0)=\tilde{u}(T)
	\\
	\tilde{w}_t&= w_t=u=\tilde{u}\text{ on }[0,T]\times \Gamma
	\\
	\partial_{\mathbf n}\tilde{w}&=\partial_{\mathbf n} (w-\mean{w})=\partial_{\mathbf n} (u-\mean{u})=\partial_{\mathbf n}\tilde{u}\text{ on }[0,T]\times \Gamma.
	\end{aligned}
\end{align}

\subsection{The Hyperbolic Estimate}
In this Section we consider the following over-determined initial-boundary value problem for the wave equation with a right hand side $g\in L^2(0,T;L^2(\Omega_W))$. We will consider boundary condition $h\in L^2(0,T;H^1(\Gamma))$ and first assume $\mean{g}\equiv 0\equiv \mean{h}$. {  We will later  utilize a function $g$ with the property that $\mean{g}\neq 0$; this function will be estimated a posteriori from \eqref{meanw} via elliptic theory. In explicit we consider solutions $w$ to} 
\begin{align}
	\label{weq}
	\begin{aligned}
		w_{tt} -\Delta w &= g\text{ on }[0,T]\times \Omega_W,\quad w(0)=w(T), ~{  w_t(0)=w_t(T)}
		\\
		w_t&= h \text{ and }\partial_{\bn} w= 0 \text{ on }[0,T]\times \Gamma,
		\\
		w &=0\text{ on }\Gamma_W.
	\end{aligned}
\end{align}
From \eqref{weq} of the previous section we always have that $\mean{w}=0$ on $\Gamma$ by time-periodicity. Thus from the fundamental theorem of calculus we have
$$
w(t,x)=\int_{s}^t h(\sigma,x)\, d\sigma+w(s,x),\quad s\in [0,T],\; x\in\Gamma
$$
By taking the mean value, we obtain the identification
\begin{align}\label{PrimitiveH}
w(t,x)=&~ \frac{1}{T}\int_0^T\int_s^th(\sigma,x)\, d\sigma\, ds
\end{align}
from which we define $H(t,x) \equiv w(t,x)$ for use as a boundary condition below.
Thus, the solution $w$ of \eqref{weq} also satisfies the overdetermined system (with no reference to $w_t$ on $\Gamma$):
\begin{align}
	\label{weqP}
	\begin{aligned}
		w_{tt} -\Delta w &= g\text{ on }[0,T]\times \Omega_W,\quad w(0)=w(T)
		\\
		w&= H \text{ and }\partial_{\mathbf n} w= 0 \text{ on }[0,T]\times \Gamma,
		\\
		w &=0\text{ on }\typo{\Gamma_W}{[0, T]\times \Gamma_W}.
	\end{aligned}
\end{align}

For the above system we have the  estimate below. We note that the estimate is stated for additional regularity, as it is of interest in its own right. Subsequently, through our geometric assumptions, we will obtain the requisite regularity for constructed approximate solutions, leading to our final result for weak solutions.

\begin{lemma}\label{KeyEstimateLemma} 
Let $\Omega_W$ be a Lipschitz domain. Assume that $\flow\in C^{3,1}(\Omega_W)$, and take data $h\in L^2((0,T)\times\Gamma)$, $g\in L^2((0,T)\times\Omega_W)$, with $\mean{g}\equiv 0\equiv \mean{h}$.
 Then any solution $w$ to \eqref{weq}  with  $w\in L^2(0,T; H^1_{\Gamma_W}(\Omega_W)\cap H^{3/2^+}(\Omega_W))$ and  $\Delta w \in L^2(0,T;L^2(\Omega_W))$,  satisfies the identity
	\begin{align*}
		&\inttime \int_{\Omega_W} 
		(\grad w)^T \grad \flow \grad w\, 
		+ \inttime \int_{\partial\Omega_W} \left( \frac{1}{2}\; \abs{\grad w}^2 \; \flow\cdot \nrml - (\grad w\cdot \nrml) (\grad w\cdot\flow) \right) \, 
		\\
		&= \inttime \int_{\Omega_W} \left( g  (\nabla w\cdot \mathbf b) + \frac{1}{2}\; g w \diverg\flow  +\frac{\abs{w}^2}{4} \Delta\diverg \flow\right)\, 
				+  \inttime \intgamm \left[ \frac{\abs{h}^2}{2} \flow \cdot \nrml { -}\frac{\abs{H}^2}{4}\partial_\nrml \diverg\flow\right],
	\end{align*}
	recalling that $\partial_tH=h$. 
\end{lemma}

\begin{proof} The identity is obtained utilizing weak solutions and additional regularity; namely, with $g, \Delta w \in L^2(0,T;L^2(\Omega_W))$, we obtain from the equation that $w \in H^2(0,T; L^2(\Omega_W))$. Moreover, all traces appearing in the identity are well defined with $w \in H^{3/2+\epsilon}(\Omega_W)$. 

	We obtain an a priori estimate in two steps. First, we take the derivative of $w$ in the direction of vector field $\flow$ as a test function and, in the second step, we combine the obtained identity with an equipartition-type identity. 
	\begin{align}\label{multident1}
		\inttime \int_{\Omega_W} w_{tt} \pdb w \,  - \inttime \int_{\Omega_W} \Delta w\; \pdb w \, 
				= \inttime \int_{\Omega_W} g \pdb w \, 
	\end{align}
	We rewrite both terms on the left hand side.
	\begin{align}
		\inttime \int_{\Omega_W} w_{tt} \pdb w \, &=
		- \inttime \int_{\Omega_W} \frac{1}{2} \;\pdb \abs{{w}_t}^2\, 
		+ \left[\int_{\Omega_W} {w}_t \typo{\pdb w\, dx}{\pdb w}\right]_{t=0}^{t=T}
		\\
		&= \frac{1}{2} \inttime \int_{\Omega_W} \abs{{w}_t}^2 \diverg{\flow} \, 
		- \frac{1}{2} \inttime \intgamm \abs{h}^2 \;\flow \cdot \nrml\, 
	\end{align}
	Above, we have utilized the time regularity and
	 time periodicity of the solution to eliminate terms at $t=0,T$, homogeneous Dirichlet conditions on $\Gamma_W$, and we invoke the boundary condition
 $w_t = h$ on $\Gamma$.
	
	For the Laplacian term, we can approximate by smooth functions. We consider the calculations  as follows
	\begin{align}
		-\int_{\Omega_W} \Delta w\; \pdb w \,  &= \int_{\Omega_W} \grad w \cdot \grad \pdb w \,   - \intbndwave \pdb w \; \grad w \cdot \nrml \, 
		\\ &= \int_{\Omega_W} \grad w \cdot  \Big(\pdb \grad w + \grad \flow \grad w \Big)\,   - \intbndwave \pdb w \; \grad w \cdot \nrml\, 
		\\ &= \int_{\Omega_W} \frac{1}{2}\; \pdb \abs{\grad w}^2 +  (\grad w)^T \grad \flow \grad w \,  
		- \intbndwave  (\grad w)^T \nrml \flow^T \grad w \,  
	\end{align}
	and the first term can be simplified by the divergence theorem
	\begin{align}
		\int_{\Omega_W} \frac{1}{2}\; \pdb \abs{\grad w}^2 \,   = - \frac{1}{2} \int_{\Omega_W} \abs{\grad w}^2 \diverg\flow  \,  
		+ \frac{1}{2}\int_{\partial \Omega_W}  \abs{\grad w}^2 \; \flow\cdot \nrml\, 
	\end{align}
	Thus we obtain the identity:
	\begin{align}\nonumber
		-\int_{\Omega_W} \Delta w\; \pdb w \,  =&~- \frac{1}{2} \int_{\Omega_W} \big(\abs{\grad w}^2 \diverg\flow  \, +  (\grad w)^T \grad \flow \grad w \,\big) \\ &
		+ \frac{1}{2}\int_{\partial \Omega_W}  \big(\abs{\grad w}^2 \; \flow\cdot \nrml\  
		-  (\grad w)^T \nrml \flow^T \grad w\big)
		\end{align}
		We now note that this identity is then justified via density for any $w \in H^{3/2+\epsilon}(\Omega_W)\cap H^1_{\Gamma_W}(\Omega_W)$ with $\Delta w \in L^2(\Omega_W)$. 
	
So, from the above we obtain that \eqref{multident1} becomes: { 
	\begin{align}\begin{aligned}
		\frac{1}{2} \inttime \int_{\Omega_W} \abs{{w}_t}^2 \diverg{\flow}\,  
		- \frac{1}{2} \inttime \int_{\Omega_W}  \abs{\grad w}^2 \diverg\flow 
		&+\inttime \int_{\Omega_W} (\grad w)^T \grad \flow \grad w \,  
		\\
		+ \inttime \int_{\partial\Omega_W} \left( \frac{1}{2}\;  \abs{\grad w}^2 \; \flow\cdot \nrml - (\grad w)^T \nrml \flow^T \grad w \right)\, 
		&= \inttime \int_{\Omega_W} g \pdb w \,  
		+ \frac{1}{2} \inttime \intgamm \abs{h}^2 \;\flow \cdot \nrml. 
		\end{aligned}
		\label{test-by-flow-derivative}
	\end{align}}
This identity will be combined with the following equipartition-type identity, using $w \diverg \flow$ as a test function. This yields  
	\begin{align}
		\inttime \int_{\Omega_W}\left[ w_{tt} w \diverg \flow -  \Delta w\; w \diverg \flow  \, \right]
				= \inttime \intwave g  w \diverg \flow \,  .
	\end{align}
	We rewrite these terms as follows, starting with time derivative
	\begin{align}
		\inttime \intwave w_{tt} w \diverg \flow \,  =
		- \inttime \intwave \abs{w_t}^2 \diverg\flow \,  ,
	\end{align}
	where the $t=0,T$ terms vanish, thanks to periodicity. The second term becomes
	\begin{align}
		- \intwave \Delta w\; w \diverg \flow \,  
		&= \intwave \grad w \cdot \grad (w\diverg\flow ) \,   - \intbndwave \grad w \cdot \nrml\; w\diverg\flow \, 
		\\ 
		&= \intwave \abs{\grad w}^2 \;\diverg\flow \,  
		+ \intwave w  \grad w \cdot \grad \diverg\flow \,  ,
	\end{align}
	{  where we have invoked the boundary conditions from the overdetermined problem \eqref{weqP} to eliminate the boundary term on the right hand side above.}
The last term can be integrated by parts again, obtaining
    \begin{align}
        \intwave w  \grad w \cdot \underset{=:\; \textbf{d}}{\underbrace{\grad \diverg\flow}} \,   &=
		{ \frac{1}{2}\int_{\Omega_W}\nabla|w|^2\cdot \textbf{d}\, }
        \\ &=
        -\frac{1}{2} \intwave \abs{w}^2 \diverg\textbf{d}\,  
        + \frac{1}{2}\intbndwave \abs{w}^2 \;\textbf{d}\cdot\nrml\, 
        \\ &=
        -\frac{1}{2} \intwave \abs{w}^2 \diverg\textbf{d}\,  
        + \frac{1}{2}\int_\Gamma \abs{w}^2 \;\textbf{d}\cdot\nrml\,.
    \end{align}
  where we recall that $w = 0$ on $\Gamma_W$. 
	Hence we find
	\begin{align} \nonumber
		&- \inttime \intwave \abs{w_t}^2 \diverg\flow 
		+ \inttime\intwave \abs{\grad w}^2 \;\diverg\flow 
    -\frac{1}{2} \inttime\intwave \abs{w}^2 \Delta\diverg \flow
            + \frac{1}{2}\inttime\int_\Gamma \abs{w}^2 \; \partial_\nrml\diverg\flow\, 
            \\
		&= \inttime \intwave g  w \diverg \flow.\label{equipartition1}
	\end{align}
	Now, we multiply the above equation by $1/2$ and add it to the identity in \eqref{test-by-flow-derivative}. This leads to the desired identity, after invoking the given data in the problem.
	\begin{align}
	\begin{aligned}
	\label{eq:key}
		&\inttime \intwave 
		(\grad w)^T \grad \flow \grad w\, 
		+ \inttime \intbndwave \left( \frac{1}{2}\; \abs{\grad w}^2 \; \flow\cdot \nrml - (\grad w)^T \nrml \flow^T \grad w \right) \, 
		\\
		&= \inttime \intwave \left( g \pdb w + \frac{1}{2}\; g w \diverg\flow  +\frac{\abs{w}^2}{4} \Delta\diverg \flow\right)\, 
		+  \inttime \intgamm \left (\frac{\abs{h}^2}{2} \flow \cdot \nrml { -}\frac{\abs{H}^2}{4}\partial_\nrml \diverg\flow \right )\,.
		\end{aligned}
	\end{align}
		\end{proof}
	\begin{remark} The calculations above mirror the classical stabilization of the wave equation from the Neumann boundary condition, adapted to the framework here---see \cite{lt1,lt2}.\end{remark}

\begin{proposition}\label{prop:WaveEst}
	Let the assumptions of Lemma \ref{KeyEstimateLemma} be satisfied, and in addition, assume the integrated data has $H\in H^1((0,T)\times\Gamma)$. 
	\begin{enumerate}
		\item If $(\Omega_W,\Gamma_W)$ satisfies the Generalized Optics Condition \ref{def:GOC}, then the following a priori estimate holds for weak solutions to \eqref{weqP}:
		\begin{equation}\label{est1}
	\norm{w}_{W^{1,\infty}(0,T;L^2(\Omega_W))}+\norm{w}_{L^\infty(0,T;H^1(\Omega_W))}
	\lesssim \norm{g}_{L^2((0,T)\times\Omega_W)}+\norm{H}_{H^1((0,T)\times\Gamma)}
	\end{equation}
		\item If $(\Omega_W,\Gamma_W)$ satisfies the  Graph Optics Condition \ref{def:GGC}, then the following estimate holds:
		\begin{equation}\label{est2}
		\norm{w}_{L^2((0,T)\times \Omega_W))}+\norm{\partial_{\bf b} w}_{L^2((0,T)\times \Omega_W)}
		\lesssim \norm{g}_{L^2((0,T)\times \Omega_W)}+\norm{H}_{H^1((0,T)\times\Gamma)}.
		\end{equation}
	\end{enumerate}
	The constants associated to $\lesssim$ above depend on $\mathbf b$ and its norms.
	\end{proposition}
	
	\begin{remark}
	\label{rem:sharp}
	Let us show that Example~\ref{ex:1} and Example~\ref{ex:2} are sharp due to \eqref{est2}. Indeed, in the examples, $h=H=0$. As we are in the graph setting (compare with Condition~\ref{def:graph}) we find that one can choose $\flow=(y-2) e_2$. Hence \eqref{est2} becomes
	\begin{equation*}
		\norm{w}_{L^2((0,T)\times \Omega_W)}+\norm{\partial_{y} w}_{L^2((0,T)\times \Omega_W)}
		\lesssim \norm{g}_{L^2((0,T)\times \Omega_W)}
		\end{equation*}
		which is exactly what is satisfied by $\hat{w}$ constructed in Example~\ref{ex:1}.	
Now if $\partial_t g\in L^2((0,T)\times \Omega_W)$, one finds by time-differentiation and \eqref{est2} that $\partial_t w, \partial_t \partial_y w\in L^2((0,T)\times \Omega_W)$. Now by equipartition \eqref{equipartition1}, one finds that $\nabla w\in L^2(0,T;\mathbf L^2(\Omega_W))$ and hence $w$ will constitute a weak solution. This is precisely the regularity that was satisfied by $\hat{w}$ in Example~\ref{ex:2}.	But the estimate \eqref{est2} implies that in this setting no further differentiability in space for $g$ is needed to get a weak solution.
	\end{remark}

\begin{proof}[Proof of Proposition \ref{prop:WaveEst}] 
{ 
The proof is based on the identity in Lemma \ref{KeyEstimateLemma}. We first simplify the boundary term on the left-hand side of the identity.
Let $\nabla_{\Gamma}$ and $\nabla_{\Gamma_W}$ denote the tangential gradients over $\Gamma$ and $\Gamma_W$, respectively. 
Then we will show that}
\begin{align}\label{eq:step2} \inttime &\int_{\partial\Omega_W} \left( \frac{1}{2}\; \abs{\grad w}^2 \; \flow\cdot \nrml - (\grad w\cdot \nrml) (\grad w\cdot\flow) \right) 
\\
&\quad = \frac{1}{2}\inttime\int_{\Gamma}\abs{\nabla_{\Gamma} H}^2(\flow\cdot\nrml) 
-\frac{1}{2} \inttime\int_{\Gamma_W}|\partial_{\mathbf n} w|^2(\flow\cdot\nrml)
\end{align}
{ To prove the above identity, we note the following.}
First, from \eqref{weq} we have  $\partial_{\mathbf n} w=0$ on $\Gamma$ and $\nabla_{\Gamma_W}w=0$ on $\Gamma_W$. Therefore, writing $\nabla w = (\nabla w \cdot \mathbf n)\mathbf n + \sum_{j=1}^{d-1}(\nabla w \cdot \mathbf t_j) \mathbf t_j$, we obtain
$$
\int_{\partial\Omega_W}(\grad w\cdot \nrml) (\grad w\cdot\flow)
=\int_{\Gamma_W}|\partial_{\mathbf n} w|^2(\flow\cdot\nrml).
$$
Further
$$
\int_{\partial\Omega_W}\abs{\grad w}^2 \; \flow\cdot \nrml\, dx
=\int_{\Gamma_W}|\partial_{\mathbf n} w|^2(\flow\cdot\nrml)+\int_{\Gamma} \abs{\nabla_{\Gamma} w}^2 (\flow\cdot\nrml),
$$
which implies \eqref{eq:step2} by the given boundary conditions.

\noindent \textbf{Case 1:} We assume the {\em Generalized Optics Condition} \ref{def:GOC}. From the contractive domain hypothesis, \eqref{eq:contraction}, and the sign control in \eqref{GOCSign} we obtain
\begin{align*}
\|\nabla w\|^2_{L^2(0,T;\mathbf L^2(\Omega_W))}
&\lesssim \inttime \intwave (\grad w)^T \grad \flow \grad w
\\
& \lesssim \inttime \int_{\Omega_W} \left( g  (\nabla w\cdot \mathbf b) + \frac{1}{2}\; g w \diverg\flow  +\frac{\abs{w}^2}{4} \Delta\diverg \flow\right)\, 
\\
&\quad +  \inttime \intgamm \left[ \frac{\abs{h}^2}{2} \flow \cdot \nrml- \frac{\abs{\nabla_{\Gamma} H}^2}{2} \flow \cdot \nrml { -}\frac{\abs{H}^2}{4}\partial_\nrml \diverg\flow\right].
\end{align*}
The terms on right-hand side of the identity in Lemma \ref{KeyEstimateLemma} are then estimated using the Cauchy-Schwarz and Young inequalities.  Classical Poincar\'e---since $\gamma_0w =0$ on $\Gamma_W$---provides $L^2(0,T;L^2(\Omega_W))$ control from $\|\nabla w\|^2_{L^2(0,T;\mathbf L^2(\Omega_W))}$ control. 
Thus we have
\begin{equation}
\norm{w}_{L^2(0,T;H^1_{\Gamma_W}(\Omega_W))}
	\lesssim \norm{g}_{L^2((0,T)\times\Omega_W)}+\norm{H}_{H^1((0,T)\times\Gamma)}\end{equation}
Finally, to obtain control of $||w_t||^2_{L^2(0,T;L^2(\Omega_W))}$ we utilize \typo{a repeat a simple}{a simple} equipartition argument. We test the equation \eqref{weq} by the solution itself to obtain
$$\int_0^T\int_{\Omega_W} [w_{tt}-\Delta w]w = \int_0^T\int_{\Omega_W} gw.$$
Integrating by parts in space and time, as before,  invoking periodicity of $w$ and the double boundary conditions in \eqref{weq}, we obtain
$$\typo{\int_0^T\int_{\Omega_W}}{\int_0^T}||\nabla w||^2_{\mathbf L^2(\Omega_W)} - ||w_t||^2_{L^2(\Omega_W)} =  \int_0^T\int_{\Omega_W} gw.$$
And from this, we observe
$$||w_t||^2_{L^2(0,T;L^2(\Omega_W))} \lesssim ||g||^2_{L^2(0,T;L^2(\Omega_W))}+|| w||^2_{L^2(0,T;H^1_{\Gamma_W}(\Omega_W))},$$
and the estimate in \eqref{est1} follows. 

Finally, using the obtained estimates and \cite[Lemma II.4.1]{temam}, we derive the final $L^{\infty}$-in-time estimate.

\noindent \textbf{Case 2}: We now assume the  {\em Graph Optics Condition \ref{def:GGC}}. From the assumed Poincar\'e-type inequality in \eqref{GPoincare}
$$
\inttime\intwave|w|^2 \lesssim \left (\intwave (\grad w)^T \grad \flow \grad w-{ \inttime \int_{\Gamma_W}|\partial_n w|^2\flow\cdot\nrml}\right )+\|H\|^2_{L^2((0,T)\times\Gamma)}.
$$
From the { second} item in Condition \ref{def:GGC} we have that 
$$\int_{\Gamma_W} |\partial_{\mathbf b}w|^2 \lesssim \left (\intwave (\grad w)^T \grad \flow \grad w\right ).$$
Now, returning to the identity in Lemma \ref{KeyEstimateLemma}, we combine the previous two inequalities, and use the assumed sign control  in Condition \ref{def:GGC} on the term involving $\Delta \text{div}\mathbf b$. This allows us to estimate the remaining terms exactly as in the previous case, and  thus the claim follows.
\end{proof}

For applications to our construction, the following version of the above estimates will be  more suitable for use. This will be apparent in the next section, where periodic solutions are constructed. 
\begin{corollary}
\label{cor:key} 
Let $\Omega_W$ be a Lipschitz domain, assume that $\flow\in C^{3,1}(\Omega_W)$, and assume  $\Omega_W$ has the elliptic regularity property w.r.t. $\flow$ (Definition \ref{EllipticRegCond}).  Take data $h\in L^2((0,T)\times\Gamma)$, $g\in L^2((0,T)\times\Omega_W)$, with $\mean{g}\equiv 0\equiv \mean{h}$. 
Finally, assume that
	\[
\flow\cdot \nrml\geq 0\text{ on }\Gamma.
\]
Then a given finite energy solution satisfying the overdetermined problem \eqref{weq} satisfies:	
\begin{enumerate}
	\item If $(\Omega_W,\Gamma_W)$ satisfies the Generalized Optics Condition \ref{def:GOC} then:
		\[
	\norm{w}_{W^{1,\infty}(0,T;L^2(\Omega_W))}+\norm{w}_{L^\infty(0,T;H^1(\Omega_W))}
	\lesssim \norm{g}_{L^2((0,T)\times\Omega_W)}+\norm{H}_{H^{1}(0,T;L^2(\Omega_W))}
	\]
	\item If $(\Omega_W,\Gamma_W)$ satisfies the  Graph Optics Condition \ref{def:GGC} then:
	\[
	\norm{w}_{L^2((0,T)\times \Omega_W))}+\norm{\partial_{\bf b} w}_{L^2((0,T)\times \Omega_W)}
	\lesssim \norm{g}_{L^2((0,T)\times \Omega_W)}+\norm{H}_{H^{1}(0,T;L^2(\Omega_W))},
	\]
	{  where $H$ is defined by \eqref{PrimitiveH}}.
\end{enumerate}
\end{corollary}
\begin{proof}
{  We start by mollifying the equation in time as was done in the proof of Theorem~\ref{th:uniqueness}. This defines $w_\rho$ a smooth in time solution, with boundary values $H_\rho$ on $\Gamma$ and right hand side $g_\rho$. For this solution we show that we can apply Lemma~\ref{KeyEstimateLemma}. The assumed elliptic regularity property of $\Omega_W$ w.r.t.\ the vector field ${\bf b}$ can be applied as $-\Delta w_\rho=g_\rho-\partial_t^2w_\rho\in L^2([0,T]\times \Omega_W)$, $w_\rho=0\text{ on }\Gamma_W$ and $\partial_{\bf n} w_\rho=0$ a. e.\ on $\Gamma$. Hence $w_\rho\in L^2(0,T;H^{\frac{3}{2}+}(\Omega_W))$ and in particular $H_\rho=w_\rho\in L^2(0,T;H^1(\Gamma))$, by the trace theorem. Now we can apply Lemma~\ref{KeyEstimateLemma} and find by \eqref{eq:key} and \eqref{eq:step2} that
\begin{align*}
		&\inttime \intwave 
		(\grad w_\rho)^T \grad \flow \grad w_\rho\, 
		+ \frac{1}{2}\int_{\Gamma}\abs{\nabla_{\Gamma} H_\rho}^2(\flow\cdot\nrml) 
-\frac{1}{2} \int_{\Gamma_W}|\partial_{\mathbf n} w_\rho|^2(\flow\cdot\nrml)\, 
		\\
		&= \inttime \intwave \left( g_\rho \pdb w_\rho + \frac{1}{2}\; g_\rho w_\rho \diverg\flow  +\frac{\abs{w_\rho}^2}{4} \Delta\diverg \flow\right)\, 
		+  \inttime \intgamm \left (\frac{\abs{h_\rho}^2}{2} \flow \cdot \nrml -\frac{\abs{H_\rho}^2}{4}\partial_\nrml \diverg\flow\right ),
	\end{align*}
	hence the assumed condition $\flow\cdot {\bf n}\geq 0$ on $\Gamma$ allows to complete the estimate independent of $\rho$ in the wished form by the same arguments as in the proof of Proposition~\ref{prop:WaveEst}.
}
%
\end{proof}

\section{Construction of Solutions}\label{sec:construction}

In this section, we complete the proofs of Theorems \ref{th:main1} and \ref{th:main2} by constructing approximate solutions. Subsequently we employ the hyperbolic estimates derived in the previous section in order to pass to the limit. However, the approximate solutions do not satisfy the assumptions of Corollary \ref{cor:key}. Specifically, there are two issues:
\begin{enumerate}
	\item The approximate solutions do not have zero Neumann boundary values on the interface $\Gamma$. We address this issue by constructing suitable extensions. 
	\item Secondly, the hyperbolic estimates in Corollary \ref{cor:key} require a level of regularity that our weak solutions lack.
 This issue is resolved by employing elliptic regularity for a mixed boundary value problem associated with the Laplacian. This approach suffices because periodicity allows us to derive uniform estimates for higher-order time derivatives.
\end{enumerate}

The approximate solutions are constructed using an abstract semigroup framework. To apply a general abstract result for constructing approximate periodic solutions, it is necessary to ``translate" the problem by introducing artificial dissipation into each of the constitute equations (in a first order formulation of the heat-wave system).

In the next step, we extend the results to less regular domains by approximating general domains with ones that satisfy the elliptic regularity property. It is important to emphasize that the final uniform estimates depend solely on spatial gradients and, therefore, do not require any elliptic regularity. Nonetheless, elliptic regularity theory plays a critical role in the intermediate steps leading to the final estimates, primarily for technical reasons in dealing with boundary triple points.

To construct periodic approximate solutions, we rely on the following abstract theorem:
\begin{theorem}\cite[Proposition 3.4.]{bostan}\label{thm:bostan}	
	Let $H$ be a Hilbert space and suppose that $A: \mathcal D(A) \subset H \rightarrow H$ is a maximal monotone operator. Suppose that $f \in C^1_{\sharp}(\mathbb{R} ; H)$ is $T$-periodic. Then for every $\epsilon>0$ the equation
$$
x^{\prime}(t)+A x(t)+\epsilon x(t)=f(t), \quad t \in \mathbb{R},
$$
has a unique $T$-periodic smooth solution in $C_{\sharp}^1(\mathbb{R} ; H) \cap C_{\sharp}(\mathbb{R} ; \mathcal D(A))$.
\end{theorem}

The solution above is smooth, and thus satisfies the abstract, translated Cauchy problem pointwise in time. However, it also satisfies the  classical variation of parameters formula. Namely, when $A$ is linear, let $S(t) \in \mathscr L(H)$ be the semigroup associated to the generator, $-A$. Then the following holds for all $x_0 \in \mathcal D(A)$ and $f$ as above:
$$x(t) = S(t)x_0+\int_0^tS(t-s)f(s)ds.$$
This can be extended by density to hold (see \cite{galdi,pazy}) for $x_0 \in H$ and $f \in C([0,T];H)$. 

{  
\begin{corollary}
	\label{cor:bostan}
	Assume that hypotheses of Theorem \ref{thm:bostan} are satisfied, and additionally $f\in C^k_{\sharp}(\mathbb{R} ; H)$ for some $k\in\N$. Then the unique $T$-periodic solution $x$ to the equation in Theorem \ref{thm:bostan} satisfies $x\in C_{\sharp}^k(0,T ; H) \cap C_{\sharp}^{k-1}(0,T ; \mathcal D(A))$.
\end{corollary}
\begin{proof}
	We prove the claim by induction. The case $k=1$ is Theorem \ref{thm:bostan}. Assume the claim holds for $k=n\geq 1$. Then $x\in C_{\sharp}^n(0,T ; H) \cap C_{\sharp}^{n-1}(0,T; \mathcal D(A))$. Differentiating the equation in Theorem \ref{thm:bostan}, we obtain
	$$
	x^{(n+1)}(t)+A x^{(n)}(t)+\epsilon x^{(n)}(t)=f^{(n)}(t), \quad t \in \mathbb{R},
	$$
	where $f^{(n)}\in C_{\sharp}^1(0,T ; H)$. By the induction hypothesis, we have $x^{(n)}\in C_{\sharp}^1(0,T ; H) \cap C_{\sharp}(0,T; \mathcal D(A))$. Thus, by Theorem \ref{thm:bostan}, we conclude that $x^{(n)}\in C_{\sharp}^1(0,T ; H) \cap C_{\sharp}(0,T ; \mathcal D(A))$, which implies that $x\in C_{\sharp}^{n+1}(0,T ; H) \cap C_{\sharp}^{n}(0,T ; \mathcal D(A))$. This completes the induction and the proof.
\end{proof}
}

Let $\epsilon>0$ be a (small) parameter. In our case, we will consider the following translated problem, in line with the abstract theorem above:
\begin{align}
	\label{eq:solid*}
	\ddot{w}-\Delta w+2\epsilon \dot{w}+\epsilon^2 w&=g\;{\rm in}\; (0,T)\times\Omega_W,\\
	\label{eq:fluid*}
	\dot{u}- \Delta u+\epsilon u&=f\;{\rm in}\; (0,T)\times\Omega_H,
	\\
	w_t&=u\;{\rm on}\; (0,T)\times \Gamma,
	\label{bcde}
	\\
	\
	\partial_{\mathbf n} u&= \partial_{\mathbf n} w\;{\rm on}\; (0,T)\times \Gamma,\\
	u & \equiv 0\;{\rm on}\; (0,T)\times\Gamma_{H},\\
	w & \equiv 0\;{\rm on}\; (0,T)\times\Gamma_{W},\\
	\label{eq:periodic*}
	u(0)=u(T), & \;w(0)=w(T),\; w_t(0)=w_t(T). 
\end{align}

The semigroup formulation of the heat-wave system with $\epsilon=0$---and a proof of generation---can be found in \cite{trig1,ZZ1,dutch}. Here we present the description of the generator for the heat-wave dynamics (taking $\epsilon =0$ above in \eqref{eq:solid*}--\eqref{eq:periodic*}).
For the  dynamics, we consider the reduced first order system in the variable $y$ on the state space $Y$:
\begin{equation} y = (w, v, u) \in Y=H^1_{\Gamma_W}(\Omega_W) \times L^2(\Omega_W) \times L^2(\Omega_H) \typo{where.}{.}\end{equation}
 We topologize both $H^1_{\Gamma_W}(\Omega_W)$ and $H_{\Gamma_H}^1(\Omega_H)$ via the gradient norm, using Poincar\'e's inequality.
Then the dynamics operator $\mathbf A$ has action given by:
\begin{equation}
\mathbf A y = \begin{bmatrix} 0 & -I & 0 \\ -\Delta  & 0 & 0 \\ 0 & 0 & -\Delta \end{bmatrix} \begin{pmatrix}w \\v \\ u \end{pmatrix} \in Y, \end{equation}
for $y \in \mathcal D(\mathbf A)$ with the domain: 
\begin{equation}
\mathcal D(\mathbf A) = \left\{ y \in H^1_{\Gamma_W}(\Omega_W) \times H^1_{\Gamma_W}(\Omega_W) \times H_{\Gamma_H}^1(\Omega_H)~:~(1.)\text{--}(4.) ~~\text{hold below}\right\}:
\end{equation}
 $${\footnotesize(1.)~\Delta w \in L^2(\Omega_W), ~(2.)~\Delta u \in L^2(\Omega_H), ~(3.)~\gamma_0(w)=\gamma_0(u) ~\text{ on }~\Gamma,~(4.)~\gamma_1(u)=-\gamma_1(w)~ \text{ on }~\Gamma.}$$
We remark that, in (3.), the velocity matching condition is viable in $H^{1/2}(\Gamma)$, owing to classical trace theory. In (4.), we interpret the Neumann traces in the sense of $H^{-1/2}(\Gamma)$. This is a globally defined object on $\partial \Omega_W\cup \partial \Omega_H$ which can then be restricted to boundary subdomains \cite{grisvard,dauge,aubin}. 
\begin{remark} Defining $\mathcal D(\mathbf A)$ in the above way admits the (only) Lipschitz domains and gives meaning to traces for the generator without appealing to elliptic regularity. \end{remark}

The above linear operator $\mathbf A$ is maximal monotone on $Y$, and thus  $-\mathbf A$ generates a $C_0$-semigroup of contractions on $Y$ \cite{pazy}. As such we may invoke Theorem \ref{thm:bostan}. By direct application thereof, we obtain the existence of a {\em strong} periodic solution to \eqref{eq:fluid*}--\eqref{eq:periodic*}, viewing this as an $\epsilon$-translation of $\mathbf A$. Namely, as we see in the corollary below, the solution is smooth in time $C^1(0,T;Y)$ and properly periodic; moreover, it has spatial regularity in the sense of $C([0,T]; \mathcal D(\mathbf A))$

\noindent
Thus we have the existence and uniqueness of a smooth in time approximate solution. 
\begin{corollary}
	For every $\eps>0$ and $(f,g)\in C_{\sharp}^{ {k}}([0,T];L^2(\Omega_H)\times L^2(\Omega_W))$, ${ k\in\N}$ there exists a unique smooth solution to problem \eqref{eq:solid*}--\eqref{eq:periodic*}. In particular, the solution $(w,w_t,u) \in C_{\sharp} ^{{ k-1}}\left([0,T];H^1_{\Gamma_W}(\Omega_W) \times L^2(\Omega_W) \times L^2(\Omega_H)\right)$.
\end{corollary}
\noindent In the above result, we consider the semigroup framework for the dynamics $(w,w_t,u)$ in the heat-wave system, as in \cite{ZZ2}. We consider data 
\[[0,g,f]^T \in C^1\left([0,T];H^1_{\Gamma_W}(\Omega_W) \times L^2(\Omega_W) \times L^2(\Omega_H)\right).\]

Now we turn to weak solutions. It is immediate that smooth solutions in the sense of the above corollary are weak, in the sense defined below.
{ \begin{definition}\label{DefWeakApp}
We say that
$$(u^{\eps},w^{\eps})\in L^2(0,T;H^1_{\Gamma_H}(\Omega_H))\times \left [L^2(0,T;H^1_{\Gamma_W}(\Omega_W))\cap H^1_{\sharp}(0,T;L^2(\Omega_W))\right]$$
is a weak solution to the translated problem \eqref{eq:solid*}--\eqref{eq:periodic*} iff \eqref{bcde} holds in the sense of \eqref{bcdw} and the following equality is satisfied for every pair of test functions $(\phi,\psi) \in C_{\sharp}^1([0,T];H^1_{\Gamma_H} \times H^1_{\Gamma_W})$ such that $\phi|_{\Gamma}=\psi|_{\Gamma}$, ~$\phi|_{\Gamma_H}=0, ~\psi|_{\Gamma_W}=0$:
\begin{align}\label{WeakReg}
	-&\int_0^T\int_{\Omega_H}u^{\eps}\partial_t\phi+\int_0^T\int_{\Omega_H}\nabla u^{\eps}\cdot\nabla\phi
	+\eps\int_0^T\int_{\Omega_H}u^{\eps}\phi - \int_0^T\int_{\Omega_W}\partial_t w^{\eps}\partial_t\psi
	\\\nonumber
	+&\int_0^T\int_{\Omega_W}\nabla w^{\eps}\cdot\nabla\psi
	+2\eps\int_0^T\int_{\Omega_W}\partial_t w^{\eps}\psi+\eps^2\int_0^T\int_{\Omega_W}w^{\eps}\psi
	=\int_0^T\langle f,\phi\rangle+\int_0^T\int_{\Omega_H}g\psi.
\end{align}
\end{definition}}

Such weak solutions automatically  satisfy the following formal energy estimates stated below.
\begin{lemma}\label{EnergyRegularized}
	Let $(u^{\eps},w^{\eps})$ be a smooth solution to the translated problem, \eqref{eq:solid*}--\eqref{eq:periodic*}. Then the solution satisfies the following energy inequalities for $k\in\N_0$:
	\begin{align*}
		\|\nabla \partial_t^k u^{\eps}\|^2_{L^2(0,T;\mathbf L^2(\Omega_H))}
		+\eps\|\partial_t^ku^{\eps}\|^2_{L^2(0,T;L^2(\Omega_H))}
		+\eps\|\partial^{k+1}_t w^{\eps}\|^2_{L^2(0,T;L^2(\Omega_W))}
		\\
		\leq
		C\left (\|\partial_t^k f\|^2_{L^2(0,T;H_{\Gamma_{H}}^1(\Omega_H)')}
		+\frac{1}{\eps}\|\partial_t^k g\|^2_{L^2(0,T;L^2(\Omega_W))}\right ), 
	\end{align*}
where the constant $C$ is independent of $\eps$.
\end{lemma}
\begin{proof}
	The proof is standard, so we omit the details. One may differentiate the equation with respect to the time variable and use the test function $(\partial_t^k u^{\eps},\partial_t^{k+1}w^{\eps})$. Note that the boundary conditions ensure that the Poincar\'e's inequality hold, which may { be used} to estimate $L^2$ quantities on the RHS.
\end{proof}

Note that Lemma \ref{EnergyRegularized} cannot be used when passing to the limit as $\eps\to 0$. It is immediately clear that the estimate degenerates for the wave component in $\epsilon$, but it also degenerates for the heat component. This is due to the fact that Gronwall's inequality must be applied to derive the energy estimate for the coupled system. Therefore, we turn to the analysis from Section \ref{Sec:Hyp} to obtain uniform estimates independent of the parameter $\eps$. 

\begin{remark} 
We note that a different approach can yield the existence of a weak solution for the approximate problem. Namely:
for every $\eps>0$ and $f\in L^2(0,T;L^2(\Omega_H))$, $g\in L^2(0,T;L^2(\Omega_W))$, there exists a unique weak solution $(u^{\eps},w^{\eps})$ to problem \eqref{eq:solid*}--\eqref{eq:periodic*} in sense of Definition \ref{DefWeakApp}. 
This result is obtained through regularization of the data $(f,g) \in L^2(0,T;L^2(\Omega_H) \times L^2(\Omega_W))$ by $(f_n, g_n) \in C_0^1([0,T];H^1_{\Gamma_H}(\Omega_H) \times H^1_{\Gamma_W}(\Omega_W))$ and limit passage. By invoking Theorem \ref{thm:bostan}, we obtain a smooth solution which then reconstructs the initial data $(w^n(0),w_t^n(0),u^n(0)) \in \mathcal D(\mathbf A) \subset H^1_{\Gamma_W}(\Omega_W) \times L^2(\Omega_W) \times L^2(\Omega_H)$ (with associated estimates). Extending by density (in the variable $n$) yields a weak solution. This demonstrates that with $(f,g) \in L^2(0,T;L^2(\Omega_H) \times L^2(\Omega_W))$ and $\epsilon>0$, the weak solution as in Definition \ref{DefWeakApp} demonstrates no loss of (spatial) regularity from that of the associated Cauchy problem.
\end{remark}

\subsection{Uniform Estimates}\label{unifest}
The key tool for obtaining uniform estimates is Corollary \ref{cor:key}. However, to apply this result, several preparatory steps are required. The first step is to derive uniform estimates for the mean values $\mean{u^{\eps}}$ and $\mean{w^{\eps}}$. 

\begin{proposition}[Mean Value Estimates]\label{MeanValueEst}
	Let $k\in \N_0$, $f\in H^k(0,T;H_{\Gamma_{H}}^1(\Omega_H)')$ and $g\in H^k(0,T;L^2(\Omega_W))$, and let $(u^{\eps},w^{\eps})$ be a solution to the translated  problem \eqref{eq:solid*}--\eqref{eq:periodic*} in the sense of Theorem \ref{thm:bostan}. Then the following estimates hold
	\begin{align}\label{MeanHeatEst}
		\|\mean{\partial_t^k u^{\eps}}\|_{H^1(\Omega_H)}\lesssim & ~\|\mean{\partial_t^k f}\|_{H_{\Gamma_{H}}^1(\Omega_H)'},
		\\
		\label{MeanWaveEst}
		\|\mean{\partial_t^k w^{\eps}}\|_{H^1(\Omega_W)}\lesssim &~
		\|\mean{\partial_t^k f}\|_{H_{\Gamma_{H}}^1(\Omega_H)'},
		+\|\mean{\partial_t^k g}\|_{L^2(\Omega_W)}.
	\end{align}
\end{proposition}
\begin{proof}
	Analogous to the derivation of \eqref{MeanValueHeatEq}, we find that $\mean{u^{\eps}}$
	satisfies the following decoupled equation:
	$$
	-\Delta \mean{u^{\eps}}+\eps\mean{u^{\eps}}=\mean{f}\; \text{ in } \Omega_H,~
	\quad 
	\mean{u^{\eps}}=0\;{\rm on}\;\Gamma.
	$$
	Therefore,  standard elliptic theory (Lax-Millgram) yields \eqref{MeanHeatEst} for $k = 0$. To obtain estimates for $k > 0$, we differentiate the equation in time. We  note that time differentiation commutes with taking the mean value, and we may invoke the Poincar\'e inequality. As before in Section \ref{Sec:invariances}, this yields the following estimate for the normal trace $\gamma_1[\partial_t^ku^{\epsilon}]$:
	\begin{align*}
		\|\partial_{\mathbf n}\mean{\partial_t^ku^{\eps}}\|_{H^{-1/2}(\Gamma)}
		\lesssim \left (\|\nabla \mean{\partial_t^ku^{\eps}}\|_{L^2(\Omega_H)}
		+\|\Delta \mean{\partial_t^ku^{\eps}}\|_{H_{\Gamma_{H}}^1(\Omega_H)'}\right )
		\lesssim \|\mean{\partial_t^k f}\|_{H_{\Gamma_{H}}^1(\Omega_H)'}.
	\end{align*}
	Now, the estimate in \eqref{MeanWaveEst} again follows from standard elliptic theory. We recall  that $w^{\eps}$ satisfy the following boundary value problem:
	$$
	-\Delta \mean{w^{\eps}}+\eps^2\mean{w^{\eps}}=\mean{g}~\text{ in }~\Omega_W,
	\quad 
	\partial_{\mathbf n} \mean{w^{\eps}}=\partial_{\mathbf n} \mean{u^{\eps}} \text{ on} \;\Gamma,
	\; \mean{w^{\eps}}=0\;~\text{ on } \;\Gamma_W.
	$$
	The estimates then follow exactly as  in the previous step.
\end{proof}

\begin{remark}
Here we used the following version of the standard result on existence of normal traces in $H^{-\frac{1}{2}}$. Let ${\bf v}\in \mathbf L^2(\Omega_H)$ such that ${\rm div}~{\bf v}\in (H^1_{\Gamma_H}(\Omega))'$. Then ${\bf v}\cdot\nrml\in H^{-\frac{1}{2}}(\Gamma)$ and the following estimate holds:
$$
\| {\bf v}\cdot\nrml\|_{H^{-\frac{1}{2}}(\Gamma)}
\leq C\left (\|{\bf v}\|_{L^2(\Omega_H)}+\|{\rm div}{\bf v}\|_{(H^1_{\Gamma_H}(\Omega_H))'}\right ).
$$
The proof is analogous to the standard proof. Namely, for smooth $v$ and $\phi$ using integration by parts we have:
$$
\int_{\Gamma} ({\bf v}\cdot\nrml) \phi=\int_{\Omega_H} {\bf v}\cdot\nabla\phi+\int_{\Omega_H}({\rm div} {\bf v})\phi.
$$
Now, let $E:H^{\frac{1}{2}}(\Gamma)\to H^1_{\Gamma_H}(\Omega_H)$ be an extension operator. We define the normal trace as follows:
$$
{}_{H^{-\frac{1}{2}}(\Gamma)}\langle {\bf v}\cdot\nrml,\phi\rangle_{H^{\frac{1}{2}}(\Gamma)}
\equiv\int_{\Omega_H}{\bf v}\cdot\nabla E\phi
+\langle{\rm div}{\bf v},E\phi\rangle_{H^{1}_{\Gamma_H}(\Omega_H)'\times H^{1}_{\Gamma_H}(\Omega_H)}.
$$
\end{remark}

In the next step we need to estimate the mean-free part $\widetilde{w}^{\eps}=w^{\eps}-\mean{w^{\eps}}$. Since the overdetermined problem studied in Section \ref{Sec:Hyp} satisfies a homogeneous Neumann condition---which is not the case for $\tilde{w}$---we construct an extension to homogenize the Neumann boundary conditions.

\begin{lemma}[On Harmonic Extension]\label{ExtnesionLemma}
	There { exists} a $T$-periodic harmonic extension $e^{\eps}$ of $\partial_{\mathbf n} \widetilde{u^{\eps}}=\partial_{\mathbf n} u^{\eps}-\partial_{\mathbf n}\mean{u^{\eps}} $ to $\Omega_W$ in the following sense: 
	\begin{itemize} \item $e^{\eps}(t)\in H_{\Gamma_W}^1(\Omega_W)$, 
	\item  for every pair of $T$-periodic test functions $(\psi,\phi)$ such that $\psi|_{\Gamma_{int}}=\phi|_{\Gamma_{int}}$, $\phi|_{\Gamma_H}=\psi|_{\Gamma_W}=0$, the following equality is satisfied:
	\begin{align}\nonumber
		\int_0^T\int_{\Omega_W}\nabla e^{\eps}\nabla\phi&=	-\int_0^T\int_{\Omega_H}\partial_t\widetilde{u}^{\eps}\psi-\int_0^T\int_{\Omega_H}\nabla \widetilde{u}^{\eps}\cdot\nabla\psi
		\\\label{HarmonicExtWeak}
		&\quad -\eps\int_0^T\int_{\Omega_H}\widetilde{u}^{\eps}\psi
		+\int_0^T\int_{\Omega_H}\tilde{f}\psi,
	\end{align}
where $\widetilde{f}=f-\mean{f}$.
\end{itemize}
	Moreover the following estimate is satisfied for $k\in\N_0$:
	\begin{align}\label{harmonicestimate}
		\|\partial_t^k \nabla e^{\eps}\|_{L^2(0,T;L^2(\Omega_W))}\leq &~	C\Big (
		\|\partial^{k+1}_t \widetilde{u}^{\eps}\|_{L^2(0,T;H_{\Gamma_H}^1(\Omega_H)')}+\|\partial_t^k\nabla \widetilde{u}^{\eps}\|_{L^2(0,T;L^2(\Omega_H))} \\ \nonumber
& 		+\|\partial_t^k \tilde{f}\|_{L^2(0,T;H_{\Gamma_H}^1(\Omega_H)')}\Big),
	\end{align}
	where the constant $C$ {\em does not depend on $\eps$}.
\end{lemma}
\begin{proof}
	Here we present formal estimates, since the harmonic extension can be constructed in a standard way. Moreover, since we consider a stationary problem, let us fix $t\in [0,T]$ and define $F(t)$ as
	$$
	\langle F(t),\phi \rangle \equiv  -\int_{\Omega_H}\partial_t\widetilde{u}^{\eps}(t)\psi-\int_{\Omega_H}\nabla \widetilde{u}^{\eps}(t)\cdot\nabla\psi
	-\eps\int_{\Omega_H}\widetilde{u}^{\eps}(t)\psi
	+\int_{\Omega_H}\widetilde{f}(t)\psi,
	$$
	where $\phi\in H_{\Gamma_W}^1(\Omega_W)$ and $\psi=E_{\Omega_H}(\phi|_{\Gamma})$ for some $E_{\Omega_H}$  an extension operator (e.g., harmonic) from $\Gamma$ to $\Omega_H$,  {  which has zero boundary values on $\Gamma_H$}.
	
	Then the following estimate holds:
	\begin{align*}
	|\langle F(t),\phi \rangle |\lesssim & ~
	\left (\|\partial_t \widetilde{u}^{\eps}\|_{H_{\Gamma_H}^1(\Omega_H)'}+\|\nabla \widetilde{u}^{\eps}\|_{\mathbf L^2(\Omega_H)}
	+\eps\|\widetilde{u}^{\eps}\|_{L^2(\Omega_H)}+\|\widetilde{f}\|_{H_{\Gamma_H}^1(\Omega_H)'}
	\right ) \|\psi\|_{H^1(\Omega_H)}\\
	\lesssim & ~\left (\|\partial_t \widetilde{u}^{\eps}\|_{(H^1(\Omega_H)_{\Gamma_H})'}+\|\nabla \widetilde{u}^{\eps}\|_{\mathbf L^2(\Omega_H)}
	+\eps\|\widetilde{u}^{\eps}\|_{L^2(\Omega_H)}+\|\widetilde{f}\|_{H_{\Gamma_H}^1(\Omega_H)'}
	\right ) \|\phi\|_{H^1(\Omega_W)}.
	\end{align*}
	Therefore $F(t)\in H_{\Gamma_W}^1(\Omega_W)'$ and hence the weak problem 
	$$
	\int_{\Omega_W}\nabla e^{\eps}\nabla\phi=\langle F(t),\phi \rangle ,\quad \phi\in H_{\Gamma_W}^1(\Omega_W)
	$$
	has a unique solution. This defines the harmonic extension $e^{\eps}$, as desired. 
	
	Finally notice that the above problem can be solved for a.e. $t\in (0,T)$ and the right hand side may be $T$-periodic, and therefore $e^{\eps}$ is $T$-periodic in the same sense. The estimate for $e^{\eps}$ is obtained by utilizing the test function $\phi=e^{\eps}$ and integrating $\int_0^T$. Higher order estimates are obtained by differentiating the whole system in the time variable.
\end{proof}

Now, we are ready to produce the final estimate for the mean-free part of the approximate solution $(u^{\epsilon},w^{\epsilon})$. Recall that $\widetilde f= f -\langle f \rangle.$
\begin{theorem}\label{UniformEstamtesTm}
	Let $\widetilde{f}\in H^4_{\sharp}(0,T;L^2(\Omega_H))$, $\widetilde{g}\in H^8_\sharp(0,T;L^2(\Omega_W))$, and let $(u^{\eps},w^{\eps})$ be a smooth solution to the regularized problem \eqref{eq:solid*}--\eqref{eq:periodic*}, as established in the previous sections. Additionally, we assume that $\Omega_W$ has the elliptic regularity property w.r.t. $\flow.$
	\begin{enumerate}
		\item If $(\Omega_W,\Gamma_W)$ satisfies the Generalized Optics Condition \ref{def:GOC}, then the following estimate for the mean-free part of the approximate solution holds:
		\begin{align}\nonumber
			\|\nabla \widetilde{w}^{\eps}\|_{L^2(0,T;L^2(\Omega_W))}+\|\partial_t \widetilde{w}^{\eps}\|_{L^2(0,T;L^2(\Omega_W))}
			+\|\nabla \widetilde{u}^{\eps}\|^2_{H^3(0,T;L^2(\Omega_H))}
			\\\label{UnifromEstimatesStrong}
			\leq C\left (
			\|\widetilde{f}\|^2_{H^3(0,T;L^2(\Omega_H))}
			+\|\widetilde{g}\|^2_{H^6(0,T;L^2(\Omega_W))}
			\right ).	
		\end{align}
	\item If $(\Omega_W,\Gamma_W)$ satisfies the Graph Optics Condition \ref{def:GGC}, then the following estimate for the mean-free part of the approximate solution holds:
		\begin{align}\nonumber
		\|\nabla \widetilde{w}^{\eps}\|_{L^2(0,T;L^2(\Omega_W))}+\|\partial_t \widetilde{w}^{\eps}\|_{L^2(0,T;L^2(\Omega_W))}
		+\|\nabla \widetilde{u}^{\eps}\|^2_{H^4(0,T;L^2(\Omega_H))}
		\\\label{UnifromEstimates}
		\leq C\left (
		\|\widetilde{f}\|^2_{H^4(0,T;L^2(\Omega_H))}
		+\|\widetilde{g}\|^2_{H^8(0,T;L^2(\Omega_W))}
		\right ).	
	\end{align}
	\end{enumerate}
\end{theorem}
\begin{proof}
{\bf Step 1: Correction.} The first issue is that we cannot directly make use of Corollary \ref{cor:key}, as we do not have $\flow\cdot{\bf n}\geq 0$ on $\Gamma$. Therefore we define $\widetilde{\Gamma}_W=\{x\in\partial\Omega_{ W}:\flow\cdot{\bf n}\leq 0\}$ and $\widetilde{\Gamma}=\partial\Omega_{ W}\setminus\widetilde{\Gamma}_W$. From the assumed geometric condition  $\mathbf b \cdot \mathbf n \le 0 $ on $\Gamma_W$ in both Condition \ref{def:GOC} and Condition \ref{def:GGC}, it is immediate that $\widetilde{\Gamma}\subset\Gamma$. From \eqref{PrimitiveH} we may define
\begin{align}
	{ \widetilde{w}}^{\eps}(t,x)= \frac{1}{T}\int_0^T\int_s^t \widetilde{u}^{\eps} (\sigma,x)\, d\sigma\, ds \equiv H^{\eps}(t,x), \; x\in\Gamma.
\end{align}
Let $\beta^{\eps}(t,.)\in H^1(\Omega_W)$ be the solution to
$$
-\Delta\beta^{\eps}=0\;{\rm in}\;\Omega_W,\quad
\beta^{\eps}=0\;{\rm on}\;\Gamma_W,\quad
\beta^{\eps}=H^{\eps}(t,.)-e^{\eps}(t,.)\;{\rm on}\;\Gamma\setminus\widetilde{\Gamma},\quad
\partial_{\mathbf n} \beta^{\eps}=0\;{\rm on}\;\widetilde{\Gamma},
$$
{ 
where $e^{\eps}$ is the extension defined in Lemma~\ref{ExtnesionLemma} using $\tilde{\Gamma}$ there.
Let us derive the estimate for the corrector $\beta^{\eps}$. First, notice that $\beta^{\eps}$ is the solution of the stationary problem and the only time dependency comes from the boundary data. Therefore by \eqref{harmonicestimate}, for $k\in\N_0$, we have the following estimate:
\begin{align}\nonumber
&\|\partial_t^k\nabla \beta^{\eps}\|_{L^2(0,T;L^2(\Omega_W))}
\lesssim\|\partial_t^k(H^{\eps}-e^{\eps})\|_{L^2(0,T;H^{1/2}(\Gamma\setminus\widetilde{\Gamma}))}
\\\label{estimatebeta}
&\quad \lesssim\left(\|\partial_t^k\nabla \widetilde{u}^{\eps}\|_{L^2(0,T;L^2(\Omega_H))}
+\|\partial_t^k\nabla e^{\eps}\|_{L^2(0,T;L^2(\Omega_W))}\right)
\\
&\quad \lesssim \Big (\|\partial^{k+1}_t \widetilde{u}^{\eps}\|_{L^2(0,T;H_{\Gamma_H}^1(\Omega_H)')}+\|\partial_t^k\nabla \widetilde{u}^{\eps}\|_{L^2(0,T;L^2(\Omega_H))} +\|\partial_t^k \tilde{f}\|_{L^2(0,T;H_{\Gamma_H}^1(\Omega_H)')}\Big)
\nonumber
\end{align}
}
Next $\overline{w}^{\eps}\equiv\widetilde{w}^{\eps}-e^{\eps}-\beta^{\eps}$ satisfy the following over-determined problem for the wave equation:
\begin{align}
	\label{fsi}
	\begin{aligned}
		\overline{w}^{\eps}_{tt} -\Delta \overline{w}^{\eps} &= \overline{g}^{\eps}\text{ on }[0,T]\times \Omega_W,\quad \overline{w}^{\eps}(0)=\overline{w}^{\eps}(T),\quad \overline{w}^{\eps}_t(0)=\overline{w}^{\eps}_t(T),
		\\
		{ \overline{w}^{\eps}}&= H^{\eps}-{ e^{\eps}-\beta^{\eps} \text{ and }} \partial_{\mathbf n} { \overline{w}^{\eps}}= 0 \text{ on }[0,T]\times \widetilde{\Gamma},
		\\
		{ \overline{w}^{\eps}} &=0\text{ on }\widetilde{\Gamma}_W,
	\end{aligned}
\end{align}
where we have denoted 
$$\overline{g}^{\eps}\equiv\widetilde{g}+\partial_{tt}(e^{\eps}+\beta^{\eps})-2\eps\partial_t \widetilde{w}^{\eps}-\eps^2\widetilde{w}^{\eps}.$$ 
{  Corollary \ref{cor:bostan} implies that $u^{\eps}$ and $w^{\eps}$ are smooth-in-time periodic functions, which in turn implies that $\overline{g}^{\eps}$ is a smooth-in-time periodic function.} In order to further simplify the notation, we denote $$E^{\eps}=e^{\eps}+\beta^{\eps}$$ and find that $E^{\eps}$, by construction { \eqref{estimatebeta}}, satisfies the estimate { \eqref{harmonicestimate}} from Lemma \ref{ExtnesionLemma}.

In particular $\overline{w}^{\eps}$ satisfy the following mixed boundary value problem for every fixed $t$:
$$
-\Delta \overline{w}^{\eps}=-\overline{w}^{\eps}_{tt}+\overline{g}^{\eps}\;{\rm in}\;\Omega_W,
\quad \overline{w}^{\eps}=0\;{\rm on}\;\tilde{\Gamma}_W,
\quad \partial_n\overline{w}^{\eps}=0\;{\rm on}\;\tilde{\Gamma}.
$$
Since $\Omega_W$ satisfy the elliptic regularity property we have $\overline{w}^{\eps}\in C([0,T];H^{\frac{3}{2}^+}(\Omega))$. Hence $\overline{w}^{\eps}$ is a over-determined weak energy solution satisfying all assumptions of Corollary \ref{cor:key}. 

\vskip.2cm
\noindent {\bf Step 2: Estimates with Generalized Optics Condition}.
Corollary \ref{cor:key} immediately implies the following estimate { for $k\in\N_0$}:
\begin{align*}
\|\partial^k_t\nabla\bar{w}^{\eps}\|_{L^{2}(0,T;\mathbf L^2(\Omega_W))}
+\|\partial^{k+1}_t\bar{w}^{\eps}\|_{L^{2}(0,T;L^2(\Omega_W))}
\lesssim
\|\partial^k_t\bar{g}^{\eps}\|_{L^2(0,T;L^2(\Omega_W))}+\|\partial^k_t (H^{\eps}-{ E^{\eps}})\|_{H^1(0,T;L^2(\widetilde{\Gamma}))}
\end{align*}
By the triangle inequality inequality, Lemma \ref{ExtnesionLemma} and definitions of $\bar{g}^{\eps}$ and $H^{\eps}$ above, we have:
\begin{align*}
\|\partial^k_t\nabla \widetilde{w}^{\eps}\|_{L^{2}(0,T;\mathbf L^2(\Omega_W))}
+&\|\partial^{k+1}_t\widetilde{w}^{\eps}\|_{L^{2}(0,T;L^2(\Omega_W))} \\ 
\lesssim  &~
\|\partial^k_t\nabla { E}^{\eps}\|_{L^2(0,T;\mathbf L^2(\Omega_W))}+\|\partial^{k+1}_t E^{\eps}\|_{L^2(0,T;L^2(\Omega_W))}\\ 
& + \|\partial^k_t\widetilde{g}+\partial^{k+2}_t E^{\eps}-2\eps\partial^{k+1}_t \widetilde{w}^{\eps}-\eps^2\partial^k_t\widetilde{w}^{\eps}\|_{L^2(0,T;L^2(\Omega_W))} \\ & +\|{ \partial^{k+1}_t\widetilde{u}^{\eps}}\|_{L^2(0,T;L^2(\Gamma))}{ +\|\partial_t^{k+1} E^{\eps}\|_{L^2(0,T;L^2(\Gamma))}}.
\\
\lesssim & ~
\|\partial^{k+3}_t \widetilde{u}^{\eps}\|_{L^2(0,T;(H_{\Gamma_H}^1(\Omega_H))')}+\|\partial^{k+2}_t\nabla \widetilde{u}^{\eps}\|_{L^2(0,T;\mathbf L^2(\Omega_H))} \\
&+ \|\partial^{k+2}_t \widetilde{f}\|_{L^2(0,T;H_{\Gamma_H}^1(\Omega_H)')}+\eps\|\partial^{k+1}_t \widetilde{w}^{\eps}\|_{L^2(0,T;L^2(\Omega_W))}\\ 
&+\|\partial^k_t \widetilde{g}\|_{L^2(0,T;L^2(\Omega_W))}
+\eps^2\|\partial^k_t \widetilde{w}^{\eps}\|_{L^2(0,T;L^2(\Omega_W))}
\end{align*}
Notice that, for $\eps$ sufficiently small, the last terms on the right hand side (RHS) can be absorbed in the LHS.

To close the estimate, we combine the obtained estimate with a higher order energy estimate. {  Since, by Corollary \ref{cor:bostan}, $\widetilde{u}^{\eps}$ and $\widetilde{w}^{\eps}$ are smooth-in-time periodic functions}, we can take \begin{equation} \psi=\sum_{k=0}^l (-1)^k\partial^{2k}_t\widetilde{u}^{\eps},~~~~\phi=\sum_{k=0}^l (-1)^k\partial^{2k+1}_t\widetilde{w}^{\eps},\end{equation} in \eqref{WeakReg} (note that the mean-free part of the solution satisfies the same type of problem as the original problem) and obtain
\begin{align}\nonumber
\sum_{k=0}^l\|\partial_t^k\nabla \widetilde{u}^{\eps}\|^2_{L^2(0,T;\mathbf L^2(\Omega_H))}
+&\eps\sum_{k=0}^l\|\partial_t^k \widetilde{u}^{\eps}\|^2_{L^2(0,T;L^2(\Omega_H))}
+2\eps\sum_{k=0}^l\|\partial^{k+1}_t\widetilde{w}^{\eps}\|^2_{L^2(0,T;L^2(\Omega_W))}
\\
\label{HOE}
=& \int_0^T\int_{\Omega_H}\widetilde{f}\sum_{k=0}^l (-1)^k\partial^{2k}_t\widetilde{u}^{\eps}
+\int_0^T\int_{\Omega_W}\widetilde{g} \sum_{k=0}^l (-1)^k\partial^{2k+1}_t\widetilde{w}^{\eps} \\
\equiv& ~ RHS
\end{align}
To complete the estimates, we integrate by parts with respect to the time variable on the right-hand side and take $l = 3$. In addition, we apply the previously derived hyperbolic estimate for $\|\partial_t\widetilde{w}^{\eps}\|_{L^2(0,T;L^2(\Omega_W))}$.

\begin{align*}
RHS= & ~
\int_0^T\int_{\Omega_H}\left (\widetilde{f}\widetilde{u}^{\eps}-\widetilde{f}\partial^2_t\widetilde{u}^{\eps}-\partial_t\widetilde{f}\partial^3_t \widetilde{u}^{\eps}
+\partial^3_t \widetilde{u}^{\eps}\partial^3_t \widetilde{f}\right )
+\int_0^T\int_{\Omega_W}\partial_t \widetilde{w}^{\eps}\sum_{k=0}^3(-1)^k\partial^{2k}_t\widetilde{g}
\\
\leq & ~ \|f\|_{H^3(0,T;L^2(\Omega_H))}\sum_{k=0}^3\|\partial^k_t \widetilde{u}^{\eps}\|_{L^2(0,T;L^2(\Omega_H))}
+\|\widetilde{g}\|_{H^6(0,T;L^2(\Omega_W))}\|\partial_t \widetilde{w}^{\eps}\|_{L^2(0,T;L^2(\Omega_W))}
\\
\leq& ~
\frac{C}{\delta}\left (
\|\widetilde{f}\|^2_{H^3(0,T;L^2(\Omega_H))}+\|\widetilde{g}\|^2_{H^6(0,T;L^2(\Omega_W))}
\right )
+\delta \sum_{k=0}^3\|\partial_t^k\nabla \widetilde{u}^{\eps}\|^2_{L^2(0,T;\mathbf L^2(\Omega_H))}.
\end{align*}
Now the $\delta$ term can be absorbed in the left-hand side and thus estimate \eqref{UnifromEstimatesStrong} is proved.

\vskip.2cm \noindent {\bf Step 3: Estimates with Graph Optics Condition}.
We apply Corollary \ref{cor:key} to $\overline{w}$ once again, but this time we obtain a weaker estimate since only the  Graph Optics Condition is satisfied. The remainder of the estimate is analogous to Step 2, and for $k\in\N_0$ we obtain.
\begin{align*}
	\|\partial^k_t\widetilde{w}^{\eps} & \|_{L^{2}(0,T;L^2(\Omega_W))}
	+  \|\partial^k_t\partial_{\flow}\widetilde{w}^{\eps}\|_{L^{2}(0,T;L^2(\Omega_W))}
	\\
	\lesssim & ~
	\|\partial^{k+3}_t \widetilde{u}^{\eps}\|_{L^2(0,T;H_{\Gamma_H}^1(\Omega_H)')}+\|\partial^{k+2}_t\nabla \widetilde{u}^{\eps}\|_{L^2(0,T;\mathbf L^2(\Omega_H))}
	+\|\partial^{k+2}_t \widetilde{f}\|_{L^2(0,T;H_{\Gamma_H}^1(\Omega_H)')} \\
	+& \|\partial^k_t \widetilde{g}\|_{L^2(0,T;L^2(\Omega_W))}
	+\eps\|\partial^{k+1}_t \widetilde{w}^{\eps}\|_{L^2(0,T;L^2(\Omega_W))}+\eps^2\|\partial^k_t \widetilde{w}^{\eps}\|_{L^2(0,T;L^2(\Omega_W))}
\end{align*}
Note that for small $\eps$, the last term can be absorbed into LHS. The penultimate term, on the other hand, cannot be absorbed immediately and will be handled on the LHS of \eqref{HOE}. We will now follow the same strategy as in Step 2, but since we need to estimate the time derivative of $\widetilde{w}^{\eps}$ using the hyperbolic estimate, we take $k = 1$ in the preceding estimate and, consequently, $l = 4$ in \eqref{HOE}, to obtain the desired result.
\begin{align*}
RHS=& ~
\int_0^T\int_{\Omega_H}\left (\tilde{f}\tilde{u}^{\eps}-\tilde{f}\partial^2_t\tilde{u}^{\eps}+\tilde{f}\partial^4_t \tilde{u}^{\eps}
-\partial^4_t \tilde{u}^{\eps}(\partial^2_t \tilde{f}-\partial^4_t \tilde{f})\right )
+\int_0^T\int_{\Omega_W}\partial_t \tilde{w}^{\eps}\sum_{k=0}^4(-1)^k\partial^{2k}_t\tilde{g}
\\
\leq&~ \|\tilde{f}\|_{H^4(0,T;L^2(\Omega_H))}\sum_{k=0}^4\|\partial^k_t \tilde{u}^{\eps}\|_{L^2(0,T;L^2(\Omega_H))} \\ 
&+\|\tilde{g}\|_{H^8(0,T;L^2(\Omega_W))}\Big (\|\partial_t \tilde{w}^{\eps}\|_{L^2(0,T;L^2(\Omega_W))}+C\eps\|\partial^2_t \tilde{w}^{\eps}\|_{L^2(0,T;L^2(\Omega_W))}\Big )
\\
\leq & ~
\frac{C}{\delta}\left (
\|\tilde{f}\|^2_{H^4(0,T;L^2(\Omega_H))}+\|\tilde{g}\|^2_{H^8(0,T;L^2(\Omega_W))}
\right )
 \\ & +\delta \Big (\sum_{k=0}^4\|\partial_t^k\nabla \tilde{u}^{\eps}\|^2_{L^2(0,T;\mathbf L^2(\Omega_H))}
+C\eps\|\partial^2_t \tilde{w}^{\eps}\|_{L^2(0,T;L^2(\Omega_W))}\Big ).
\end{align*}
By taking $\delta$ small enough we complete the proof of estimate \eqref{UnifromEstimates}
\end{proof}

\subsection{Conclusion of the Argument: Proof of Theorems \ref{th:main1} and \ref{th:main2}}
We may approximate the functions $(f, g)$ as given in the $\epsilon$-translated equations \eqref{eq:solid*}--\eqref{eq:periodic*}  with smooth, $T$-periodic functions, thereby justifying the estimates obtained in Section \ref{unifest}. Additionally, the solution can then be decomposed into its mean value and mean-free components:
$$
(u^{\eps},w^{\eps})=(\mean{u^{\eps}},\mean{w^{\eps}})+(\tilde{u}^{\eps},\tilde{w}^{\eps}).
$$
The mean values can be uniformly estimated in $\epsilon$ using Proposition $\ref{MeanValueEst}$, while the mean-free parts can be uniformly estimated via Theorem $\ref{UniformEstamtesTm}$. Since the problem is linear, these estimates are sufficient to pass to the limit in $\epsilon$ to obtain a weak solution in the sense of Definition \ref{DefWeakApp} with $\epsilon=0$; in turn, we have established the existence of a weak solution (via the identified limits) in the sense of Definition \ref{weaksol}. Uniqueness of said weak solution is established in Theorem $\ref{th:uniqueness}$, thus completing the proofs of the main results under the assumptions that the wave domain has elliptic regularity property from Definition \ref{EllipticRegCond}.

It remains to prove the claim under the assumption that $\Omega_W$ satisfies the E-property (see Definition \ref{EProperty}). Since the steps of this proof are relatively standard, we present only a sketch of the argument.

Let ${\Omega_W^n}$ be a sequence of domains with the elliptic regularity property that approximate $\Omega_W$ in the sense of Definition \ref{EProperty}. Let $(u_n, w_n)$ denote the corresponding weak solutions constructed in the first part of the proof.
On the boundary $\partial \Omega_W^n$, the zero Dirichlet boundary condition is imposed for both unknowns. Consequently, these solutions can be extended by zero to the entire space $\R^d$. At this stage, it suffices to estimate the oscillatory part of the solution, as elliptic regularity is not utilized in estimating the mean-free component. From the first part of the proof, we conclude that the oscillatory parts of the solutions, $(\chi_{\Omega_H}\tilde{u}_n, \chi_{\Omega_W^n}\tilde{w}_n)$, satisfy the uniform estimates provided in Theorem \ref{UniformEstamtesTm}.

Using these uniform estimates, we can extract weakly convergent subsequences in the topologies determined by the estimates. It follows immediately that the limit is of the form $(\chi_{\Omega_H}u, \chi_{\Omega_W}w)$. Furthermore, since the problem is linear, passing to the limit in \eqref{weakdef} is straightforward, and we find that $(u, w)$ is a weak solution in the sense of Definition \ref{weaksol}. Finally, due to the weak lower semi-continuity of norms, all estimates are preserved in the limit, thereby completing the proof.

\qed

{ \textbf{Acknowledgment} We would like to thank the anonymous referee for carefully reviewing our paper and for the insightful comments and suggestions that have significantly improved the quality of this work.}

\section*{Appendix A: Further Examples of Spatial Domains}
In this appendix we include some illustrative domains which further demonstrate the breadth of wave domains for which our results hold.

\subsection*{Example 1: Spirals}\label{appendix:Spiral-Construction}
In this part we construct spiral domains, based on a rotational flow. Consider
\begin{align}
	\flow(x,y) = \left(\begin{matrix}
		-y\\
		x
	\end{matrix}\right)
	+
	\alpha\left(\begin{matrix}
		x\\
		y
	\end{matrix}\right)
\end{align}
with parameter $\alpha\geq 0$. Then
\begin{align*}
	\grad\flow(x,y) &= \left(
	\begin{matrix}
		\alpha & -1 \\
		1 & \alpha
	\end{matrix}
	\right),~~
	\diverg\flow(x,y) = 2\alpha, 
\end{align*} 
and thus 
$$
	\xi^T \grad\flow(x,y) \xi = \alpha \left(\xi_1^2 + \xi_2^2 \right).
$$
This provides the contractive domain property (with the full gradient estimate). Thus, if we construct a boundary part $\Gamma_W$ of $\Omega_W$ in such a way that $\flow\cdot\nrml \leq 0$ is satisfied on $\Gamma_W$, we will have produced a domain satisfying the Generalized Optics Condition \ref{def:GOC}. 
\begin{figure}[!htb]
	\centering
	\begin{minipage}[b]{0.35\textwidth}
		\includegraphics[width=\textwidth]{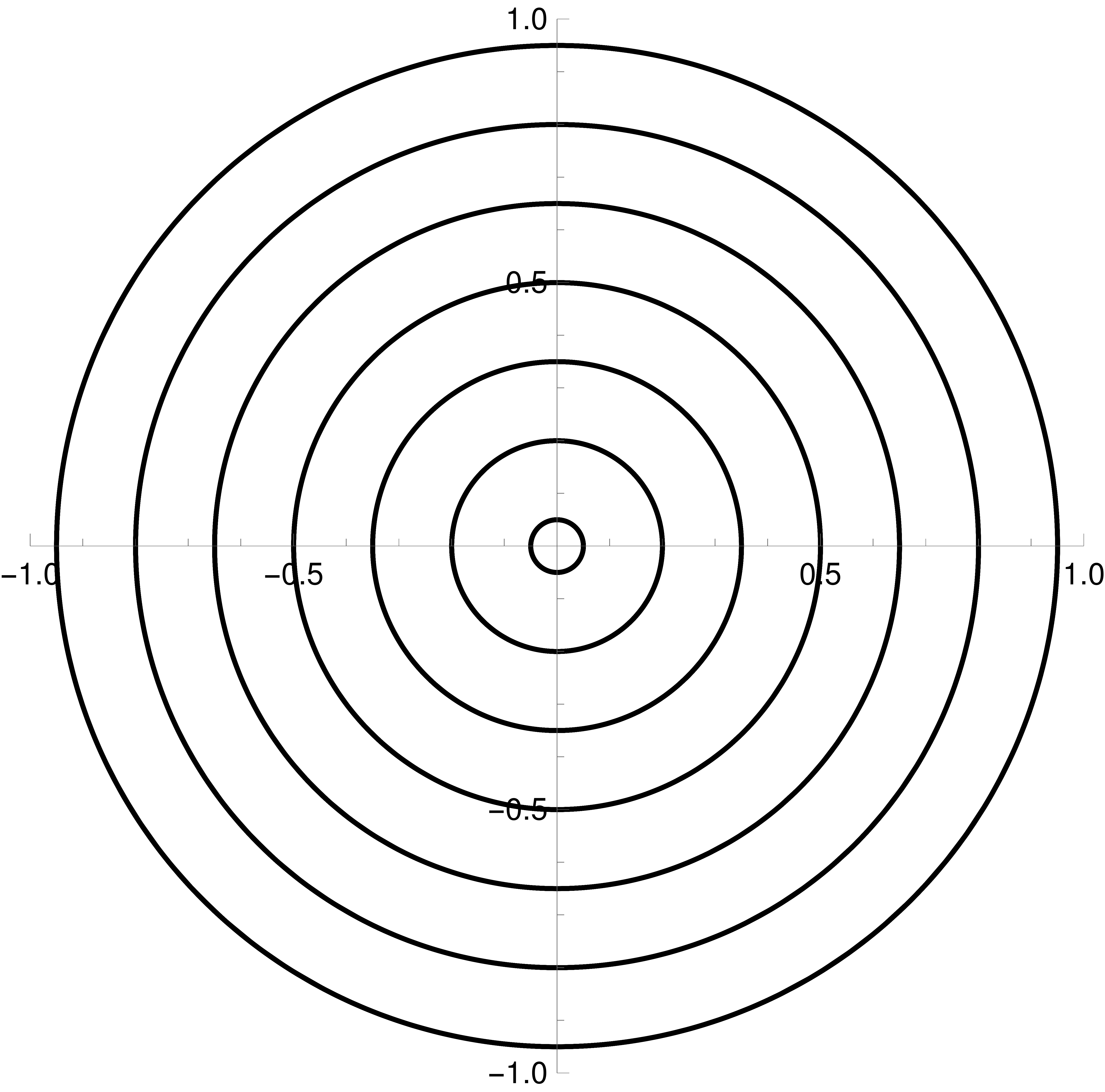}
		\caption{Integral Lines ($\alpha=0$).}
	\end{minipage}
	\hspace{1cm}
	\begin{minipage}[b]{0.35\textwidth}
		\includegraphics[width=\textwidth]{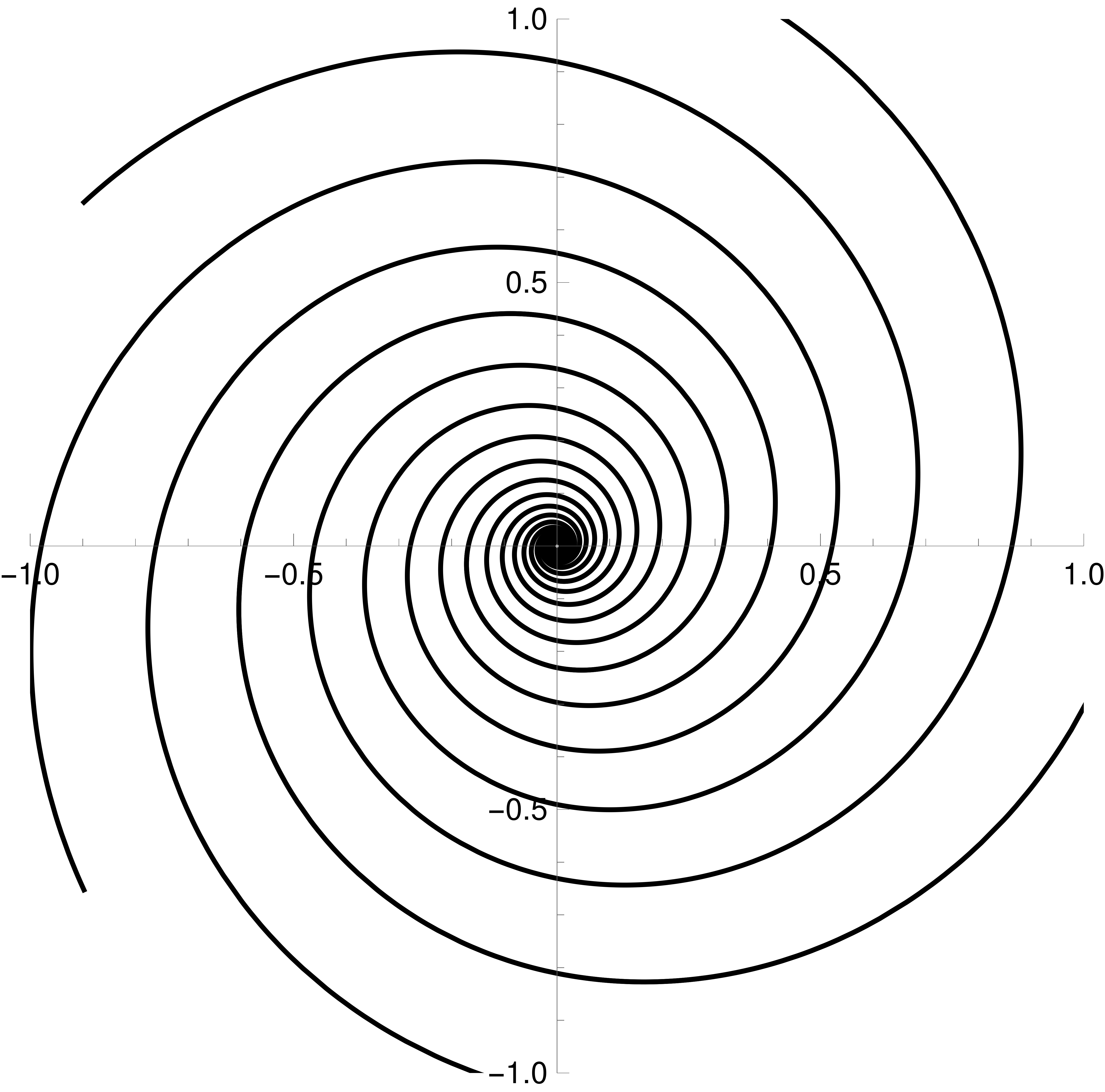}
		\caption{Integral lines ($\alpha={1}/{5}$).}
		\label{fig-lines-spiral}
	\end{minipage}
\end{figure}
A trivial way  to achieve this is to take $\Gamma_W$ as parts of some integral lines of $\flow$, hence, satisfying
\begin{align*}
	x'(t) &= -y(t) + \alpha x(t) \\
	y'(t) &= x(t) + \alpha y(t)
\end{align*}
with an arbitrary initial condition $(x(0), y(0))$. The solution formula is as follows
\begin{align*}
	x(t) &= x(0) \exx^{\alpha t}\cos{t} - y(0) \exx^{\alpha t}\sin{t}, \\
	y(t) &= y(0) \exx^{\alpha t}\cos{t} + x(0) \exx^{\alpha t}\sin{t},
\end{align*}
where the parameter $t\in\R$---see Figures 8 and 9. Connecting the integral lines together through the $\Gamma$ portion  of the boundary, we can create the following two domains, as  shown in Figures 10 and 11.

\begin{figure}[!htb]
	\centering
	\begin{minipage}[b]{0.4\textwidth}
		\includegraphics[width=\textwidth]{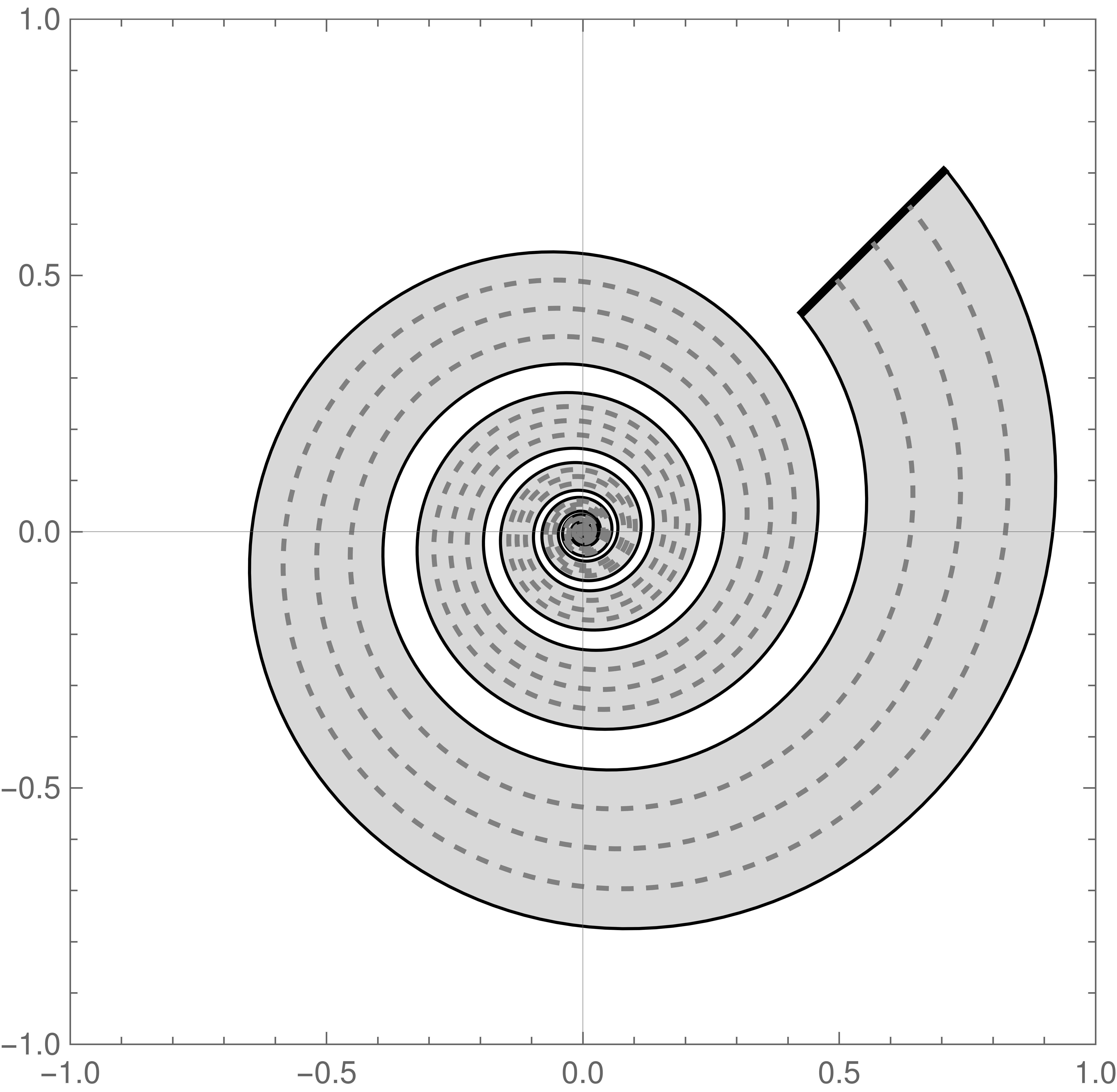}
		\caption{Spiral Domain ($\alpha={1}/{9}$).}
	\end{minipage}
	\hspace{1cm}
	\begin{minipage}[b]{0.4\textwidth}
		\includegraphics[width=\textwidth]{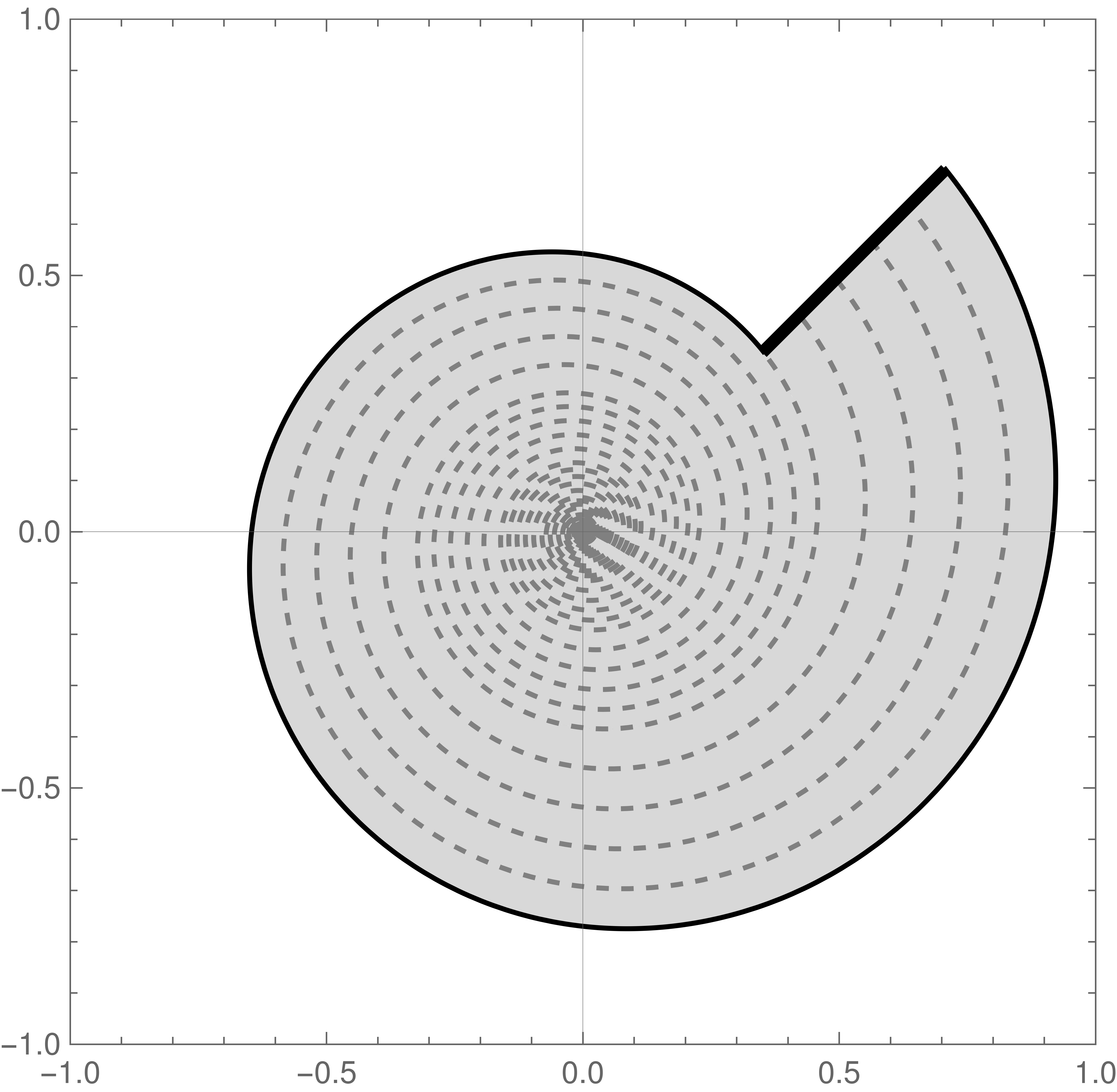}
		\caption{Shell Domain ($\alpha={1}/{9}$).}
	\end{minipage}
\end{figure}

\subsection{Example 2: Arc (cf. Figure~6)}\label{appendix:Arc-Construction}
In the previous setting, with $\alpha=0$, we cannot directly invoke the rotation vector field $\flow(x, y) = (-y, x)^T$, as all  terms providing useful information vanish. Hence, we look for a scalar multiplier $\Phi: \R^2 \to \R$, such that the renormalized field satisfies $\diverg(\Phi\flow) = 1$. In polar coordinates it is easy to see that $\Phi$ only depends on $\theta$. Hence in the upper half plane ($y > 0$) we can take
\begin{equation}
	\Phi(x, y) = \ArcCos\left(\frac{x}{\sqrt{x^2+y^2}}\right) > 0.
\end{equation}
Then 
\begin{align*}
	\xi^T \grad\big[\Phi(x,y)\flow(x,y)\big] \xi &= 
	\left(
	\frac{y}{\sqrt{x^2+y^2}}\; \xi_1 - \frac{x}{\sqrt{x^2+y^2}} \; \xi_2
	\right)^2
	=\frac{\abs{\xi \cdot \flow}^2}{x^2+y^2},
\end{align*} 
and because $\sqrt{x^2+y^2}$ is uniformly bounded away from zero, this yields the desired estimate. The arc domain is again created in such a way that the boundary conditions for the flow satisfy the sign condition on $\Gamma_W$ (see Figure 6). Thus the domain in Figure 6 satisfies the {\em Graph Optics Condition} \ref{def:GGC} but not the {\em Generalized Optics Condition} \ref{def:GOC}.

\section*{Appendix B: Very Weak Solutions}
In this section we present a definition of very weak solutions and provide the proof of Theorem \ref{th:main1-general}. The proof will directly follow from Theorem \ref{th:main1} and the simple observation that a time-derivative of a time-periodic solution is again a time-periodic solution to the time-differentiated problem. Thus we will prove the theorem by taking primitives of the right-hand side and using the invariances introduced in Section \ref{Sec:invariances}.

{  We start by introducing dual spaces of spaces of time-periodic functions. Let $X$ be a Banach space of functions. For $k\in \mathbb{N}$ we
define $f \in H^{-k}_{\sharp}(0,T;X')\equiv\left (H^k_{\sharp}(0,T;X)\right )'$, i.e., 
\[
H^{-k}_{\sharp}(0,T;X')\equiv\left \{f: H^{k}_{\sharp}(0,T;X)\to \mathbb{R}:f \text{ is bounded and linear}\right \}.
\]
Note that this implies that one can define the mean value for such functionals, 
$$
\skp{\mean{f}}{\phi}_{X'\times X}\equiv\frac{1}{T}\skp{f}{1\otimes \phi}_{H^{-k}_\sharp(0,T;X')\times H^k_{\sharp}(0,T;X)},\quad \phi\in X.
$$
Furthermore, we can prove the following representation result.

\begin{lemma}
\label{lem:dual}
Let $k\in\N$ and $f\in H^{-k}_{\sharp}(0,T;X')$. Then for any $m\in\N$ there exists $F \in H^m_{\sharp}(0,T;X')$, such that $\partial_t^{k+m}F=f-\mean{f}$ in the sense of distributions.
\end{lemma}
\begin{proof}
We introduce the notation $\skp{f}{g}=\int_0^T\skp{f(t)}{g(t)}_{X'\times X}$. By a standard argument used for the representation of functionals in Bochner spaces, and by the Riesz representation theorem, there exist $f_0,...,f_n\in L^2(0,T;X')$ such that
\[
\skp{f}{\phi}=\sum_{j=0}^k \skp{f_j}{\partial^{(j)}_t\phi}.
\]
Since $f_i$, $i\in\{1,...,k\}$ are acting on time derivatives of periodic functions, which are mean-free, we have that $\mean{f_i}=0$. Therefore $\mean{f}=\mean{f_0}$. The goal is now to prove that $f_i$ can be written as the derivative of time-periodic functions.

Since the primitive of a function is time-periodic if and only if it has zero mean, we need to subtract the mean value from $f_0$. Therefore, we define:
$$
F_0^1(t)\equiv\int_0^t (f_0-\mean{f_0})=\int_0^t f_0(s)ds-t\mean{f}\in H^1_{\sharp}(0,T;X'),
$$ 
and $\frac{d}{dt}F_0^1=f_0-\mean{f}$. However, this procedure can be iterated. We define $F_0^2\equiv\int_0^t (F_0^1-\mean{F_0^1})$, and it holds that
\[
\skp{F_0^2}{\phi''}  = \skp{f_0}{\phi} \text{ and }F_0^2\in H^2_{\sharp}(0,T;X').
\]
This can be  further iterated up to $F_0^{m+k}$. Furthermore, the same procedure can be applied to all functions $f_j$, for $j \in \{1, \dots, k-1\}$, defining $F_j^{m+k-j}$. Thus, $F \equiv \sum_{j=0}^{k-1} F_j^{m+k-j}+f_k \in H^m_{\sharp}(0,T;X')$ satisfies the desired property.
 \end{proof}
 }

\begin{definition}\label{VeryWeakDef}
	Let $k\in\N$. Assume that the forcing functions satisfy $f \in H^{-k}_{\sharp}(0,T;H^1_{\Gamma_H}(\Omega_H)')$ and $g \in H^{-k}_{\sharp}(0,T;L^2(\Omega_W))$. We say that $$(u,w)\in H^{-k}_{\sharp}(0,T;H^1_{\Gamma_H}(\Omega_H))\times \left [H^{-k}_{\sharp}(0,T;H^1_{\Gamma_W}(\Omega_W))\cap H^{-k+1}_{\sharp}(0,T;L^2(\Omega_W))\right]$$ is a very weak solution to \eqref{sys2}--\eqref{sys2e}, if there exist functions $F \in L^2(0,T;H^1_{\Gamma_H}(\Omega_H)')$, $G \in L^2(0,T;L^2(\Omega_W))$ and $(U,W)\in L^2(0,T;H^1_{\Gamma_H}(\Omega_H))\times \left [L^2(0,T;H^1_{\Gamma_W}(\Omega_W))\cap H^1(0,T;L^2(\Omega_W))\right]$, such that 
	$$
	\partial_t^kF={  f-\mean{f}},\;\partial_t^kG={  g-\mean{g}},\quad
	\partial_t^k(U,W)=(u-u_s,w-w_s),
	$$
	where $(U,W)$ is a finite-energy weak solution in the sense of Definition \ref{weaksol} with right-hand sides $(F,G)$ {  and $\left(u_s,w_s\right)\in H^1_0(\Omega_H)\times H^1_{\Gamma_W}(\Omega_W)$ are defined by solving stationary problems with the right-hand sides $\mean{f}$ and $\mean{g}$ as in \eqref{MeanValueHeatEq} and \eqref{meanw}.}
\end{definition}
{  The above definition implies the following canonical weak formulation, for a general distribution, using the characterization in Lemma~\ref{lem:dual}
\begin{align*}
&-\int_0^T\int_{\Omega_H}[U \partial_t^{(k+1)}\phi +\nabla U \cdot \nabla \partial_t^{(k)}\phi] + T\int_{\Omega_H}\nabla u_s\cdot\nabla\mean{\phi}
\\
&+ \int_0^T\int_{\Omega_W}[-\partial_t W\partial_t^{(k+1)}\psi + \nabla W \cdot \nabla \partial_t^{(k)}\psi] + T\int_{\Omega_H}\nabla w_s\cdot\nabla\mean{\psi}
= \int_0^T\langle f,\phi \rangle +\int_0^T\langle g ,\psi\rangle,
\end{align*}
for all $(\phi,\psi)\in H^{k+1}_{\sharp}(0,T;H^1(\Omega_W))\times H^{k+1}_{\sharp}(0,T;H^1(\Omega_H))$, with $\phi=\psi$ on $\Gamma$, $\phi=0$ on $\Gamma_H$ and $\psi=0$ on $\Gamma_W$, a.e.\ in time.
}
\begin{proof}[Proof of Theorem \ref{th:main1-general}] 
Let $f \in H^{-k+3}(0,T;H^1_{\Gamma_H}(\Omega_H)')$ and $g \in H^{-k+6}(0,T;L^2(\Omega_W))$. {  Since $\mean{f}\in H^1_{\Gamma_H}(\Omega_H)'$ and $\mean{g}\in L^2(\Omega_W)$ are well defined, we can define $u_s$ and $w_s$ as the decoupled steady part, see Section \ref{Sec:invariances}. We use Lemma \ref{lem:dual} with $m=3$ and $m=6$ to find functions $F \in H^3_{\sharp}(0,T;H^1_{\Gamma_H}(\Omega_H)')$ and $G \in H^6_{\sharp}(0,T;L^2(\Omega_W))$, such that $\partial_t^kF=f-\mean{f}$ and $\partial_t^kG=g-\mean{g}$. Functions $F$ and $G$ satisfy the assumptions of Theorem \ref{th:main1}, and therefore there exists a weak solution $(\tilde{U},\tilde{W})$ in the sense of Definition \ref{weaksol} with right-hand side $(F,G)$.} 
Hence, the very weak solution in the sense of Definition \ref{VeryWeakDef} is given by 
$$
(u,w)=\partial_t^{k}(U,W)+(u_s,w_s).
$$
\end{proof}

Theorem \ref{th:main2-general} is proved completely analogously by using Theorem \ref{th:main2}.


\footnotesize
\begin{thebibliography}{99}

\bibitem{aubin} Aubin, J.P., 2007. Approximation of elliptic boundary-value problems. Courier Corporation.



\bibitem{trig1} Avalos, G., Lasiecka, I. and Triggiani, R., 2016. Heat-wave interaction in 2-3 dimensions: optimal rational decay rate. {\em Journal of Mathematical Analysis and Applications}, 437(2), pp.782--815.

\bibitem{BAL} Ball, J.M., 1977. Strongly continuous semigroups, weak solutions, and the variation of constants formula. {\em Proceedings of the American Mathematical Society}, 63(2), pp.370--373.

\bibitem{isaac} Benson, I. and Webster, J.T., 2024. Resonance and Periodic Solutions for Harmonic Oscillators with General Forcing. arXiv preprint arXiv:2407.17144.

\bibitem{brezis} Bahri, A. and Br\'ezis, H., 1980. Periodic solutions of a nonlinear wave equation. {\em Proceedings of the Royal Society of Edinburgh Section A: Mathematics}, 85(3--4), pp.313--320.

\bibitem{gaz1} Bonheure, D., Gazzola, F. and Dos Santos, E.M., 2019. Periodic solutions and torsional instability in a nonlinear nonlocal plate equation. {\em SIAM Journal on Mathematical Analysis}, 51(4), pp.3052--3091.

\bibitem{gaz2} Bonheure, D., Galdi, G.P. and Gazzola, F., 2024. Flow-induced Oscillations via Hopf Bifurcation in a Fluid-Solid Interaction Problem. arXiv preprint {\em arXiv:2406.04198}.


\bibitem{bostan}  Bostan, M., 2002. Periodic solutions for evolution equations. {\em Electronic Journal of Differential Equations}, pp.3--41.

\bibitem{breziscoron} Br\'ezis, H., Coron, J.M. and Nirenberg, L., 1980. Free vibrations for a nonlinear wave equation and a theorem of P. Rabinowitz. {\em Communications on Pure and Applied Mathematics}, 33(5), pp.667--684.

\bibitem{casanova} Casanova, J.J., 2019. Existence of time-periodic strong solutions to a fluid-structure system. {\em Discrete and Continuous Dynamical Systems}, 39(6), pp.3291--3313.

\bibitem{celik} Celik, A., 2020. Non-resonant solutions in hyperbolic-parabolic systems with periodic forcing. {\em Logos Verlag}, Berlin GmbH.

\bibitem{cesari} Cesari, L., 1965. Existence in the large of periodic solutions of hyperbolic partial differential equations. {\em Archive for Rational Mechanics and Analysis}, 20(3), pp.170--190.

\bibitem{coron} Coron, J.M., 1983. Periodic solutions of a nonlinear wave equation without assumption of monotonicity. {\em Mathematische Annalen}, 262(2), pp.273--285.



\bibitem{dauge} Dauge, M., 2006. Elliptic boundary value problems on corner domains: smoothness and asymptotics of solutions (Vol. 1341). Springer.

\bibitem{dutch} Duyckaerts, T., 2007. Optimal decay rates of the energy of a hyperbolic-parabolic system coupled by an interface. {\em Asymptotic Analysis}, 51(1), pp.17--45.

\bibitem{eller} Eller, M., 2023. {\em Unique continuation and Carleman estimates for Partial Differential Equations}. 

\bibitem{eller2} Eller, M. and Toundykov, D., 2012. A global Holmgren theorem for multidimensional hyperbolic partial differential equations. {\em Applicable Analysis}, 91(1), pp.69-90.

\bibitem{evans} Evans, L.C., 2022. Partial differential equations (Vol. 19). {\em American Mathematical Society.}

\bibitem{fw1}
Feireisl, E., 1988. Time periodic solutions to a semilinear beam equation. {\em Nonlinear Analysis: Theory, Methods \& Applications}, 12(3), pp.279--290.

\bibitem{fw2}
Feireisl, E. and Pol\'a\v{c}ik, P., 2000. Structure of periodic solutions and asymptotic behavior for time-periodic reaction-diffusion equations on $\mathbf R$.

\bibitem{fw3}
Feireisl, E., 1988. On the existence of periodic solutions of a semilinear wave equation with a superlinear forcing term. {\em Czechoslovak Mathematical Journal}, 38(1), pp.78--87.

\bibitem{fn1}
Feireisl, E., Mucha, P.B., Novotn\'y, A. and Pokorn\'y, M., 2012. Time-periodic solutions to the full Navier--Stokes--Fourier system. {\em Archive for Rational Mechanics and Analysis}, 204, pp.745--786.

\bibitem{fn2}
Feireisl, E., Ne\v{c}asov\'a, \v{S}., Petzeltov\'a, H. and Stra\v{s}kraba, I., 1999. On the Motion of a Viscous Compressible Fluid Driven by a Time‐Periodic External Force. {\em Archive for Rational Mechanics and Analysis}, 149, pp.69--96.


\bibitem{playing} Fonda, A., 2023. Playing Around with Resonance, {\bf Springer}.


\bibitem{galdi} Galdi, G.P., Mohebbi, M., Zakerzadeh, R. and Zunino, P., 2014. Hyperbolic-Parabolic Coupling and the Occurrence of Resonance in Partially Dissipative Systems. {\em Fluid-structure Interaction and Biomedical Applications}, pp.197--256.

\bibitem{g1}
Galdi, G.P., 2013. Existence and uniqueness of time-periodic solutions to the Navier-Stokes equations in the whole plane. {\em Discrete Contin. Dyn. Syst. Ser. S}, 6(5), pp.1237--1257.

\bibitem{gs} Galdi, G.P. and Silvestre, A.L., 2007. The steady motion of a Navier–Stokes liquid around a rigid body. {\em Archive for Rational Mechanics and Analysis}, 184(3), pp.371--400.

\bibitem{g2}
Galdi, G.P. and Sohr, H., 2004. Existence and uniqueness of time-periodic physically reasonable Navier-Stokes flow past a body. {\em Archive for Rational Mechanics and Analysis}, 172, pp.363--406.

\bibitem{grisvard} Grisvard, P., 2011. Elliptic problems in nonsmooth domains. Society for Industrial and Applied Mathematics.

\bibitem{holmgren} H\"ormander, L., 1993. Remarks on Holmgren's uniqueness theorem. In: {\em Annales de l'institut Fourier} (Vol. 43, No. 5, pp. 1223--1251).

\bibitem{srd} Kreml, O., M\'acha, V., Ne\v{c}asov\'a, \v{S}. and Trifunovi\'c, S., 2023. On time-periodic solutions to an interaction problem between compressible viscous fluids and viscoelastic beams. arXiv preprint arXiv:2307.02687.

\bibitem{redbook} Lasiecka, I. and Triggiani, R., 2000. {\em Control theory for partial differential equations: Volume 1, Abstract parabolic systems: Continuous and approximation theories} (Vol. 1). Cambridge University Press.

\bibitem{lt1} Lasiecka, I. and Triggiani, R., 1992. Uniform stabilization of the wave equation with Dirichlet or Neumann feedback control without geometrical conditions. {\em Applied Mathematics and Optimization}, 25(2), pp.189--224.

\bibitem{lt2} Lasiecka, I. and Triggiani, R., 1989. Exact controllability of the wave equation with Neumann boundary control. {\em Applied Mathematics and Optimization}, 19(1), pp.243--290.

\bibitem{littman} Walter Littman, Remarks on global uniqueness theorems for partial differential equations, Differential geometric methods in the control of partial differential equations, {\em Contemp. Math.}, vol. 268, Amer. Math. Soc., Providence, RI, 2000, pp. 363?371. 

\bibitem{lunardi} Lunardi, A., 1988. Bounded solutions of linear periodic abstract parabolic equations. {\em Proceedings of the Royal Society of Edinburgh Section A: Mathematics}, 110(1--2), pp.135--159.

\bibitem{giusy} Mazzone, G. and Mohebbi, M., 2024. On periodic motions of a harmonic oscillator interacting with incompressible fluids. {\em Physica D: Nonlinear Phenomena}, p.134259.

\bibitem{sebastian} Mindril\v a, C. and Schwarzacher, S., 2022. Time-periodic weak solutions for an incompressible Newtonian fluid interacting with an elastic plate. {\em SIAM Journal on Mathematical Analysis}, 54(4), pp.4139--4162.

\bibitem{sebastian2} {Mindril\v a, C. and Schwarzacher, S., 2023. Time-periodic weak solutions for the interaction of an incompressible fluid with a linear Koiter type shell under dynamic pressure boundary conditions, arXiv preprint {\em arXiv:2303.13625}}

\bibitem{russell} Quinn, J.P. and Russell, D.L., 1977. Asymptotic stability and energy decay rates for solutions of hyperbolic equations with boundary damping. {\em Proceedings of the Royal Society of Edinburgh Section A: Mathematics}, 77(1-2), pp.97--127.

\bibitem{pata} Pata, V. and Webster, J.T., 2024. An Observation About Weak Solutions of Linear Differential Equations in Hilbert Spaces. {\em Applied Mathematics and Optimization}.

\bibitem{pazy} Pazy, A., 2012. Semigroups of linear operators and applications to partial differential equations (Vol. 44). {\em Springer Science \& Business Media}.

\bibitem{rab} Rabinowitz, P.H., 1969. Periodic solutions of nonlinear hyperbolic partial differential equations. {\em Communications on Pure and Applied Mathematics}, 22(1), pp.15--39.



\bibitem{temam}
\newblock Temam, R., 1988.
\newblock Infinite-Dimensional Dynamical Systems in Mechanics and Physics
\newblock {\em Springer-Verlag}, New York.

\bibitem{Vejvoda81}
Otto Vejvoda, Leopold Herrmann, Vladim\'ir Lovicar, Miroslav Sova,
Ivan Stra\v{s}kraba, and Milan \v{S}t\v{e}dr\'y.
\newblock {\em Partial differential equations: time-periodic solutions}.
\newblock Martinus Nijhoff Publishers, The Hague, 1981.


\bibitem{ZZ1} Zhang, X. and Zuazua, E., 2003. Control, observation and polynomial decay for a coupled heat-wave system. {\em Comptes Rendus Mathematique}, 336(10), pp.823--828.

\bibitem{ZZ2} Zhang, X. and Zuazua, E., 2007. Long-time behavior of a coupled heat-wave system arising in fluid-structure interaction. {\em Archive for Rational Mechanics and Analysis}, 184, pp.49--120.


\end{thebibliography}
\end{document}